[1994/12/01]
\documentclass[10pt, reqno]{amsart}
\usepackage{a4wide}
\usepackage[english, activeacute]{babel}
\usepackage{amsmath,amsthm,amsxtra,times}
\usepackage{epsfig}
\usepackage{amssymb}
\usepackage{latexsym}
\usepackage{amsfonts}
\usepackage{hyperref}
\title{On the soliton dynamics under slowly varying medium for generalized KdV equations}
\author{Claudio Mu\~noz C.}
\address{Universit\'e de Versailles Saint-Quentin-en-Yvelines \\ LMV-UMR 8100, 45 av. des Etats-Unis, 78035 Versailles cedex, France}
\email{Claudio.Munoz@math.uvsq.fr}
\thanks{This research was supported in part by a CONICYT-Chile and an \emph{Allocation de Recherche} grants}

\chardef\bslash=`\\ 

\newtheorem{thm}{Theorem}[section]
\newtheorem{cor}[thm]{Corollary}
\newtheorem{lem}[thm]{Lemma}
\newtheorem{prop}[thm]{Proposition}

\theoremstyle{definition}
\newtheorem{defn}{Definition}[section]

\theoremstyle{remark}
\newtheorem{rem}{Remark}[section]

\newtheorem{Cl}{Claim}

\numberwithin{equation}{section}

\newcommand{\R}{\mathbb{R}}
\newcommand{\N}{\mathbb{N}}

\newcommand{\la}{\lambda}

\newcommand{\al}{\alpha}
\newcommand{\ga}{\gamma}

\newcommand{\supp}{\operatorname{supp}}

\def\bm{\left( \begin{array}{cc}}
\def\endm{\end{array}\right)}

 \providecommand{\abs}[1]{\lvert#1 \rvert}
 \providecommand{\norm}[1]{\lVert#1 \rVert}

\newcommand{\ve}{\varepsilon}

\newcommand{\be}{\begin{equation}}
\newcommand{\ee}{\end{equation}}
\newcommand{\ba}{\begin{equation*}}
\newcommand{\ea}{\begin{equation*}}
\newcommand{\bea}{\begin{eqnarray}}
\newcommand{\eea}{\end{eqnarray}}
\newcommand{\bee}{\begin{eqnarray*}}
\newcommand{\eee}{\end{eqnarray*}}

\newcommand{\eval}[2][\right]{\relax
  \ifx#1\right\relax \left.\fi#2#1\rvert}

\let\abs=\envert

\let\norm=\enVert

\begin{document}
\begin{abstract}
We consider the problem of existence and global behavior of a soliton in a slowly varying medium, for a generalized Korteweg - de Vries equations (gKdV). We prove that slowly varying media induce on the soliton dynamics large dispersive effects at large time. Moreover, unlike gKdV equations, we prove that there is no pure-soliton solution in this regime.    
\end{abstract}
\maketitle \markboth{Soliton dynamics for perturbed gKdV equations} {Claudio Mu\~noz}
\renewcommand{\sectionmark}[1]{}
\section{Introduction and Main Results}

\smallskip

In this work we consider the following \emph{generalized Korteweg-de Vries equation} (gKdV) on the real line
\be\label{gKdV0}
u_t + (u_{xx} + f(x,u))_x =0,\quad \hbox{ in }  \R_t\times \R_x. 
\ee

Here $u=u(t,x)$  is a real-valued function, and $f: \R\times \R\to \R$ a nonlinear function. This equation represents a mathematical generalization of the \emph{Korteweg-de Vries} equation (KdV), namely the case $f(x,s) \equiv s^2$,
\be\label{KdV}
u_t + (u_{xx} +u^2)_x =0, \quad \hbox{ in }\R_t\times \R_x;
\ee
another physically important case is the cubic one, $f(x, s)\equiv s^3$. In this case, the equation (\ref{gKdV0}) is often refered as the (focusing) \emph{modified} KdV equation (mKdV). In general, mathematicians denote by \emph{generalized Korteweg-de Vries}  (gKdV) the following equation
\be\label{gKdV}
u_t + (u_{xx} +u^m)_x =0, \quad \hbox{ in }\R_t\times \R_x;\quad m\geq 2 \hbox{ integer}.
\ee
 
Concerning the KdV equation, it arises in Physics as a model of propagation of dispersive long waves, as was pointed out by J. S. Russel in 1834 \cite{Miura}. The exact formulation of the KdV equation comes from Korteweg and de Vries (1895) \cite{KdV}. This equation was re-discovered in a numerical work by N. Zabusky and M. Kruskal in 1965 \cite{KZ}.

After this work, a great amount of literature has emerged, physical, numerical and mathematical, for the study of this equation, see for example \cite{Bo3, Bo14, Shih34, Mi30, Miura}. 
This continuous, focused research on the KdV (and gKdV) equation can be in part explained by some striking algebraic properties. One of the first properties is the existence of localized, exponentially decaying, stable and smooth solutions called \emph{solitons}. Given two real numbers $ x_0$ and $c>0$, solitons are solutions of (\ref{gKdV}) of the form
\be\label{(3)}
u(t,x):= Q_c(x-x_0-ct ),  \quad Q_c(s):=c^{\frac 1{m-1}} Q(c^{1/2} s),
\ee  
and where $Q$ is a explicit Schwartz function satisfying the second order nonlinear differential equation
\be\label{soliton}
Q'' -Q + Q^m =0, \quad Q(x) = \Big[ \frac{m+1}{2\cosh^2(\frac {(m-1)}2 x)}\Big]^{\frac 1{m-1}}.
\ee
In particular, this solution represents a \emph{solitary wave} defined for all time moving to the right without \emph{any change} in shape, velocity, etc.   

In addition, equation (\ref{gKdV}) remains invariant under space and time \emph{translations}. From the No\"ether theorem, these symmetries are related to \emph{conserved quantities}, invariant under the gKdV flow, usually called \emph{mass} and \emph{energy}:
\be\label{M}
M[u](t):= \int_\R u^2(t,x)\,dx = \int_\R u_0^2(x)\, dx =M[u](0), \quad \hbox{(Mass),}
\ee
and
\bea
E[u](t) &:= & \frac 12 \int_\R u_x^2(t,x)\,dx - \frac 1{m+1}\int_\R  u^{m+1}(t,x)\,dx \label{E}  \\
& & \quad =  \frac 12 \int_\R (u_0)_x^2(x)\,dx - \frac 1{m+1}\int_\R u_0^{m+1}(x)\,dx =  E[u](0). \; \hbox{(Energy) }\nonumber
\eea


Let us now review some facts about the gKdV equation (\ref{gKdV}), with $m\geq 2$ integer. The Cauchy problem for equation \eqref{gKdV0} (namely, adding the initial condition $u(t=0)=u_0$) is \emph{locally well-posed} for $u_0\in H^1(\R)$ (see Kenig, Ponce and Vega \cite{KPV}). In the case $m<5$, any $H^1(\R)$ solution is global in time thanks to the conservation of mass and energy (\ref{M})-(\ref{E}), and the Galiardo-Nirenberg inequality
\be\label{GNe}
\int_\R u^{p+1} \leq K(p) \big(\int_\R u^2\big)^{\frac{p+3}{4}}\big(\int_\R u_x^2\big)^{\frac{p-1}{4}}.
\ee
For $m=5$, solitons are shown to be \emph{unstable} and the Cauchy problem for the corresponding gKdV equation has finite-time blow-up solutions, see \cite{Me,MMannals,MMblow} and references therein. It is believed that for $m> 5$ the situation is the same. Consequently, in this work, \emph{we will discard high-order nonlinearities}, at leading order. 

\smallskip

In addition, there exists another conservation law, formally valid only for $L^1(\R)$ solutions:
\be\label{L1}
 \int_\R u(t,x)dx = \hbox{constant}.
\ee

The problem to consider in this paper possesses a long and extense physical literature. In the next subsection we briefly describe the main results concerning the propagation of solitons in slowly varying medium. 


\subsection{Statement of the problem, historical review} The dynamical problem of soliton interaction with a slowly varying medium is by now a classical problem in nonlinear wave propagation. By soliton-medium interaction we mean, loosely speaking, the following problem: In (\ref{gKdV0}), consider a nonlinear function $f=f(t,x,s)$, slowly varying in space and time, possibly of small amplitude, of the form
$$
f(t,x,s) \sim s^m \quad \hbox{ as } x\to \pm \infty, \quad \hbox{ for all time;}
$$
(namely (\ref{gKdV0}) behaves like a gKdV equation at spatial infinity.) \emph{Consider} a soliton solution of the corresponding variable coefficient equation (\ref{gKdV0}) with this nonlinearity, at some early time. Then we expect that this solution does interact with the medium in space and time, here represented by the nonlinearity $f(t,x,s)$. In a slowly varying medium this interaction, small locally in time, may be significantly important on the long time behavior of the solution. The resulting solution after the interaction is precisely the object of study. In particular, one considers if any change in size, position, or shape, even creation or destruction of solitons, after some large time, may be present.    

Let us review some relevant works in this direction. After the works of Fermi, Pasta and Ulam \cite{FPU}, Zabusky and Kruskal \cite{KZ} (see \cite{Miura} for a review), where complete integrability was established for KdV and other equations, a new branch of research emerged to study the dynamics of solitons solutions of KdV under a slowly varying (\emph{in time}) medium. Kaup and Newell \cite{KN1} and after Karpman and Maslov \cite{KM1} considered the study of perturbations of integrable equations, in particular, they considered the perturbed (in time $\tau$) gKdV equation
\be\label{Kaup}
u_\tau + ( \beta(\ve \tau) u_{xx} + \alpha(\ve \tau) u^m)_x = 0, \quad m=2,3; \quad \alpha,\beta >0.
\ee
This last equation models, for example, the propagation of a wave governed by the KdV equation along a \emph{canal of varying depth}, among many other physical situations, see \cite{KM1, Asano} and references therein.

Note that this equation leaves invariant (\ref{M}) and (\ref{L1}), but the corresponding energy for this equation is not conserved anymore.
After the transformation $t:= \int_0^\tau  \beta(\ve s)ds$, $\tilde u(t,x) := \big(\frac{\al}{\beta} \big)^{\frac 1{m-1}} (\ve \tau )u(\tau, x),$ the above equation becomes
\be\label{utilde}
\tilde u_t + (\tilde u_{xx} + \tilde u^m)_x = \ve \ga(\ve t) \tilde u, \quad \hbox{ where }\quad \ve \ga(\ve t) := \frac 1{m - 1}\partial_t \big[\log (\frac{\alpha}{\beta}) (\ve \tau(t))\big].
\ee
The authors then performed a perturbative analysis of the inverse scattering theory to describe the dynamics of a soliton (for the integrable equation) in this variable regime. Interestingly enough, the existence of a \emph{dispersive shelf-like tail behind the soliton} was formally described. This phenomena is indeed related to the \emph{lack of energy conservation} (\ref{E}) for the equation (\ref{utilde}).

Subsequently, this problem has been addressed in several other works and for different integrable models, see for example  \cite{KK,FFR,Gr1,Gr2}. Moreover, using inverse-scattering techniques, the production of a \emph{second --small-- solitary wave} was pointed out in \cite{W}, see also \cite{Gr3}, but an analytical and satisfactory mathematical proof of this phenomenon is by now out of reach of the current technology. The reader may consult e.g. the monograph by Newell \cite{New}, pp. 87--97, for a more detailed account of the problem.

In addition, another important motivation for the study of this problem comes from an interesting point of view, given in Lochak \cite{Lo}, see also \cite{LoMe} for a more detailed description. Based in formal conservation laws, the author points out that, in the case of equation (\ref{utilde}), well-modulated solitons are good candidates to be \emph{adiabatically stable} objects for this infinite dimensional dynamical system. See \cite{Lo, LoMe} for more details.

\smallskip

In this paper we address the problem of soliton dynamics in the case of a slowly varying, inhomogeneous medium, but constant in time. This model, from the mathematical point of view, introduces several difficulties to the study of the dynamical problem, as we will see below; but at the same time reproduces the production of a shelf-like tail behind the soliton, formally seen by physicists. Our main result states that, as a consequence of this tail, there is no pure soliton-solution (unlike gKdV) for this regime. This result illustrates the \emph{lack of pure solutions} of non-trivial perturbations of gKdV equations.

\smallskip

Now let us explain in detail the model we will study along this paper.

\subsection{Setting and hypotheses} 

Let us come back to the general equation (\ref{gKdV0}), and consider $\ve>0$ a small parameter. Following equation (\ref{Kaup}), along this work we will assume that the nonlinearity $f$ is a slowly varying $x$-dependent function of the power cases, independent of time, plus a (possibly zero) linear term:
\be\label{f}
\begin{cases}
f(x,s) := -\la s + a_\ve(x) s^m, \quad \la \geq 0, \;  m=2,3 \hbox{ and } 4. \\
a_\ve (x) : =a(\ve x) \in C^3(\R).
\end{cases}
\ee
We will suppose the parameter $\la$ fixed, independent of $\ve$. Concerning the function $a$ we will assume that there exist constants $K, \ga>0$ such that
\be\label{ahyp} 
\begin{cases}
1< a(r) < 2, \quad a'(r)>0 \; \hbox{ for all } r\in \R, \\
0<a(r) -1 \leq  Ke^{\ga r}, \; \hbox{ for all } r\leq 0, \, \hbox{ and} \\
0<2-a(r)\leq K e^{-\ga r} \; \hbox{ for all } r\geq 0.
\end{cases}
\ee
In particular, $\lim_{r\to -\infty} a(r) =1$ and $\lim_{r\to +\infty} a(r) =2$. 
%
We emphasize that the special choice ($1$ and $2$) of the limits are irrelevant for the results of this paper. The only necessary conditions are that
$$
0<a_{-\infty} :=  \lim_{r\to -\infty} a(r) <\lim_{r\to +\infty} a(r) =: a_{\infty}<+\infty.
$$

Of course the decay hypothesis on $a$ in (\ref{ahyp}) can be relaxed, and the results of this paper still should hold, with more difficult proofs; but for brevity and clarity of the exposition these issues will not be considered in this work.

Finally, to deal with a special stability property of the mass in Theorems \ref{Tm1} and \ref{Tp1} (cf. also (\ref{tm2})), we will need the following additional (but still general) hypothesis: there exists $K>0$ such that for $m=2,3$ and $4$,
\be\label{3d1d}
| (a^{1/m})^{(3)}(s) | \leq K (a^{1/m})'(s), \quad \hbox{ for all } s\in \R.
\ee
This condition is generally satisfied, however $a'$ must not be a compact supported function.

Recapitulating, given $0\leq \la <1$ , we will consider the following \emph{aKdV} equation
\be\label{aKdV}
\begin{cases}
u_t + (u_{xx} -\la u + a_\ve (x) u^m)_x =0 \quad \hbox{ in } \R_t \times \R_x,\\
m=2,3 \hbox{ and } 4;\quad  0< \ve\leq\ve_0;  \quad a_\ve \hbox{ satisfying } (\ref{ahyp}) \hbox{-}(\ref{3d1d}). 
\end{cases}
\ee

The main issue that we will study in this paper is the interaction problem between a soliton and a slowly varying medium, here represented by the \emph{potential} $a_\ve$. In other words, we intend to study for (\ref{aKdV}) whether it is possible to generalize the well-known soliton-like solution $Q$ of gKdV. Of course, it is by now well-known that in the case $f(t,x,s) =f(s)$, and under reasonable assumptions (see for example Berestycki and Lions \cite{BL}), there exist soliton-like solutions, constructed via \emph{ground states} of the corresponding elliptic equation for a \emph{bound state}. However, in this paper our objective will be the study of soliton solutions under a variable coefficient equation, where no evident ground state is present.

\medskip

To support our beliefs, note that at least heuristically, (\ref{aKdV}) behaves at infinity as a gKdV equation:
\be\label{desc}
\begin{cases}
u_t + (u_{xx} -\la u + 1u^m)_x =0 \quad \hbox{ as } x\to -\infty, \\
u_t + (u_{xx}-\la u + 2 u^m)_x =0 \quad \hbox{ as } x\to +\infty.
\end{cases}
\ee
In particular, one should be able of to construct a soliton-like solution $u(t)$ of (\ref{aKdV}) such that 
$$
u(t) \sim Q(\cdot -(1-\la)t ), \quad \hbox{ as } t\to -\infty,
$$
in some sense to be defined. Here $Q$ is the soliton of the standard gKdV equation given by (\ref{soliton}). Indeed, note that $Q(\cdot -(1-\la)t)$ is an actual solution for the first equation in (\ref{desc}), but on the whole real line, moving towards the left , to the right, or being a steady state depending on $\la >1$, $\la<1$ or $\la=1$ respectively.

On the other hand, after passing the interaction region, by stability properties, this solution \emph{should behave} like (for small $\ve$)
\be\label{1d2Q}
{2^{-\frac 1{m-1}}} Q_{c_\infty}(x- (c_\infty-\la) t -\rho(t)) + \hbox{ lower order terms in $\ve$}, \quad \hbox{ as } t\to +\infty, 
\ee
where $c_\infty$ is a unknown, positive number, a limiting scaling parameter, and $\rho(t)$ small compared with $(c_\infty-\la)t$. In fact, note that if $v =v(t)$ is a solution of (\ref{gKdV}) then $u(t) := 2^{-\frac 1{m-1}} v(t)$ is a solution of 
\be\label{simpli}
u_t + (u_{xx} -\la u + 2 u^m)_x =0 \quad \hbox{ in } \R_t \times \R_x.
\ee
In conclusion, this heuristic suggests that even if the potential varies slowly, the soliton will experiment non trivial transformations on its scaling and shape, of the same order that of the amplitude variation of the potential $a$. 

\smallskip

Before stating our results, some important facts are in order. First, unfortunately equation (\ref{aKdV}) is not anymore invariant under scaling and spatial translations. Moreover, a nonzero solution of (\ref{aKdV}) \emph{might lose or gain some mass}, depending on the sign of $u$, in the sense that, at least formally, the quantity
\be\label{Ma}
M[u](t)=\frac 12 \int_\R u^2(t,x)\,dx
\ee
satisfies the identity
\be\label{dMa}
 \partial_t M[u](t) = -\frac{\ve}{m+1} \int_\R a'(\ve x) u^{m+1}.
\ee
Another key observation is the following: in the cubic case $m=3$, with our choice of $a_\ve$, the mass is always non increasing. This simple fact will have important consequences in our results, at the point of saying that the cubic case corresponds to a well-behaved problem, a sort of good generalization of the pure power case. 

On the other hand, the novel energy ($\la\geq 0$)
\be\label{Ea}
E_a [u](t) :=  \frac 12 \int_\R u_x^2(t,x)\,dx + \frac \la 2 \int_\R u^2(t,x)\, dx - \frac 1{m+1}\int_\R a_\ve (x)  u^{m+1}(t,x)\,dx
\ee
remains formally constant for all time. Moreover, a simple energy balance at $\pm \infty$ allows to determine heuristically the limiting scaling in (\ref{1d2Q}), for example in the case $\la=0$, if we suppose that the \emph{lower order terms} are of zero mass at infinity. Indeed, from (\ref{1d2Q}) we have
$$
E_{a\equiv 1}[u](-\infty) = E[Q] \sim 2^{-\frac 2{m-1}} c_\infty^{\frac 2{m-1} +\frac 12} E[Q] = E_{a\equiv 2}[u](+\infty) , \quad E[Q]\neq 0,
$$
(cf. Appendix \ref{IdQ}). This implies that $ c_\infty \sim 2^{\frac 4{m+3}}>1$. 
These formal arguments suggest the following definition.
\begin{defn}[Pure generalized soliton-solution for aKdV]\label{PSS}~

Let $0\leq \la <1$ be a fixed number. We will say that (\ref{aKdV}) admits a \emph{pure} generalized soliton-like solution (of scaling equals $1$) if there exist a $C^1$ real valued function $\rho=\rho(t)$ defined for all large times and a global in time $H^1(\R)$ solution $u(t)$ of (\ref{aKdV}) such that 
$$
\lim_{t\to - \infty}\|u(t) -Q(\cdot -(1-\la)t)\|_{H^1(\R)} = \lim_{t\to +\infty} \big\|u(t) - 2^{-\frac 1{m-1}}Q_{c_\infty} (\cdot - \rho(t)) \big\|_{H^1(\R)} =0,
$$
with $\lim_{t\to +\infty} \rho(t) =+\infty$, and where $c_\infty= c_\infty(\la)$ is the scaling suggested by the energy conservation law (\ref{Ea}).
\end{defn}

\begin{rem}
Note that the existence of a translation parameter $\rho(t)$ is a necessary condition: it is even present in the orbital stability of small perturbations of solitons for gKdV, see e.g. \cite{Benj, BSS, CL}. Note that we have not included the case $\rho(t) \to -\infty $ as $t\to +\infty$ (= a reflected soliton), but we hope to consider this case elsewhere.
\end{rem}

\subsection{Previous analytic results on the soliton dynamics under slowly varying medium}

The problem of describing analytically the soliton dynamics of different integrable models under a slowly varying medium has received some increasing attention during the last years. Concerning the KdV equation, our belief is that the first result in this direction was given by Dejak, Jonsson and Sigal in \cite{SJ,DS}. They considered the long time dynamics of solitary waves (solitons) over slowly varying perturbations of KdV and mKdV equations
\be\label{DJSigal}
u_t + (u_{xx} -  b( t, x)u + u^m)_x =0 \quad \hbox{ on } \R_t\times \R_x, \quad m=2,3,
\ee
and where $b$ is assumed having small size and small variation, in the sense that for $\ve$ small,  
$$
|\partial_t^n \partial_x^p b| \leq \ve^{n+p+1}, \quad \hbox{ for } 0\leq n+ p\leq 2. 
$$ 
(Actually their conclusions hold in more generality, but for our purposes we state the closest version to our approach, see \cite{SJ} for the detailed version.) With these hypotheses the authors show that if $m=2$ and the initial condition $u_0$ satisfies the \emph{orbital stability} condition 
$$
\inf_{\substack{ 0<c_0<c<c_1 \\ a\in \R}}\|u_0 -Q_c(\cdot -a)\|_{H^1(\R)} \leq \ve^{2s},\quad  s<\frac 12, \; c_0,c_1 \hbox{ given},
$$
then for any for time $t\leq K\ve^{-s}$ the solution can be decomposed as
$$
u(t,x) = Q_{c(t)} (x-\rho(t)) + w(t,x),  
$$
where $ \|w(t)\|_{H^1(\R)} \leq K \ve^s$ and $\rho(t), c(t)$ satisfies the following differential system
$$
\rho'(t) = c(t) -b(t,a(t)) + O(\ve^{2s}), \quad c'(t) = O(\ve^{2s});
$$
during the above considered interval of time. In the cubic case ($m=3$) their results are slightly better, see \cite{SJ}.  

Note that our model can be written as a generalized, time independent Dejak-Jonsson-Sigal equation of the type (\ref{DJSigal}), after writting $v(t,x):= \tilde a (\ve x) u(t,x)$, with $\tilde a(\ve x) := a^{\frac 1{m-1}}(\ve x)$. 
From these considerations we expect to recover and to improve the results that they have obtained.

Very recently Holmer \cite{Holmer} has announced some improvements on the Dejak-Sigal results, by assuming $b$ of amplitude $O_{L^\infty}(1)$. He proves that 
$$
 \sup_{t\lesssim \delta \ve^{-1}|\log\ve| }\|w(t)\|_{H^1(\R)} \lesssim \ve^{1/2-\delta},  
$$
for some $\delta>0$. 

 \medskip
 
In this paper, in order to achieve a deep understanding of the phenomenon we have preferred to avoid the inclusion of a time depending potential, and to treat  the infinite time prescribed and pure data, instead of the standard Cauchy problem. This election will be positively reflexed in the main Theorem, where we will describe with accuracy the dynamical problem, including its asymptotics as $t\to +\infty$.

The interaction soliton-potential can be also considered in the case of the nonlinear Schr\"odinger equation
\be\label{SCH}
iu_t + u_{xx} -V(\ve x)u + |u|^2 u =0,  \quad \hbox{ on }\quad  \R_t\times \R_x.
\ee
We mention the recent works of Holmer, Marzuola and Zworski \cite{HZ,HMZ0, HMZ} and Gustafson et al. \cite{GFJS, FGJS}, where similar results to the above one were obtained. Finally we point out the very recent work of Perelman \cite{Pe}, concerning the critical quintic NLS equation.

\subsection{Main Results}

Let
\be\label{Te}
T_\ve := \frac 1{1-\la}\ve^{-1 -\frac 1{100}}>0.
\ee
This parameter can be understood as the \emph{interaction time} between the soliton and the potential. In other words, at time $t=-T_\ve$ the soliton should remain almost unperturbed, and at time $t=T_\ve$ the soliton should have completely crossed the influence region of the potential. Note that the asymptotic $\la \sim 1$ is a degenerate case and it will be discarded for this work.

In this paper we will prove (cf. Theorems \ref{MT}, \ref{LTB} and \ref{MTcor}) that for a suitable general case a pure soliton-like solution as in Definition \ref{PSS} does not exist, in the sense that the lower order terms appearing after the interaction have always positive mass. This phenomenon will be a consequence of the dispersion produced during the crossing of the soliton with the main core of the potential $a_\ve$.

\medskip

In what follows, we assume the validity of above hypotheses, namely (\ref{f}), (\ref{ahyp}) and (\ref{3d1d}). As has  been previously claimed, our first result describes in accuracy the dynamics of the \emph{pure} soliton-like solution for aKdV (\ref{aKdV}).

\begin{thm}[Dynamics of interaction of solitons for gKdV equations under variable medium]\label{MT}~

Let $m=2,3$ and $4$, and let $0\leq \la\leq \la_0 := \frac{5-m}{m+3}$ be a fixed number. There exists a small constant $\ve_0>0$ such that for all $0<\ve<\ve_0$ the following holds. 

\begin{enumerate}

\item \emph{Existence of a soliton-like solution}. There exists a solution $u\in C(\R, H^1(\R))$ of (\ref{aKdV}), global in time, such that 
\be\label{Minfty}
\lim_{t\to -\infty} \|u(t) - Q(\cdot -(1-\la)t) \|_{H^1(\R)} =0,
\ee
with conserved energy $E_a[u](t) = (\la-\la_0)M[Q] \leq 0.$
This solution is unique in the cases $m=3$; and $m=2,4$ provided $\la>0$.

\smallskip

\item \emph{Interaction soliton-potential}. There exist $K>0$ and numbers $c_\infty(\la) \geq1$,  $\rho_\ve, \tilde T_\ve \in \R $ such that the solution $u(t)$ above constructed satisfies
\be
\big\|u(\tilde T_\ve) - 2^{-1/(m-1)}Q_{c_\infty} (x-  \rho_\ve) \big\|_{H^1(\R)}\leq K\ve^{1/2}.
\ee
Moreover, 
\be\label{PARA}
c_\infty(\la=0) = 2^{\frac{4}{m+3}}, \quad  \hbox{ and } \quad c_\infty(\la=\la_0) =1.
\ee
Finally we have the bounds
\be\label{Para}
|T_\ve -\tilde T_\ve|  \leq \frac{T_\ve}{100};  \quad (1-\la) T_\ve \leq \rho_\ve \leq (2c_\infty(\la) -\la-1) T_\ve.
\ee
\end{enumerate}
\end{thm}


\begin{rem}
Let us say some words about the special parameter $\la_0$ from above. First, note that $\la_0 =\la_0(m)$ is always less than $1$ for $m=2,3$ and $4$; with $\la_0(m=5)=0$ (= the $L^2$-critical case). In addition, note that for $\la=\la_0$ we have $E_a[u](t)=(\la-\la_0)M[Q]=0$; and if $\la<\la_0$ one has $E_a[u](t)<0$, for all $t\in \R$. For more details about the consequences of this property, and a detailed study of $c_\infty(\la)$ see Lemma \ref{ODE}.   
\end{rem}

\begin{rem}
The proof of this result is based on the construction of an approximate solution of (\ref{aKdV}) in the interaction region, satisfying certain symmetries. However, at some point we formally obtain an {\bf infinite mass term}, see also \cite{MMfin} for a similar problem. It turns out that to obtain a localized solution we need to break the symmetry of this solution (see Proposition \ref{CV} for the details). This lack of symmetry leads to the error $\ve^{1/2}$ in the theorem above stated. At this price we overtake the interaction region, a completely new result. 
\end{rem}

The next step is the understanding of the long time behavior of our generalized soliton solution. 
\begin{thm}[Long time behavior]\label{LTB}~

Under the assumptions of Theorem \ref{MT}, suppose now in addition that $0<\la\leq\la_0$ for the cases $m=2,4$, and $0\leq \la \leq \la_0$ if $m=3$. Let $0<\beta <\frac 12 (c_\infty(\la) -\la)$. There exists a constant $\ve_0>0$ such that for all $0< \ve\leq \ve_0$ the following hold. 

There exist $K,c^+>0$ and a $C^1$-function $\rho_2(t)$ defined in $[T_\ve,+\infty)$ such that
$$
w^+(t,\cdot) := u(t,\cdot) - 2^{-1/(m-1)} Q_{c^+} (\cdot -\rho_2(t)) 
$$   
satisfies
\begin{enumerate}
\item \emph{Stability}. For any $t\geq T_\ve$,
\be\label{MT2}
\|w^+(t)\|_{H^1(\R)}+ |c^+-c_\infty(\la)| +|\rho_2'(t) -(c_\infty(\la) -\la)|\leq K\ve^{1/2}.
\ee
\item \emph{Asymptotic stability}. 
\be\label{MT3}
\lim_{t\to +\infty} \| w^+(t)\|_{H^1(x>\beta t)} =0.
\ee
\item \emph{Bounds on the parameters}. Define $\theta := \frac 1{m-1} -\frac 14>0.$ The limit
\be\label{MT4}
\lim_{t\to +\infty} E_a[w^+](t) =: E^+ 
\ee
exists and satisfies the identity
\be\label{Pc}
E^+  =   \frac {(c^+)^{2\theta}}{2^{2/(m-1)}}(\la_0c^+ -\la)M[Q] + (\la-\la_0)M[Q],
\ee
and for all $m=2,3,4$ and $0<\la\leq \la_0$ there exists $K(\la)>0$ such that
\be\label{Pc1}
\frac 1K \limsup_{t\to +\infty} \|w^+(t)\|_{H^1(\R)}^2 \leq E^+ \leq K\ve.  
\ee
Furthermore, in the case $m=3$, $\la=0$, we have $\frac 32 E^+  = (\frac{c^+}{c_\infty}\big)^{3/2} -1$, and for all $\la>0$,
\be\label{Pc2}
 \frac 1K \limsup_{t\to +\infty }\|w^+(t)\|_{H^1(\R)}^2 \leq  \big( \frac{c^+}{c_\infty} \big)^{2\theta} -1 \leq K\ve.
\ee
\end{enumerate}

\end{thm}

\begin{rem}
Stability and asymptotic stability of solitary waves for generalized KdV equations have been widely studied since the '80s. The main ideas of our proof are classical in the literature. For more details, see e.g. \cite{Benj,CL,BSS,MMT,PW}.
\end{rem}

\begin{rem}
The sign of $a'(\cdot)$ is a sufficient condition to ensure stability; however, it may be possibly relaxed by assuming for example the weaker condition $a'(s)>0$ for all $s>s_0$. In this paper we will not pursue on these assumptions. It is not known whether under more general potentials stability still holds true, see also below. 
\end{rem}

\begin{rem}[Decreasing potential] Pick now a potential $a(\cdot)$ satisfying $a'(s)<0$ and 
$$
1=\lim_{s\to -\infty} a(s) >a(s)>\lim_{t\to +\infty} a(s) =\frac 12. 
$$
Let us explain the main changes in the above theorems. First of all, Theorem \ref{MT} part (1) holds true, however we do not know whether the solution constructed is unique. On the other hand, part (2) holds true with the coefficient $2^{\frac 1{m-1}}$ in front of $Q_{c_\infty}$, $\frac \la{\la_0}<c_\infty (\la) <1$, and $c_\infty(\la =0) =2^{-p}$ (see Lemma \ref{ODE} to see this). (\ref{Para}) holds true with the obvious changes. Finally, Long time stability (=Theorem \ref{LTB}) for this case is an open question.
\end{rem}

A fundamental question arises from the above results, namely is the final solution an exactly pure soliton for the aKdV equation with $a_\ve\equiv 2$? (cf. Definition \ref{PSS}.) This question is equivalent to decide whether  
$$
\limsup_{t\to +\infty} \|w^+(t)\|_{H^1(\R)} =0.
$$
Our last result shows that indeed this behavior cannot happen.

\begin{thm}[Non-existence of pure soliton-like solution for aKdV]\label{MTcor}~

Under the context of Theorems \ref{MT} and \ref{LTB}, suppose $m=2,3,4$ with $0<\la\leq \la_0$. There exists $\ve_0>0$ such that for all $0<\ve<\ve_0$, 
\be\label{MTcor1}
\limsup_{t\to +\infty} \|w^+(t)\|_{H^1(\R)} >0. 
\ee
\end{thm}

\begin{rem}
We have been unable to solve several questions related to these results. In addition to the classical problem of the extension of these results to more general \emph{potentials} $a(\cdot)$, we have the following questions in mind:

\begin{enumerate}
\item A first basic question is to decide if every solution of (\ref{aKdV}) with $H^1(\R)$ data is globally bounded in time. In Proposition \ref{GWP0} we prove that every solution is globally well defined for all positive times, and uniformly bounded if  $\la >0$ or $m=3$. However, for the cases $m=2,4$ and $\la=0$ we only have been able to find an exponential upper bound on the $H^1$-norm of the solution. Is every solution described in Theorem \ref{MT} globally bounded? 

\item In the cases $m=2,4$ and $\la=0$, is the solution constructed in Theorem \ref{MT} unique? Is it stable for large times? (cf. Theorem \ref{Tp1}). 

\item What is the behavior of the solution for a coefficient $\la_0<\la<1$? We believe in this situation the soliton still survives, but becomes reflected to the left by the potential.

\item It is possible to obtain in Theorem \ref{MTcor} a quantitative lower bound on the defect at infinity? 

\item Is there scattering modulo the soliton solution, at infinity?
\end{enumerate} 
\end{rem}

\begin{rem}
The case of the Schr\"odinger equation considered in (\ref{SCH}) will be treated in another publication (see \cite{Mu2}.)
\end{rem}

\begin{rem}[Case of a time dependent potential]
As expected, our results are also valid, with easier proofs, for the following time dependent gKdV equation:
\be\label{time}
u_t + (u_{xx} -\la u + a(\ve t) u^m)_x =0, \quad \hbox{ in } \R_t\times \R_x.
\ee
Here $a$ satisfies (\ref{ahyp})-(\ref{3d1d}) now in the time variable. Note that this equation is invariant under scaling and space translations. In addition, the $L^1$ integral and the mass $M[u]$ remain constants and the energy 
$$
\tilde E[u](t) := \frac 12 \int_\R u_x^2 +\frac \la 2 \int_\R u^2 -\frac{a(\ve t)}{m+1} \int_\R u^{m+1}
$$
satisfies
$$
\partial_t \tilde E[u](t) = -\frac{\ve a'(\ve t)}{m+1}\int_\R u^{m+1}.
$$ 
Furthermore, Theorems \ref{MT} and \ref{LTB} still hold with $c_\infty(\la =0) = 2^{4/(5-m)}$ (because of the mass conservation), for \textbf{any} $\la\geq 0$, $m=2,3$ and 4 (follow Lemma \ref{ODE} to see this). We left the details to the  reader.
\end{rem}


Before starting the computations, let us explain how the proof of the main results work.

\subsection{Sketch of proof}\label{sop}

Our arguments are an adaptation of a series of works by Martel, Merle and Mizumachi \cite{Martel, MMcol1, MMMcol, MMas2, MMfin}, and some new computations. The idea is as follows: we separate the analysis among three different time intervals: $t\ll -\ve^{-1} $, $\abs{t} \leq \ve$ and $\ve^{-1} \ll t$. On each interval the solution possesses a specific behavior which we briefly describe:      

\begin{enumerate}
\item  ($t\ll -\ve^{-1}$). In this interval of time we prove that $u(t)$ remains very close to a soliton-solution with no change in the scaling and shift parameters (cf. Theorem \ref{Tm1}). This result is possible for negative very large times, where the soliton is still far from the interacting region $|t|\leq \ve^{-1}$. 

\item ($\abs{t}\leq \ve^{-1}$). Here the soliton-potential interaction leads the dynamics of $u(t)$. The novelty here is the construction of an \emph{approximate solution} of (\ref{aKdV}) with high order of accuracy such that $(a)$ at time $t\sim -\ve^{-1}$ this solution is close to the soliton solution and therefore to $u(t)$; $(b)$ it describes the soliton-potential interaction inside this interval, in particular we show the existence of a remarkable dispersive tail behind the soliton; and $(d)$ it is close to $u(t)$ in the whole interval $[-\ve^{-1}, \ve^{-1}]$, uniformly on time, modulo a modulation on a translation parameter (cf. Theorem \ref{T0}). 

\item ($t\gg \ve^{-1}$)  Here some \emph{stability} properties (see Theorem \ref{Tp1}) will be used to establish the convergence of the solution $u(t)$ to a soliton-like solution with modified parameters. 

\end{enumerate}

Additionally, by using a contradiction argument, it will be possible to show that the residue of the interaction at time $t\sim \ve^{-1}$ is still present at infinity. This gives the conclusion of the main Theorems \ref{MT}, \ref{MTcor}. 
Indeed, recall the $L^1$ conserved quantity from (\ref{L1}). This expression is in general useless when the equation is considered on the whole real line $\R$, however it has some striking applications in the blow-up theory (see \cite{Me}). In our case, it will be useful to discard the existence of a pure soliton-like solution. 

\begin{rem}[General nonlinearities]
We believe that the main results of this paper are also valid for general, subcritical nonlinearities, with stable solitons. In this case the scaling property of the soliton is no longer valid, so in order to construct an approximate solution one should modify the main argument of the proof.  
\end{rem}

Finally, some words about the organization of this paper, according to the sketch above mentioned. First in Section \ref{2} we introduce some basic tools to study the interaction and asymptotic problems. Next, Section \ref{3} is devoted to the construction of the soliton like solution for negative large time. Sections \ref{sec:2} and \ref{sec:3} deal with the proof of Theorem \ref{MT}. In Section \ref{6} we proof the asymptotic behavior as $t\to +\infty$, namely Theorem \ref{LTB}. Finally we prove Theorem \ref{MTcor} (Section \ref{7}).


\bigskip

\section{Preliminaries}\label{2}

In this section we will state several basic but important properties needed in the course of this paper. 

\subsection{Notation}

Along this paper, both $C, K,\ga>0$ will denote fixed constants, independent of $\ve$, and possibly changing from one line to the other.

Finally, in order to treat the case $\la>0$ we need to extend the energy (\ref{E}) by adding a mass term. Let us define
\be\label{E0}
E_1[u](t) := \frac 12\int_\R u_x^2(t) +\frac \la 2 \int_\R u^2(t) -\frac 1{m+1} \int_\R u^{m+1}(t),
\ee
namely $E_1[u] = E_{a\equiv 1} [u]$.

\subsection{Cauchy Problem} 

First we develop a suitable local well-posedness theory for the Cauchy problem associated to (\ref{aKdV}). 

Let $u_0\in H^s(\R)$, $s\geq1$, $ \la \geq 0$.  We consider the following initial value problem 
\be\label{Cp1}
\begin{cases}
u_t + (u_{xx} -\la u +a_\ve (x) u^m)_x = 0 \quad \hbox{ in } \R_t \times \R_x \\
u(t=0) =  u_0,
\end{cases}
\ee
where $m=2,3$ or $4$. The equivalent problem for the generalized KdV equations (\ref{gKdV}) has been extensively studied, but for our purposes, in order to deal with (\ref{Cp1}), we will follow closely the contraction method developed in \cite{KPV}. We have the following result.

\begin{prop}[Local well-posedness in $H^s(\R)$, see also \cite{KPV}]\label{Cauchy}~

Suppose $u_0\in H^s(\R),$ $s\geq 1$. Then there exist a maximal interval of existence $I$ $($with $0\in I)$, and a unique (in a certain sense) solution $u\in C(I, H^s(\R))$ of (\ref{Cp1}). Moreover, the following properties hold:

\begin{enumerate}
\item\label{BU0} \emph{Blow-up alternative}. If $\sup I < +\infty$, then
\be\label{BU}
\lim_{t\uparrow \sup I } \|u(t)\|_{H^s(\R)} =+\infty.
\ee
The same conclusion holds in the case $\inf I>-\infty$.
\item\label{BUE} \emph{Energy conservation}. For any $t\in I$ the energy  $E_a[u](t)$ from (\ref{Ea}) remains constant.
\item\label{BUM} \emph{Mass variation}. For all $t\in I$ the mass $M[u](t)$ defined in (\ref{Ma}) satisfies (\ref{dMa}).
\item\label{BUL1} Suppose $u_0\in L^1(\R)\cap H^1(\R)$. Then (\ref{L1}) is well defined and remains constant for all $t\in I$.
\end{enumerate}
\end{prop}

\begin{proof}
The proof is standard, and it is based in straightforward application of the Picard iteration procedure and the tools developed in \cite{KPV}. We skip the details. 
\end{proof}
 
Once a local-in-time existence theory is established, the next step is to ask for the possibility of a \emph{global well-posedness theorem}. In many cases the proof reduces to the \emph{use of conservation laws} to obtain some bounds on the norm of the solution for every time. In the case of gKdV equations ($m\leq 4$) this was proved in \cite{KPV} by using the mass and energy conservation; however, in our case relation (\ref{dMa}) is not enough to control the $L^2$ norm of the solution. As we had stated in the Introduction, the global existence for cubic case $m=3$ follows from the mass decreasing property. However, to deal with the remaining cases, we will modify our arguments by introducing a perturbed mass, almost decreasing in time, in order to prove global existence. Indeed, define for each $t\in I$, $m=2,3$ and $4$,
\be\label{ttM}
\hat M[u](t) := \frac 12 \int_\R a_\ve^{1/m}(x) u^2(t,x) dx.
\ee
It is clear that $\hat M[u](t)$ is a well defined quantity, for any time $t\in I$ and $u$ solution of (\ref{Cp1}) in $H^1(\R)$. Note also that for all $t\in I$ we have the equivalence relation $M[u](t) \leq \hat M[u](t) \leq 2^{1/m}M[u](t).$
 
This modified mass enjoys of a striking property, as is showed in the following

\begin{prop}[Global existence in $H^1(\R)$]\label{GWP0}~

Consider $u(t)$ the solution of the Cauchy problem (\ref{Cp1}) with $u(0)=u_0\in H^1(\R)$ and maximal interval of existence $I$. Then $u(t)$ is continuously well defined in $H^1(\R)$ for any $t\geq 0$. More precisely, the following properties hold. 

\begin{enumerate}

\item\label{03} \emph{Cubic case}. Suppose $m=3$, $\la \geq 0$. Then $I=(\tilde t_0, +\infty) $ for some $-\infty \leq \tilde t_0 < 0$ and there exists $K=K(\|u_0\|_{H^1(\R)})>0$ such that 
\be\label{UBT}
\sup_{t\geq 0} \|u(t)\|_{H^1(\R)} \leq K.
\ee
\item\label{DtM} \emph{Almost monotony of the modified mass $\hat M$ and global existence}. For any $m=2,3$ and $4$, and for all $t\in I$ we have
\be\label{hM}
\partial_t \hat M[u](t)  = -\frac 32 \ve \int_\R (a^{1/m})'(\ve x) u_x^2 - \frac \ve 2 \int_\R [ \la (a^{1/m})'  -\ve^2 (a^{1/m})^{(3)}] (\ve x) u^2 . 
\ee
In particular, 
\begin{enumerate}
\item $I$ is of the same form as above;
\item If $\la>0$ there exists $\ve_0>0$ small such that (\ref{UBT}) holds;
\item if $\la =0$ and $m=2,4$, then we have for all $t\geq 0$ the exponential bound
\be\label{UBTexp}
\|u(t)\|_{H^1(\R)} \leq K e^{K\ve^3 t},
\ee
for some $K=K( \|u_0\|_{H^1(\R)})$.
\end{enumerate}

\end{enumerate}

\end{prop}
%

\begin{proof}[Proof of Proposition \ref{GWP0}]
First we consider the cubic case, namely $m=3$. From (\ref{dMa}) we have for any $t\in I$, $t\geq 0$
$$
M[u](t) \leq M[u](0).
$$ 
This bound implies the global existence for positive times. Indeed, the above bound rules out the $L^2$ blow-up in (positive) \emph{finite and infinite time} scenario, namely (\ref{BU}). In order to control the $H^1(\R)$ norm, we use the energy conservation, the Galiardo-Nirenberg inequality (\ref{GNe}) and the preceding bound on the mass. Indeed, for any $t\in I$, $t\geq 0$, and redefining the constant $K$ if necessary, we have
\bee
\frac 12\int_\R u_x^2  & = & E_a[u](0) -\frac 12\la\int_\R u^2 + \frac 1{m+1}\int_\R a_\ve u^{m+1} \\
& \leq &  E_a[u](0) + \la M[u](0) +K \|u(t) \|_{L^2(\R)}^{(m+3)/2}  \|u_x(t) \|_{L^2(\R)}^{(m-1)/2}.
\eee
Noticing that $\frac 14 (m-1)<1 $ for $m=2,3$ and $4$, we have that
$$
\int_\R u_x^2 \leq  K(\la, \|u_0\|_{H^1(\R)});
$$
for a large constant $K$. This bound implies the $H^1(\R)$ global existence for all positive times and the uniform bound in time (\ref{UBT}). The bound (\ref{UBT}) is direct.

In order to prove (\ref{hM}), we proceed by formally taking the time derivative. Every step can be rigorously justified by introducing mollifiers. From the equation (\ref{aKdV}) we have
\bee
\partial_t \hat M[u](t) &  = & \int_\R a_\ve^{1/m} u u_t  =  \int_\R (a_\ve^{1/m} u)_x  (u_{xx}-\la u + a_\ve u^m) \\
& =& \ve \int_\R ( (a^{1/m})' (\ve x) u u_{xx}   -  \frac 12 (a^{1/m})'(\ve x)  u_x^2  )  -\frac \la 2 \ve \int_\R (a^{1/m})' (\ve x) u^2 \\
& & \qquad + \ve \int_\R a_\ve (a^{1/m})' (\ve x) u^{m+1} - \frac \ve{m+1} \int_\R (a^{1/m+1})'(\ve x)  u^{m+1}\\
& = & - \frac 12 \ve \int_\R [ \la (a^{1/m})'(\ve x)  -\ve^2 (a^{1/m})^{(3)} (\ve x)] u^2     -  \frac 32 \ve\int_\R (a^{1/m})'(\ve x)  u_x^2.
\eee
This proves (\ref{hM}). Now, in order to establish global $H^1(\R)$ existence for positive times, we first control the $L^2$ norm using $\hat M[u](t)$. Let us consider the case $\la>0$. In this case, taking $\ve_0$ small enough, and thanks to (\ref{3d1d}), we have
$$
\partial_t \hat M[u](t) \leq 0,
$$
and thus $\hat M[u](t) \leq \hat M[u](0)$ for all $t\in I$, $t\geq 0$. The rest of the proof is identical to the cubic case.

Now we consider the last case, namely $m=2,4$ and  $\la =0$. Here the above argument is not valid anymore and then we have only the existence of $K>0$ independent of $\ve$ such that  
$$
\partial_t \hat M[u](t) \leq K \ve^3 \hat M[u](t).
$$ 
This implies that for any $t\in I$, $t\geq 0$, 
$$
M[u](t)  \leq \hat M[u](t) \leq K\hat M[u](0) e^{K\ve^3 t}.
$$
This bound rules out the $L^2$ blow-up in \emph{finite time} scenario for positive times. To control the $H^1(\R)$ norm, we use the same argument from the preceding case. Indeed, for any $t\in I$, redefining the constant $K$ if necessary, we have
$$
\int_\R u_x^2 \leq  K e^{K\ve^3 t},
$$
for some large constant $K$. This bound implies the $H^1(\R)$ global existence for positive times. The proof is now complete.
\end{proof}

\begin{rem}[Mass monotony]\label{MM}
A conclusion of the above Proposition is the following. Consider $u(t) \in H^1(\R)$ a solution of (\ref{aKdV}). Define de following modified mass
\be\label{tM}
\tilde M[u](t) := \begin{cases}
M[u](t), \quad  \hbox{ if } m=3,  \\
\hat M[u](t), \quad \hbox{ if } m=2,4 \hbox{ and } \la>0.
\end{cases}
\ee
Then there exists $\ve_0>0$ such that for all $0<\ve\leq \ve_0$ and for all $t\in \R$, $t\geq t_0$, one has $\tilde M[u](t)-\tilde M[u](t_0) \leq 0$.
\end{rem}

We will also need some properties of the corresponding linearized operator of (\ref{aKdV}). All the results here presented are by now well-known, see for example \cite{MMcol1}.

\subsection{Spectral properties of the linear gKdV operator}

In this paragraph we consider some important properties concerning the linearized KdV operator associated to (\ref{aKdV}). Fix $c>0$,  $m=2,3$ or 4, and let
\be\label{defLy}
    \mathcal{L} w(y) := - w_{yy} + c w - m Q_c^{m-1}(y) w, \quad\hbox{ where }\quad  Q_c(y) := c^{\frac 1{m-1}} Q(\sqrt{c} y).
\ee
Here $w=w(y)$. We also denote $\mathcal L_0 := \mathcal L_{c=1}$. 

\begin{lem}[Spectral properties of $\mathcal{L}$, see \cite{MMcol2}]\label{surL}~

The operator $\mathcal{L}$ defined (on $L^2(\R)$) by \eqref{defLy}  has domain $H^2(\R)$, it is self-adjoint and satisfies the following properties:
\begin{enumerate}
\item \emph{First eigenvalue}. There exist a unique $\lambda_m>0$ such that  $\mathcal{L}  Q_c^{\frac {m+1}2} =-\lambda_m Q_c^{\frac {m+1}2} $. 
\item The kernel of $\mathcal{L}$ is spawned by $Q'_c$. Moreover,
\be\label{LaQc}
\Lambda Q_c := \partial_{c'} {Q_{c'}}\big|_{ c'=c} = \frac 1c \Big[\frac 1{m-1} Q_c + \frac 12 xQ'_c \Big],
\ee
satisfies $\mathcal{L} (\Lambda Q_c)=- Q_c$. Finally, the continuous spectrum of $\mathcal L$ is given by $\sigma_{cont}(\mathcal L) =[c,+\infty)$.
\item \emph{Inverse}. For all   $h=h(y) \in L^2(\R)$ such that $\int_\R h Q_c'=0$, there exists a unique $\hat h \in H^2(\R)$  such that $\int_\R \hat hQ'_c=0$ and $\mathcal{L} \hat h = h$. Moreover,
        if $h$ is even (resp. odd), then $\hat h$ is even (resp. odd).

\item \emph{Regularity in the Schwartz space $\mathcal S$}. For $h\in H^2(\mathbb{R})$,  $\mathcal{L} h \in \mathcal{S}$ implies $h\in \mathcal{S}$.

\item\label{6a} \emph{Coercivity}. 

\begin{enumerate}
\item
There exists  $K,\sigma_c>0$ such that for all $w\in H^1(\R)$
$$
\qquad \mathcal B[w,w] :=  \int_\R (w_x^2+c w^2 - mQ_c^{m-1} w^2) \geq \sigma_c \int_\R w^2-K\abs{\int_\R  w Q_c}^2 - K\abs{\int_\R w Q_c'}^2.
$$
In particular, if $\int_\R  w Q_c = \int_\R w Q_c'=0,$ then the functional $\mathcal B[w,w]$ is positive definite in $H^1(\R)$. 
\item Now suppose that $\int_\R  w Q_c = \int_\R w xQ_c=0$. Then the same conclusion as above holds.
\end{enumerate}
 
\end{enumerate}
\end{lem}

Now we introduce some notation, taken from \cite{MMcol1}. We denote by $\mathcal{Y}$ the set of $C^\infty$ functions $f$ such that for all $j\in \N$ there exist $K_j,r_j>0$ such that for all $x\in \R$ we have
$$
|f^{(j)}(x)|\leq K_j (1+|x|)^{r_j} e^{-|x|}.
$$

Now we recall the following function to describe the effect of \emph{dispersion} on the solution. Let $c>0$ and
\be\label{varfi}
\varphi(x):=-\frac {Q'(x)}{Q(x)}, \qquad  \varphi_c (x) := -\frac{Q_c'}{Q_c} = \sqrt{c} \varphi(\sqrt{c} x).
\ee
Note that $\varphi$ is an odd function. Moreover, we have

\begin{Cl}[See \cite{MMcol2}]\label{surphi}~
The function $\varphi$ above defined satisfies:
\begin{enumerate}
\item $\lim_{x\to -\infty} \varphi(x)=-1$; $\lim_{x\to +\infty} \varphi(x)=1$.
\item For all $ x\in \R$, we have $|\varphi'(x)|+|\varphi''(x)|+|\varphi^{(3)}(x)|\leq C e^{-|x|}$.
\item Both $\varphi', (1-\varphi^2) \in \mathcal{Y}$.
\end{enumerate} 
\end{Cl}

\begin{rem}
The function $\varphi$ has been already used to describe the main order effect of the collision of two solitons for the quartic KdV equation (see \cite{MMcol1}). In that case, $\varphi $ represented the nonlinear effect on the shift of solitons due to the collision. In this paper, $\varphi$ will describe the dispersive tail behind the soliton product of the interaction with the potential $a_\ve$. For more details, see Lemma \ref{existMP}.
\end{rem}

We finish this section with a preliminary Claim taken from \cite{MMcol1}.

\begin{Cl}[Non trivial kernel, see \cite{MMcol1}]\label{DK}
There exists a unique even solution of the problem
$$
\mathcal L_0 V_0 = mQ^{m-1}, \quad V_0\in \mathcal Y.
$$
Moreover, this solution is given by the formula (cf. Lemma \ref{surL} for the definitions)
$$
V_0(y)=
\begin{cases}
 -\frac 12 \Lambda Q (y) ,  \quad \hbox{ for } m=2, \\
 -Q^2(y),  \quad \hbox{ for } m=3,  \hbox{ and }\\
 \frac 13 [ Q'(y) \int_0^y Q^2 -2Q^3(y)] , \quad  \hbox{ for } m=4 .
\end{cases}
$$
Finally, this solution satisfies $(\mathcal L_0 (1+V_0))'= (1)'=0$.
\end{Cl}

\bigskip

\section{Construction of a soliton-like solution}\label{3}

\subsection{Statement of the result}
Our first effort concerns with the proof of existence of a \emph{pure} soliton-like solution for (\ref{aKdV}) for $t\to -\infty$. Indeed, we prove that, at exponential order in time, there exists a solution $u(t)$ of the form
$$
u(t)\sim_{H^1(\R)} Q(\cdot -(1-\la)t) , \quad t\to -\infty,
$$
and where $Q$ is a soliton for the gKdV equation. 

\begin{thm}[Existence and uniqueness of a pure soliton-like solution]\label{Tm1}~

Suppose $0\leq \la <1$ fixed. There exists $\ve_0>0$ small enough such that the following holds for any $0<\ve < \ve_0$. 

\begin{enumerate}
\item \emph{Existence}. There exists a solution $u \in C(\R, H^1(\R))$ of (\ref{aKdV}) such that 
\be\label{lim0}
\lim_{t\to -\infty} \|u(t) - Q(\cdot -(1-\la)t) \|_{H^1(\R)} =0,
\ee
and energy $E_a[u](t) = (\la -\la_0)M[Q].$
Moreover, there exist constants $K,\ga>0$ such that  for all time $t\leq -\frac 1{2}T_\ve$ and $s\geq 1$,
\be\label{minusTe}
\|u(t) - Q(\cdot -(1-\la)t) \|_{H^s(\R)} \leq  K\ve^{-1} e^{\ve \ga t}.
\ee
In particular, 
\be\label{mTep}
\|u(-T_\ve) - Q(\cdot + (1-\la)T_\ve) \|_{H^1(\R)} \leq K\ve^{-1} e^{- \ga \ve^{-\frac 1{100}}} \leq K \ve^{10},
\ee
provided $0<\ve<\ve_0$ small enough.

\item \emph{Uniqueness}. In addition, this solution is unique in the following cases: $(i)$ $m=3$; and $(ii)$ $m=2,4$ and $0<\la<1$.
\end{enumerate}
\end{thm}

\begin{rem}
This result follows basically from the fact that inside the region $x\leq -\frac 1{2}T_\ve$ the potential $a_\ve$ is constant ($\equiv 1$) at exponential order (see (\ref{ahyp})). In other words, the equation (\ref{aKdV}) behaves asymptotically as a gKdV equation, for which soliton solutions exist globally. 
\end{rem}

\begin{rem}
Note that the energy identity in (1) above follows directly from (\ref{lim0}), Appendix \ref{IdQ} and the energy conservation law from Proposition \ref{Cauchy}.
\end{rem}
\begin{rem}
The uniqueness of $u(t)$ in the general case is an interesting open question. 
\end{rem}

The proof of this Theorem is standard and follows closely \cite{Martel}, where the existence of a unique N-soliton solution for gKdV equations was established. Although there exist possible different proofs of this result, the method employed in \cite{Martel} has the advantage of giving an explicit uniform bound in time (cf. (\ref{minusTe})). This bound is indeed consequence of compactness properties.  For the sake of completeness, we sketch the proof in Appendix \ref{Thm0}. 

\begin{rem}
An easy consequence of the above result is the following. Consider $u(t)$ the solution constructed in Theorem \ref{Tm1}. Then from  the negativity of the energy $E_a$ and the Galiardo-Nirenberg inequality (\ref{GNe}) there exists a constant $K>0$ such that for all time $t\in \R$,
\be\label{Equiv}
\frac 1K \|u(t)\|_{H^1(\R)} \leq  \|u(t)\|_{L^2(\R)} \leq K \|u(t)\|_{H^1(\R)}. 
\ee
Moreover, if $m=3$ or $m=2,4$ and $\la>0$, then we have
\be\label{K}
\sup_{t\in \R} \|u(t)\|_{H^1(\R)} \leq K \|u(-\frac 12T_\ve)\|_{H^1(\R)} .
\ee
This last estimate shows that, in order to understand the limiting behavior at large times of $u(t)$, we may consider only the $L^2$-norm.
\end{rem}

\bigskip

\section{Description of interaction soliton-potential}

Once we have proven the existence (and uniqueness) of a pure soliton-like solution for early times, the next step consists on the study of the interaction soliton-potential. In this sense, note that the region $[-T_\ve, T_\ve]$ can be understood as this nonlinear interaction regime, because of $a_\ve(-T_\ve) \sim 1$ and $a_\ve(T_\ve) \sim 2$ (cf. (\ref{f})-(\ref{ahyp})). 

The next result shows explicitly that perturbations induced by the potential $a_\ve$ are significative, of order one, mainly focused in the scaling and shift parameters. Moreover, the soliton exits the interaction region as a first order solution of the aKdV equation (\ref{aKdV}) with $a_\ve\equiv 2$, and a small error, dispersive term, of order $\ve^{1/2}$ in $H^1(\R)$.

Before state the main result of this section, let us recall $\la_0$ the parameter introduced in Theorem \ref{MT}. These coefficients have a crucial role to distinguish different asymptotic behaviors.

\begin{thm}[Dynamics of the soliton in the interaction region]\label{T0}~

Suppose $0\leq \la \leq \la_0 $. There exist constants $\ve_0>0$, and $ c_\infty(\la) >1$ such that the following holds for any $0<\ve <\ve_0$.
Let $u=u(t)$ be a globally defined $H^1$ solution of (\ref{aKdV}) such that
\be\label{hypINT}
\| u(-T_\ve) - Q(\cdot + (1-\la)T_\ve) \|_{H^1(\R)}\leq K \ve^{1/2}.
\ee
Then there exist $K_0=K_0(K)>0$ and  $\rho(T_\ve), \rho_1(T_\ve) \in \R$ such that
\be\label{INT41}
\|u( T_\ve +\rho_1(T_\ve)) - 2^{-1/(m-1)}Q_{c_\infty}( \cdot - \rho (T_\ve) ) \|_{H^1(\R)} \leq K_0 \ve^{1/2}.
\ee
In addition, $c_\infty(\la =0) =2^p $, $p=\frac 4{m+3}$, and $c_\infty(\la=\la_0)=1$. Finally, we have the bounds
\be\label{INT42}
\abs{\rho_1(T_\ve)} \leq \frac{T_\ve}{100}, \qquad (1-\la)T_\ve \leq \rho(T_\ve) \leq (2c_\infty(\la) -\la-1) T_\ve,   
\ee
valid for $\ve_0$ sufficiently small.
\end{thm}

\begin{rem}
The above theorem is a stability result ensuring that, under the hypotheses of Theorem \ref{MT}, the soliton survives the interaction, with the scaling predicted by the conservation of energy. 
\end{rem}

\begin{rem}
Even if from Theorem \ref{Tm1} we have an exponential decay on the error term at time $t=-T_\ve$ (cf. (\ref{mTep}) and (\ref{hypINT})), we are unable to get a better estimate on the solution at time $t=T_\ve$. This problem is due to the emergency of some order $\ve^{1/2}$ dispersive terms, hard to describe using soliton based functions. This new phenomenon has high similarity with a recent description obtained by Martel and Merle for the collision of two solitons of similar sizes for the BBM and KdV equations, see \cite{MMfin}.      
\end{rem}
\begin{rem}\label{4.3}
We do not know whether the above result is still valid in the range $\la>\la_0$. Formal computations suggest that in this regime the soliton might be \emph{reflected} after the interaction. We hope to consider this regime in a forthcoming publication. 
\end{rem}

The proof of this Theorem requires several steps, in particular this Section and Section \ref{sec:3} deal with the proof of this result. As we have mentioned in the introduction of this paper, we will construct an approximate solution of (\ref{aKdV}). In the next section we prove that the actual solution describing the interaction of the soliton and the potential $a_\ve$ is sufficiently close to our approximate solution.

\subsection{Construction of an approximate solution describing the interaction}\label{sec:2}

Let us remark that, after the time $-T_\ve$, the interaction begins to be nontrivial and must be considered in our computations. The objective of the following sections is to construct an approximate solution of (\ref{aKdV}), which describes the first order interaction between the soliton and the potential on the interval of time $[-T_\ve, T_\ve]$. The final conclusion of this construction is presented in Proposition \ref{CV} below.

Our first step towards the proof of Proposition \ref{CV} is the introduction of a suitable notation.  

\subsection{Decomposition of the approximate solution}\label{sec:2-1}

We look for $\tilde u(t,x)$, the approximate solution for (\ref{gKdV0}),  carring out a specific structure. In particular, we construct $\tilde u$ as a suitable modulation of the soliton $Q(x-(1-\la)t)$, solution of the KdV equation
\be\label{orig}
u_t +(u_{xx} -\la u + u^m)_x =0.
\ee 
Let $c=c(\ve t)\geq 1$ be a bounded function to be chosen later and 
\be\label{defALPHA}
    y:=x-\rho(t) \quad \hbox{and} \quad     R(t,x): =\frac {Q_{c(\ve t)}(y)}{\tilde a(\ve \rho(t))},
\ee
where 
\be\label{param0}
\tilde a (s) := a^{\frac 1{m-1}}(s), \quad    \rho (t) := -(1-\la) T_\ve +  \int_{-T_\ve}^t (c(\ve s)-\la)ds. 
\ee

The parameter $\tilde a$ intends to describe the shape variation of the soliton along the interaction.

The form of $\tilde u(t,x)$ will be the sum of the soliton plus a correction term:
\begin{equation}\label{defv} 
    \tilde u(t,x) :=R(t,x)+w(t,x),
\end{equation}
\begin{equation}\label{defW}
    w(t,x):= \ve A_{c} (\ve t; y),
\end{equation}
where $A_c := A_{c(\ve t)}(\ve t; y) = c^{\frac1{m-1}}A(\ve t; \sqrt{c} y) $ and $A$ is a unknown function to be determined.

We want to measure the size of the error produced by inserting $\tilde u$ as defined in (\ref{defW}) in the equation (\ref{gKdV0}). For this, let 
\be\label{2.2bis}
S[\tilde u](t,x) := \tilde u_t + (\tilde u_{xx} -\la \tilde u +a_\ve \tilde u^{m})_x.
\ee

Finally, let us recall the definition of the linear operator $\mathcal L$ given in (\ref{defLy}). Our first result is the following

\begin{prop}[First decomposition of $S(\tilde u)$]\label{prop:decomp}~

For every $t\in [-T_\ve, T_\ve]$, the following nonlinear decomposition of the error term $S[\tilde u]$ holds:
$$
S[\tilde u](t,x)  =  \ve [F_1  - (\mathcal L A_{c})_y ](\ve t; y)+ \ve^2 [(A_c)_t + c'(\ve t)\Lambda A_c](\ve t; y) +  \ve^2 \mathcal E(t,x), 
$$
where $\Lambda A_c(y):= \frac 1c (\frac 1{m-1} A_c(y) +\frac 12 y (A_c)_y(y))$ (cf. Lemma \ref{surL})  and 
\bea\label{F1}
F_1(\ve t; y)  :=   \frac{ c'(\ve t)}{\tilde a(\ve \rho(t))}\Lambda Q_c(y) + \frac{ a' (\ve \rho(t))}{\tilde a^m(\ve \rho(t))} \big[ -\frac 1{m-1} (c(\ve t)-\la) Q_c(y) +  (yQ_c^m(y))_y\big],
\eea
and $\mathcal E(t,x)$ is a bounded function in $[-T_\ve, T_\ve]\times \R$. 

\end{prop}
\begin{proof}
We prove this result in Appendix \ref{AppA}. 
\end{proof}

Note that if we want to improve the approximation $\tilde u$, the unknown function $A_c$ must be chosen such that 
$$
 (\Omega)  
 \qquad
 (\mathcal L A_{c})_y(\ve t; y) = F_1(\ve t; y),  \quad \hbox{ for all  } y\in \R. 
$$
Then the error term will be reduced to the second order quantity $S[\tilde u] =\ve^2 [(A_c)_t + c'(\ve t)\Lambda A_c](\ve t; y)  + \ve^2 \mathcal E(t,x)$. We prove such a solvability  result in a new section, of independent interest.

\subsection{Resolution of $(\Omega)$ }\label{5}

When solving problem $(\Omega)$, we will see below that it is not always possible to find a solution of finite mass. In fact, we will look for solutions such that time and space variables are separated:
\be\label{eq:st}
A_c(t,y) = b(\ve t) \varphi_c(y) + d(\ve t) + h(\ve t) \hat A_c(y);
\ee
where $b(s), d(s)$ and $ h(s)$ are exponentially decreasing in $s$, $\varphi_c$ is the bounded function defined in (\ref{varfi}) and $ \hat A_c\in \mathcal Y$ (recall that $\lim_{\pm \infty} \varphi_c = \pm \sqrt{c}$.)

This choice gives us a crucial property. Recall that $c\geq 1$. We say that $A_c$ satisfies the {\bf (IP)} property ({\bf IP} = important property) if and only if
$$
{\bf(IP)} \begin{cases} \hbox{ Any spatial derivative of $A_c(\ve t, \cdot) $ is a localized $\mathcal{Y}$-function,} \\
\hbox{ and there exists $K,\ga>0$ such that $\| A_c(\ve t, \cdot) \|_{L^\infty(\R)} \leq K e^{-\ga\ve|t|}  $ for all $t\in \R$.} \end{cases}
$$
Note that a solution of the form (\ref{eq:st}) satisfies the {\bf (IP)} property. 

\subsubsection{Resolution of a time independent model problem}

In this subsection we address the following existence problem. Let us recall from (\ref{defLy}), $\mathcal L_0 := -\partial_{yy}^2 + 1 -m Q^{m-1}(y).$

Given a bounded and even function $F = F(y)$, we look for a bounded solution $A=A(y)$ of the following model problem
\be\label{MP}
(\mathcal L_0 A)' = F.
\ee
satisfying $A$ bounded. The following result deals with the solvability theory for problem (\ref{MP}), in the same spirit that Proposition 2.1 in \cite{MMcol1} and Proposition 3.2 in \cite{Mu}. 

\begin{lem}[Existence theory for (\ref{MP})]\label{existMP}~

Suppose $F\in \mathcal Y$ even and satisfying the orthogonality condition
\be\label{OC1}
\int_\R F Q = 0.
\ee
Let $\beta = \frac12 \int_\R F$. For any $\delta \in \R$, problem (\ref{MP}) has a bounded solution $A$ of the form
\be\label{OC3}
A (y) =  \beta \varphi (y)+ \delta 
+ A_1 (y), \quad \hbox{ with }  A_1(y)\in \mathcal Y.
\ee
Finally, this solution is unique in $L^2(\R)$ modulo the addition of a constant times $Q'$.
\end{lem}

\begin{proof}
Let us write $A :=  \beta \varphi  + \delta (1+V_0) + A_1$, where $\beta,\delta \in \R$ and $A_1 \in \mathcal Y$ are to be determined. Inserting this decomposition in (\ref{MP}), we have $(\mathcal L_0 A_1)' = F -\beta (\mathcal L_0 \varphi)',$ namely
\be\label{hatA}
\mathcal L_0 A_1 = H -\beta \mathcal L_0 \varphi  + \gamma, \quad H(y) := \int_{-\infty}^y F(s) ds,
\ee
and where $\gamma:=  \mathcal L_0 A_1(0) -\int_{-\infty}^0 H(s)ds$. Without loss of generality we can suppose the constant term $\gamma = -\beta$, because from Claim \ref{DK} $\mathcal L_0 (1+V_0) =1,$
thus any constant term can be associated to the free parameter $\delta$.

Now, from Lemma \ref{surL} the problem (\ref{MP}) is solvable if and only if 
$$
\int_\R (H -\beta (\mathcal L_0 \varphi +1))Q' =\int_\R HQ' =\int_\R FQ=0,
$$
namely (\ref{OC1}) (recall that $\mathcal L_0 Q' =0$.) Thus there exists a solution $A_1$ of (\ref{hatA}) satisfying $\int_\R  A_1 Q' =0.$ Moreover, since
$$
\lim_{y\to - \infty} (H -\beta (\mathcal L_0 \varphi +1) )(y) = 0, \qquad \lim_{y\to + \infty} (H -\beta (\mathcal L_0 \varphi +1) )(y) = \int_\R F - 2\beta,
$$
we get $A_1 \in \mathcal Y$ provided $\beta = \frac 12\int_\R F$, by Lemma \ref{surL}. This finishes the proof.
\end{proof}

\subsubsection{Existence of dynamical parameters}

Our first result concerns to the existence of a dynamical system involving the evolution of first order scaling and translation parameters on the main interaction region. This system is related to the orthogonality condition $\int_\R F_1Q_c=0$, see proof of Lemma \ref{lem:omega}.

\begin{lem}[Existence of dynamical parameters]\label{ODE}~

Suppose $m=2,3$ or $4$. Let $\la_0, p, a(\cdot) $ be as in Theorem \ref{T0} and (\ref{ahyp}). There exists a unique solution $(\rho, c)$, with $c$ bounded positive, monotone, defined for all $t\geq -T_\ve$, with the same regularity than $a(\ve \cdot)$, of the following system  
\be\label{c}
\begin{cases}
c'(\ve t) = p \ c(\ve t) \big[ c(\ve t) - \frac \la{\la_0 }\big] \frac{a'}{a} (\ve \rho(t)), \qquad c(-\ve T_\ve) = 1, \\
\rho'(t) = c(\ve t) -\la, \qquad \rho(-T_\ve) =-(1-\la)T_\ve.
\end{cases}
\ee
In addition, 
\begin{enumerate}
\item If $\la=\la_0 $, one has $c\equiv 1$.
\item If $0\leq \la <\la_0 $ then for all $t\geq -T_\ve $ one has $c(\ve t)> 1$ and $\lim_{t\to +\infty} c(\ve t) = c_\infty + O(\ve^{10})$, where $c_\infty=c_\infty(\la)>1$ is the unique solution of the following algebraic equation
\be\label{cinf}
c_\infty^{\la_0 } (c_\infty   -\frac \la{\la_0 } )^{1-\la_0 } = 2^p (1 - \frac \la{\la_0 })^{1-\la_0 }, \quad c_\infty>1.
\ee
Moreover, $ \la \in [0,\la_0] \mapsto c_\infty(\la)\geq 1$ is a smooth decreasing application, and $c_\infty(\la =0) = 2^{p}$.
\end{enumerate}
\end{lem}

\begin{rem}[Case $\la=0$] In this situation, there exists a simple implicit expression for $c(\ve t)$:
$$
\rho'(t) =c(\ve t) = \frac{a^p(\ve \rho(t)) }{a^p(-\ve T_\ve)}.
$$ 
Using the strict monotony of $a$, from this identity we can find explicitly $c(\ve t)$.
\end{rem}

\begin{rem}
Note that the critical value $\la_0 $ can be seen as the exact value of $\la$ such that the solution $u(t)$ constructed in Theorem \ref{Tm1} has zero energy. Indeed, note that from Theorem \ref{Tm1} we have $E_a[u] = (\la-\la_0 ) M[Q].$ This implies that $E_a[u] =0$ $(>0, <0$ resp.) if $\la=\la_0 $ ($\la>\la_0 , \la<\la_0   $ resp.). Because of this phenomenon the study of the soliton dynamics for $\la>\la_0$ is an open question.
\end{rem}

\begin{proof}[Proof of Lemma \ref{ODE}]
The local existence of a solution $(c, \rho)$ of (\ref{c}) is a direct consequence of the Cauchy-Lipschitz-Picard theorem. 

Now we use (\ref{c}) to prove a priori estimates on the solution $c$. Note that 
$$
\frac{(c(\ve t)-\la)}{c(\ve t)( c(\ve t) -\frac \la{\la_0 })}c'(\ve t) =p(c(\ve t)-\la) \frac{a'}{a}(\ve \rho) =p\rho'(t) \frac{a'}{a}(\ve \rho).
$$
In particular,
$$
(1-\la_0 )\partial_t \log ( c(\ve t) -\frac \la{\la_0 })  + \la_0  \partial_t \log c(\ve t) = p \partial_t  \log a(\ve \rho).
$$
By integration on $[-T_\ve, t]$, using $c(-\ve T_\ve) =1$, we obtain
\be\label{boundC}
c^{\la_0 }(\ve t)  ( c(\ve t)-\frac \la{\la_0 } )^{1-\la_0 }   = ( 1-\frac \la{\la_0 } )^{1-\la_0 } \frac{a^p(\ve \rho(t))}{a^p(-\ve(1-\la)T_\ve)}.
\ee
Since $1\leq a \leq 2$, $c$ is bounded and $\rho$ is bounded on compact sets and consequently we obtain the global existence. One proves in particular $c'>0$ and 
\be\label{boundCp}
c^{\la_0 }(\ve t) < a^p(\ve \rho), \quad \hbox{ and thus } \quad 1\leq c(\ve t)\leq 2^{\frac 4{5-m}}.
\ee
 Moreover, $\lim_{t\to +\infty } c(\ve t)$ exists and satisfies $\lim_{t\to +\infty } c(\ve t) = c_\infty + O(\ve^{10})$, where $c_\infty$ is a solution of  (\ref{cinf}), after passing to the limit in (\ref{boundC}). In order to prove the uniqueness of the solution of (\ref{cinf}), consider for $\mu\geq 1$ the smooth function
$$
g(\mu; \la) := \mu^{\la_0 } (\mu   -\frac \la{\la_0} )^{1-\la_0} - 2^p (1 - \frac \la{\la_0 })^{1-\la_0}.
$$ 
Note that in the case $\la<\la_0$ we have $g(1;\la)<0$ and
$$
\partial_\mu g(\mu;\la) =\mu^{\la_0-1}(\mu   -\frac \la{\la_0} )^{-\la_0}  (\mu -\la) \geq (1-\frac \la{\la_0})^{-\la_0} >0.
$$
This implies that there exists a unique $c_\infty(\la)>1$ such that $g(c_\infty(\la); \la) =0$. This proves the uniqueness. The smoothness of the application $\la\in [0,\la_0] \mapsto c_\infty(\la)$ is an easy consequence of the Implicit Function Theorem.

Finally we prove that $\la \mapsto c_\infty(\la)$ is a decreasing map. To do this, we take derivative in (\ref{cinf}). We obtain
\bee
\frac{c_\infty(\la)^{\la_0-1} (c_\infty(\la) -\la)}{(c_\infty(\la) -\frac \la{\la_0})^{\la_0}}c_\infty '(\la) & = & (\frac 1{\la_0}- 1)\Big[ \frac{c_\infty^{\la_0}(\la)}{(c_\infty(\la) -\frac\la{\la_0})^{\la_0}}  -\frac{2^p}{(1-\frac\la{\la_0})^{\la_0}} \Big] \\
& \leq &  (\frac 1{\la_0}- 1)(1 -\frac\la{\la_0})^{-\la_0} (1-2^p)<0.
\eee
\end{proof}


\subsubsection{Conclusion of resolution of $(\Omega)$}

\begin{lem}[Resolution of $(\Omega)$]\label{lem:omega}~

Suppose $0\leq \la \leq \la_0$ and $c(\ve t)$ given by (\ref{c}). There exists a solution $A_c= A_c(\ve t; y)$ of
\be\label{A10}
(\mathcal{L}A_c )'(\ve t; y) = F_1(\ve t ; y), 
\ee
satisfying {\bf (IP)} and such that
\begin{enumerate}
\item For every $t\in [-T_\ve, T_\ve]$,
\be\label{A}
\begin{cases}
A_c(\ve t ; \cdot ) \in L^\infty(\R), \quad  A_c (\ve t; y)=b(\ve t) (\varphi_c(y)- c^{1/2}) + h(\ve t)\hat A_c(y), \\
 \hat A_c \in \mathcal Y, \quad |b(\ve t)| + |h(\ve t)| \leq Ke^{-\ga \ve |t|}.
\end{cases}
\ee
\item  $\lim_{y\to +\infty} A_c(y) =0$.
\end{enumerate}
\end{lem}

\begin{rem}
The function $A_c$ models, at first order in $\ve$, the shelf-like tail behind the soliton, a dispersive effect of the interaction soliton-potential. 
\end{rem}

\begin{proof}
We prove this lemma in three steps.

{\bf Step 1}. \emph{Reduction to a time independent problem}. We suppose $c$ given as in Lemma \ref{ODE}. Note that $F_1$ in (\ref{F1}) can be written as follows
$$
F_1(\ve t; y) = \frac{a'}{\tilde a^m} \big[ p c(c-\frac \la{\la_0 }) \Lambda Q_c  - \frac 1{m-1}(c-\la) Q_c + (yQ_c^m)' \big](y).
$$
Consider now the functions
$$
\tilde F_1 (y) :=  p  \Lambda Q  - \frac 1{m-1}Q +  (yQ^m)'  ; \quad     \hat F_1(y) :=  \frac 1{m-1}Q -\frac{p}{\la_0} \Lambda Q  = \frac 1{m-1} Q -\frac 4{5-m}\Lambda Q .
$$
We claim that if $c(\ve t)$ satisfies (\ref{c}) then every term in $F_1$ has the correct scaling, as shows the following result.

\begin{Cl}\label{Equiv1}
Suppose $\tilde A(y), \hat A(y)$ solve the stationary problems 
\be\label{Sta}
(\mathcal L_0 \tilde A )' = \tilde F_1, \quad (\mathcal L_0 \hat A )' = \hat F_1.
\ee
 Then for all $t\in \R$,
$$
A_c(\ve t; y) := \frac{a' (\ve \rho )}{\tilde a^m (\ve \rho)} c^{\frac 1{m-1}-\frac 12}(\ve t)  \Big[ \tilde A + \la c^{-1}(\ve t) \hat A \Big](c^{1/2}(\ve t) y)
$$ 
is a solution of (\ref{A10}).
\end{Cl}
\begin{proof}
Note that 
\bee
(\mathcal L A_c)' & = &  \frac{a' (\ve \rho )}{\tilde a^m (\ve \rho)} c^{\frac 1{m-1} + 1} \big[ -\tilde A'' + \tilde A - m Q^{m-1}\tilde A \big]'( c^{1/2}y) \\ 
& & \qquad \qquad + \la \frac{a' (\ve \rho )}{\tilde a^m (\ve \rho)} c^{\frac 1{m-1}} \big[ -\hat A'' + \hat A - m Q^{m-1}\hat A \big]'( c^{1/2}y) \\
& =& \frac{a' (\ve \rho )}{\tilde a^m (\ve \rho)} c^{\frac 1{m-1} +1} \tilde F_1( c^{1/2} y)  + \la \frac{a' (\ve \rho )}{\tilde a^m (\ve \rho)} c^{\frac{1}{m-1}} \hat F_1( c^{1/2} y) \\
& =& \frac{a' (\ve \rho )}{\tilde a^m (\ve \rho)} [ p c^2 \Lambda Q_c  - \frac 1{m-1}cQ_c + (yQ_c^m)' ] + \la \frac{a' (\ve \rho )}{\tilde a^m (\ve \rho)} [  \frac 1{m-1} Q_c -  \frac p{\la_0}  c \Lambda Q_c ] \\
& =& F_1(\ve t; y).
\eee
This finishes the proof. 
\end{proof}

The above Claim reduces to time independent problems.

\noindent
{\bf Step 2}. \emph{Resolution of (\ref{Sta})}.

\begin{Cl}
There exists $\tilde A, \hat A$ solutions of (\ref{Sta}) satisfying (\ref{OC3}).
\end{Cl}

\begin{proof}
According to Lemma \ref{existMP} it suffices to verify the orthogonality conditions
$$
\int_\R \tilde F_1 Q =\int_\R \hat F_1 Q =0.
$$
Indeed, using Lemma \ref{IdQ} in Appendix \ref{AidQ}
\bee
\int_\R \tilde F_1 Q & = & p \int_\R \Lambda Q Q -\frac 1{m-1} \int_\R Q^2 +\int_\R Q(yQ^m)_y \\
& = & p  \int_\R \Lambda Q Q -\frac 1{m-1} \int_\R Q^2   +  \frac 1{m+1}\int_\R Q^{m+1} \\
&  = & \frac{(5-m)}{4(m-1)} [ p -\frac{4}{m+3} ]\int_\R Q^2 =0.
\eee
Similarly
\bee
\int_\R \hat F_1 Q & = & -\frac 4{5-m} \int_\R \Lambda Q Q + \frac 1{m-1} \int_\R Q^2 \\
& = & -\frac 4{5-m} \times \frac{5-m}{4(m-1)}\int_\R Q^2 + \frac 1{m-1} \int_\R Q^2 =0.
\eee
Thus, by invoking Lemma \ref{existMP} there exist solutions $\tilde A$, $\hat A$ of (\ref{Sta}) of the form
$$
\begin{cases}
\tilde A(y) = \tilde \beta \varphi (y) + \tilde \delta + \tilde A_1(y), \quad  \tilde A_1\in \mathcal Y, \\
\hat A(y) = \hat \beta \varphi (y) + \hat \delta + \hat A_1(y), \quad  \hat A_1\in \mathcal Y,
\end{cases}
$$
and where $\tilde \beta, \hat \beta , \tilde \delta, \hat \delta \in \R$. Moreover, $\tilde \beta, \hat \beta $ are given by the formulae
\bee
\tilde \beta & := & \frac 12\int_\R \tilde F_1 = \frac 12 \int_\R (p\Lambda Q -\frac 1{m-1} Q ) =   \frac 12[ p(\frac 1{m-1} -\frac 12) -\frac 1{m-1}] \int_\R Q \\
& =&  -\frac{3}{2(m+3)}\int_\R Q < 0,
\eee
 for each $m=2,3$ and $4$. On the other hand
\bee
\hat \beta & := & \frac 12\int_\R \hat F_1 = \frac 12 \int_R (\frac 1{m-1} Q  -\frac 4{5-m}\Lambda Q )\\
&  = & \frac 12[ \frac 1{m-1} -\frac 4{5-m} \times \frac{3-m}{2(m-1)}] \int_\R Q  =  \frac{1}{2(5-m)}\int_\R Q > 0,
\eee
 for each $m=2,3$ and $4$. 
\end{proof}

{\bf Final step}. Finally, to get $\lim_{y\to +\infty } \tilde A(y) = \lim_{y\to +\infty } \hat A(y)  =0$ we choose $\tilde \delta = -\tilde \beta$ and $  \hat \delta =-\hat \beta$. This proves the last part of the lemma. With this choice we have 
$$
\tilde A(y) = \tilde \beta (\varphi (y) - 1) + \tilde A_1(y), \quad \hat A(y) = \hat \beta (\varphi (y) - 1) + \hat A_1(y), \quad  \tilde A_1, \hat A_1\in \mathcal Y.
$$
Using Claim \ref{Equiv1}, an actual solution $A_c(\ve t; y)$ of (\ref{A10}) is obtained by considering 
\bee
A_c(\ve t; y) & := & \frac{a' (\ve \rho )}{\tilde a^m (\ve \rho)} c^{\frac 1{m-1}}(\ve t) \big[ \tilde A + \la c^{-1}(\ve t) \hat A \big] (c^{1/2} y) \\
& = :& b(\ve t)  (\varphi_c (y) - c^{1/2}) + h(\ve t)\hat A_c(y), \quad  \hat A_c\in \mathcal Y , .
\eee
where 
$$
b(\ve t) :=  \frac{ a' (\ve \rho )c^{\frac 1{m-1} - 1} }{\tilde a^m (\ve \rho)} (\tilde \beta + \la c^{-1}(\ve t) \hat \beta), \quad h(\ve t) :=  \frac{a' (\ve \rho )}{\tilde a^m (\ve \rho)}.
$$
This finishes the proof of Lemma \ref{lem:omega}.
\end{proof}

\begin{rem}
We emphasize that \textbf{in any case} $A_c \in L^2(\R)$, even if it is exponentially decreasing in time. This non summable solution must be modified in order to obtain a finite mass solution.
\end{rem}
Before continuing with the construction of the approximate solution, we need some crucial estimates on the parameter $c(\ve t)$. 

\begin{rem}[Bounds for $c(\ve t)$]\label{boundsC}
From the bound on $c(\ve t)$ in (\ref{boundC}) we conclude that for all $t\in [-T_\ve,T_\ve]$
$$
1\leq c(\ve t) \leq 2^{\frac 4{5-m}}.
$$
\end{rem}

\subsection{Correction to the solution of Problem $(\Omega)$}

Consider the cutoff function $\eta \in C^\infty (\R)$ satisfying the following properties:
\be\label{eta}
\begin{cases}
0\leq \eta (s) \leq 1, \quad 0\leq  \eta' (s) \leq 1, \; \hbox{ for any } s\in \R;\\
\eta(s)\equiv 0 \; \hbox{ for } s\leq -1, \quad  \eta(s)\equiv 1 \; \hbox{ for } s\geq 1.
\end{cases}
\ee 
Define 
\be\label{etac}
\eta_\ve (y) := \eta( \ve y +  2 ),
\ee
and for $A_c=A_c(\ve t; y)$ \emph{solution} of (\ref{A10}) constructed in Lemma \ref{lem:omega}, denote
\be\label{Agato}
A_\#(\ve t; y) := \eta_\ve (y)A_c(\ve t; y).
\ee

Now redefine 
\be\label{hatu}
\tilde u := R+ w = R + \ve A_\#.
\ee
where $R$ is the modulated soliton from (\ref{defALPHA}). 
%


%

The following Proposition, which deals with the error associated to this cut-off function and the new approximate solution $\tilde u$, is the principal result of this section.

\begin{prop}[Construction of an approximate solution for (\ref{aKdV})]\label{CV}~

There exist constants $\ve_0, K>0$ such that for all $0<\ve <\ve_0$ the following holds.

\begin{enumerate}

\item Consider the localized function $A_\#$ defined in (\ref{etac})-(\ref{Agato}). Then we have

\begin{enumerate}
\item \emph{New behavior}. For all $t\in [-T_\ve, T_\ve]$,
\be\label{newb}
\begin{cases}
& A_\#(\ve t, y ) = 0 \quad \hbox{ for all } y\leq -\frac 3\ve, \\
& A_\#(\ve t, y ) =A_c(\ve t, y) \quad \hbox{ for all } y\geq -\frac 1\ve. 
\end{cases}
\ee
\item \emph{Integrable solution}. For all $t\in [-T_\ve, T_\ve]$, $A_\#(\ve t, \cdot ) \in H^1(\R)$ with
\be\label{H1}
\|\ve A_\#(\ve t, \cdot ) \|_{H^1(\R)} \leq K \ve^{\frac12}e^{-\ga \ve |t|}.
\ee
\end{enumerate}

\item The error associated to the new function $\tilde u$ satisfies  
\be\label{SH2}
\| S[\tilde u](t) \|_{H^2(\R)} \leq K \ve^{\frac 32}e^{-\ga \ve |t|},
\ee
and the following integral estimate holds
$$
\int_\R\| S[\tilde u](t) \|_{H^2(\R)} dt \leq K\ve^{1/2}.
$$
\end{enumerate}

\end{prop}

\begin{proof}
The proof of (\ref{newb}) is direct from the definition. To prove (\ref{H1}) it is enough to recall that
$$
\|\eta'_c\|_{L^2(\R)} \leq K \ve^{-1/2}.
$$
For the proof of (\ref{SH2}), see Appendix \ref{CV1}.
\end{proof}

\subsection{Recomposition of the solution}

In this subsection we present some important estimates concerning our approximate solution. More precisely, we will show that $\tilde u$ at time $\pm T_\ve$ behaves as a modulated soliton with the scaling given by the formal computations at infinity.  We start out with some model $H^1$-estimates.

\begin{lem}[First estimates on $\tilde u$]~

\begin{enumerate}
\item Decay away from zero. Suppose $f=f(y)\in \mathcal Y$. Then there exist $K,\ga>0$ constants such that for all $t\in [-T_\ve, T_\ve]$
\be\label{Est1}
\norm{a'(\ve x) f(y)}_{H^1(\R)} \leq K e^{-\ga\ve|t|}.
\ee
\item Almost soliton solution. The following estimates hold for all $t\in [-T_\ve, T_\ve]$.
\be\label{Est2}
\norm{\tilde u_t + (c-\la)\tilde u_x}_{H^1(\R)} \leq K \ve e^{-\ga\ve |t|},  \quad \norm{\tilde u_t + (c-\la)\tilde u_x}_{L^\infty(\R)} \leq K \ve e^{-\ga\ve |t|},
\ee
\be\label{Est2a}
\tilde u_{xx} -\la \tilde u + a_\ve \tilde u^m = (c-\la) \tilde u + O_{L^2(\R)}(\ve e^{-\ga\ve |t|}),
\ee
and
\be\label{Est20}
\| (\tilde u_{xx} -c\tilde u +  a_\ve\tilde u^m)_x \|_{H^1(\R)} \leq K \ve e^{-\ga\ve |t|} +K\ve^2.
\ee
\end{enumerate}
\end{lem}

\begin{proof}
The proof of (\ref{Est1}) is a direct consequence of (\ref{ahyp}) and the fact that $\rho'(t) =c(\ve t) -\la \geq 1-\la$, for all $t\in \R$.

Now let us prove (\ref{Est2}). From (\ref{hatu}) we obtain
\bee
\tilde u_t + (c-\la)\tilde u_x & = & \ve \frac {c'}{\tilde a} \Lambda Q_c  - \ve\frac {\tilde a'}{\tilde a^2} (c-\la)Q_c + \ve[ (A_\#)_t + c(A_\#)_x]\\
& =& \ve[ (A_\#)_t + c(A_\#)_x] + O_{H^1(\R)}(\ve e^{-\ga\ve |t|}).
\eee
Now, from (\ref{Gato1}) in Appendix \ref{CV1}, we know that
$$
\ve[ (A_\#)_t + c(A_\#)_x] = \ve^2 (c-\la) \eta_c' A_c + O_{H^1(\R)}(\ve^{\frac 32}e^{-\ga\ve|t|} ) =O_{H^1(\R)}(\ve^{\frac 32}e^{-\ga\ve|t|} ).
$$ 
This estimate completes the proof of the $H^1$-estimate. The $L^\infty$-estimate follows directly from the continuous Sobolev embedding $H^1(\R)\hookrightarrow L^\infty(R)$. 

Concerning (\ref{Est2a}), note that from (\ref{H1})
\bee
\tilde u_{xx} -\la \tilde u + a_\ve \tilde u^m & = & ( c-\la) \tilde u  + \ve [(A_\#)_{xx} + m a_\ve R^{m-1} A_\#] \\
& & \qquad \qquad + O_{L^2(\R)}(\ve e^{-\ga\ve |t|})+ O(\ve^2|A_\#|^2) \\
& =& (c-\la)\tilde u  + O_{L^2(\R)}(\ve e^{-\ga\ve |t|}).
\eee

Finally we deal with (\ref{Est20}). Note that $ [\tilde u_{xx} -c\tilde u +a_\ve \tilde u^m]_x  =  S[\tilde u] -( (c-\la)\tilde u_x + \tilde u_t );$ the conclusion follows directly from (\ref{SH2}) and (\ref{Est2}).
\end{proof}

The next result describes the behavior of the almost solution $\tilde u$ at the endpoints $t=-T_\ve, T_\ve$.

\begin{prop}[Behavior at $t = \pm T_\ve$]\label{atpmT}~

There exist constants $K,\ve_0>0$ such that for every $0<\ve <\ve_0$ the approximate solution $\tilde u$ constructed in Proposition \ref{CV} satisfies
\begin{enumerate}
\item Closeness to $Q$ at time $t=-T_\ve$. 
\be\label{mTe}
\| \tilde u (-T_\ve) - Q(\cdot + (1-\la)T_\ve ) \|_{H^1(\R)} \leq K \ve^{10}.
\ee
\item Closeness to $2^{-1/(m-1)} Q_{c_\infty}$ at time $t=T_\ve$. Let $c_\infty(\la)>1$ be as defined in Lemma \ref{ODE}. Then
\be\label{pTe}
\| \tilde u (T_\ve) - 2^{-1/(m-1)} Q_{c_\infty} (\cdot - \rho(T_\ve)) \|_{H^1(\R)} \leq K \ve^{10}.
\ee
\end{enumerate}
\end{prop}

\begin{proof}
By definition,
$$
\tilde u (-T_\ve) - Q(\cdot -\rho(-T_\ve)) = R(-T_\ve) - Q(\cdot +(1-\la)T_\ve) + w(-T_\ve).
$$
From Lemma \ref{CV} we have
$$
\| w(\pm T_\ve) \|_{H^1(\R)} = \| \ve A_\#(\pm T_\ve) \|_{H^1(\R)}  \leq K \ve^{1/2} e^{-\ga \ve^{-\frac 1{100}}} \leq K \ve ^{10},
$$
for $\ve$ small enough. On the other hand, from $\rho (-T_\ve ) = - (1-\la)T_\ve$ and using the monotony of $a$, we have
$$
1 \leq c(-\ve T_\ve) \leq  a^{\frac 4{5-m}}(\ve \rho(-T_\ve)) \leq 1 +\ve^{10}.
$$
In conclusion we have
$$
\| R(-T_\ve) -Q(\cdot + (1-\la)T_\ve) \|_{H^1(\R)} \leq K \ve^{10}, 
$$
as desired. Estimate (\ref{pTe}) is totally analogous, and we skip the details. 
\end{proof}

In concluding this section, we have constructed and approximate solution $\tilde u$ describing, at least formally, the interaction soliton-potential. In the next section we will show that the \emph{solution} $u$ constructed in Theorem \ref{Tm1} actually behaves like $\tilde u$ inside the interaction box $[-T_\ve, T_\ve]$.


\bigskip

\section{First stability results}\label{sec:3}

In this section our objective is to prove that the approximate solution $\tilde u$ describes the actual dynamics of interaction in the interval $[-T_\ve, T_\ve]$. The next proposition is the principal result of this section.

\begin{prop}[Exact solution close to the approximate solution $\tilde u$]\label{prop:I}~

Let $\kappa> \frac 1{100}$. There exists $\ve_0>0$ such that the following holds for any $0<\ve <\ve_0$.
Suppose that 
\be\label{INTkl}
 \left\| S[\tilde u](t)\right\|_{H^2(\R)} \leq K\ve^{1+\kappa}e^{-\ga\ve|t|} ,\quad  \int_\R \left\| S[\tilde u](t)\right\|_{H^2(\R)} dt \leq K\ve^{\kappa},    
\ee
and
\be\label{hypINTa}
\| u(-T_\ve) - \tilde u(-T_\ve) \|_{H^1(\R)}\leq K \ve^{\kappa },
\ee
with  $u=u(t)$ a $H^1(\R)$ solution of (\ref{aKdV}) in a vicinity of $t= -T_\ve$. Then $u(t)$ is defined for any $t\in [-T_\ve, T_\ve]$ and there exist $K_0=K_0(\kappa ,K)$ and a $C^1$-function $\rho_1:[-T_\ve,T_\ve ]\rightarrow \R$ such that, for all $t\in [-T_\ve,T_\ve]$,
\be\label{INT41a}
\|u(t+\rho_1(t))-\tilde u(t) \|_{H^1(\R)} \leq K_0 \ve^{\kappa },\quad |\rho_1'(t)|\leq K_0 \ve^{\kappa }.
\ee
\end{prop}

Before the proof, we clarify some important details about the statement of the proposition.

\begin{rem}
Note that $u$ has to be modulated in order to get the correct result. However, in this case we have not modulated on the scaling and spatial translation parameters because (\ref{aKdV}) is not invariant under these transformations. Nevertheless, we still have another degeneracy, due to time translations, which fortunately allows to control the dynamics of the solution $u$ for every $t\in [-T_\ve, T_\ve]$. In this sense, the \emph{new time} $s(t) := t+ \rho_1(t)$ can be interpreted as a \emph{retarded (or advanced)} time of the actual solution with respect to the approximate solution. Moreover, note that  for $\ve$ small enough,
$$
s'(t) = 1 + \rho'(t) >\frac {99}{100}>0,
$$ 
for all $t\in [-T_\ve, T_\ve]$. This means that we can inverse $s(t)$ on $s([-T_\ve, T_\ve]) \subseteq \frac {99}{100}[-T_\ve, T_\ve]$. 

From the proof we do not know the sign of $\rho_1'(t)$, so in particular we do not know if the solution $u$ is retarded or in advance with respect to the approximate solution $\tilde u$. 
\end{rem}

\begin{proof}[Proof of Proposition \ref{prop:I}]

Let $K^*>1$ be a constant to be fixed later. Let us recall that from Proposition \ref{GWP0} we have that $u(t)$ is globally well-defined in $H^1(\R)$. Since $\|u(-T_\ve)-\tilde u(-T_\ve)\|_{H^1(\R)}\leq K \ve^\kappa $, by continuity in time in $H^1(\mathbb{R})$, there exists $-T_\ve<T^*\leq T_\ve$ with
\bee
    T^*& := & \sup\big\{T\in [-T_\ve,T_\ve], \hbox{ such that for all }  t\in [-T_\ve,T], \hbox{ there exists } r(t)\in \mathbb{R} \hbox{ with } \\
    & & \quad \quad \quad     \|u(t +r(t)) -  \tilde u(t)\|_{H^1(\R)}\leq K^* \ve^{\kappa } \big\}.
\eee
The objective is to prove that $T^*=T_\ve$ for $K^*$ large enough. To achieve this, we argue by contradiction, assuming that $T^*<T_\ve $ and reaching a contradiction with the definition of $T^*$ by proving some independent estimates for $\|u(t +r(t))- \tilde u(t )\|_{H^1(\R)}$ on $[-T_\ve,T^*]$, for a special modulation parameter $r(t)$.

\subsection{Modulation} 
By using the Implicit function theorem we will construct a modulation parameter and to estimate its variation in time:
 
\begin{lem}[Modulation in time]\label{DEFZ} Assume $0<\ve<\ve_0(K^*)$ small enough.
There exists a unique $C^1$ function $\rho_1(t)$ such that, for all $t\in [-T_\ve,T^*]$,
\be\label{defz}
z(t)=u(t + \rho_1(t) )-\tilde u(t) \quad \text{satisfies}\quad 
 \int_\R z(t,x) Q_c'(y) dx =0.
\ee
Moreover, we have,  for all $t\in [-T_\ve,T^*]$,
\be\label{TRANS3}
              |\rho_1(-T_\ve)|+\|z(-T_\ve)\|_{H^1(\R)}\leq K \ve^{\kappa }, \ \|z(t)\|_{H^1(\R)}\leq  2 K^* \ve^{\kappa }.
\ee
In addition, $z(t)$ satisfies the following equation
\be\label{Eqz1}
 z_t + (1+\rho_1') \big\{ z_{xx}  -\la z +  a_\ve [ (\tilde u +z)^m - \tilde u^m ] \big\}_x  -  \rho_1' \tilde u_t + (1+ \rho_1')S[\tilde u] =0.
\ee
Finally, there exist $K,\ga>0$ independent of $K^*$ such that for every $t\in [-T_\ve, T^*]$
\be\label{rho1}
|\rho_1'(t)|\leq \frac K{c(\ve t)-\la}\Big[ \|z\|_{L^2(\R)} +  \ve e^{-\ga\ve|t| } \|z(t)\|_{L^2(\R)} +  \|z(t)\|_{L^2(\R)}^2 + \|S[\tilde u]\|_{L^2(\R)}\Big]. 
\ee
\end{lem}

\begin{proof}
The proof of (\ref{defz})-(\ref{TRANS3}) is by now well-know and it is a consequence of the Implicit Function Theorem. See e.g. \cite{MMcol1} for a detailed proof.
On the other hand, the proof of (\ref{Eqz1}) follows after a simple calculation using (\ref{aKdV}).

Finally, we prove (\ref{rho1}). From (\ref{defz})-(\ref{Eqz1}) we take time derivative and replace $z_t$  to obtain
\bee
0 & = & (1+\rho_1') \int_\R \big\{ z_{xx} -cz +  a_\ve[(\tilde u+z)^m- \tilde u^m] \big\} Q_c'' \\
& &  + \rho_1' \int_\R (\tilde u_t  - (c-\la) z_x) Q_c' - (1+\rho_1') \int_\R S[\tilde u]  Q_c'  + \ve c'(\ve t) \int_\R z \Lambda Q_c'.
\eee
First, note that
$$
\rho_1' \int_\R (\tilde u_t -(c-\la)z_x )Q_c' = - \frac {\rho_1'}{\tilde a} \Big[ (c-\la) \int_\R Q_c'^2 + O( \ve + \|z(t)\|_{L^2(\R)}) \Big].
$$ 
On the other hand,
\bee
\int_\R \big\{ z_{xx} -c z +  a_\ve[(\tilde u+z)^m- \tilde u^m]  \big\} Q_c''  & = & -\int_\R z \mathcal L Q_c'' + O(\ve e^{-\ga\ve|t|} \|z(t)\|_{L^2(\R)})\\
& & \qquad   +\;  O(\norm{z(t)}_{L^2(\R)}^2).
\eee
Collecting these estimates, and using the fact that $\|z(t)\|_{H^1(\R)}$ is small, we get desired result.
\end{proof}

\subsection{Control on the $Q_c$ direction}

We recall from (\ref{E}) that the energy of the function $u(t+\rho_1(t))$ is conserved, moreover, $
E_a[u(t+\rho_1(t))] = E_a[u](t)$ for any $t\in [-T_\ve, T^*]$. In what follows, we will made use of this identity to estimate $z$ against the degenerate direction $Q_c$. First we prove that the approximate solution $\tilde u$ has almost conserved energy. 

\begin{lem}[Almost conservation of energy]\label{ACE}~

Consider $\tilde u$ the approximate solution constructed in Proposition \ref{CV}. Then
\be\label{dE1}
\partial_t E_a[\tilde u](t) = -\int_\R (\tilde u_{xx}-\la \tilde u+ a_\ve \tilde u^m ) S[\tilde u]. 
\ee
In particular, there exists $K>0$ independent of $K^*$ such that 
\be\label{dE2}
\abs{ E_a[\tilde u](t) - E_a[\tilde u](-T_\ve) } \leq K\ve^{\kappa }.
\ee
\end{lem}

\begin{proof}
We start by showing (\ref{dE1}). From (\ref{2.2bis}) we have
\bee
\int_\R S[\tilde u](\tilde u_{xx} -\la \tilde u+ a_\ve \tilde u^m)
& = & \int_\R \tilde u_t 
(\tilde u_{xx} -\la \tilde u+ a_\ve \tilde u^m)\\
& = & -\partial_t \frac 12\int_\R \tilde u_x^2  -\partial_t \frac \la2\int_\R \tilde u^2+  \frac 1{m+1}\partial_t\int_\R a_\ve \tilde u^{m+1} \\
& =& -\partial_t E_a[\tilde u](t),
\eee
as desired.  

Now we consider (\ref{dE2}). From Cauchy-Schwarz inequality, we have
$$
\abs{\partial_t E_a[\tilde u](t)} \leq K\|S[\tilde u](t)\|_{L^2(\R)}, 
$$
for some constant $K>0$. After integration and considering (\ref{INTkl}), we get the result.  
\end{proof}

\begin{lem}[Control in the $Q_c$ direction]\label{Qdir}~

There exists $K,\ga>0$, independent of $K^*$ such that for $0<\ve<\ve_0$ small enough,
$$
\abs{\int_\R  Q_c(y)z } \leq \frac K{ c(\ve t)-\la }\Big[ \ve^\kappa  +  \ve^{1/2} e^{-\ve \ga |t|} \|z(t)\|_{L^2(\R)} +  \|z(t)\|^2_{H^1(\R)}\Big]. 
$$
\end{lem}
\begin{proof}
Consider the conserved energy $E_a[u(t+ \rho_1)]$; we expand this term and make use of the identity $u(t+ \rho_1) = \tilde u(t) + z(t)$ to obtain
\bee
E_a[\tilde u+z](t) & = & E_a[\tilde u](t) - \int_\R z(\tilde u_{xx} -\la \tilde u + a_\ve\tilde u^m) + \frac 12 \int_\R z_x^2 + \frac \la 2 \int_\R z^2 \\
& & \quad  - \frac 1{m+1}\int_\R a_\ve[ (\tilde u + z)^{m+1} - \tilde u^{m+1} -(m+1)\tilde u^m z ].  
\eee
First, note that
\bee
\int_\R z(\tilde u_{xx} -\la \tilde u+ a_\ve\tilde u^m)(t) & = &  \int_\R z(\tilde u_{xx} -\la \tilde u + a_\ve\tilde u^m)(-T_\ve) + \Big\{ E_a[\tilde u](t) -E_a[\tilde u](-T_\ve)\Big\} \\
& & \quad + \ O(\|z(t)\|_{H^1(\R)}^2).
\eee
We use now (\ref{Est2a}):
$$
\int_\R z(\tilde u_{xx} -\la u + a_\ve\tilde u^m) =  (c-\la)\int_\R \tilde u z  + O( \ve e^{-\ga\ve |t|} \|z(t)\|_{L^2(\R)})
$$
The conclusion follows from the above identity and (\ref{dE2}).
\end{proof}

\subsection{Energy functional for $z$}\label{EFz}

Consider the functional
\be\label{F}
\mathcal F(t) := \frac 12 \int_\R (z_x^2 +c(\ve t) z^2) - \frac{1}{m+1}\int_\R a_\ve [(\tilde u+ z)^{m+1} -\tilde u^{m+1} - (m+1)\tilde u^{m}z]. 
\ee

\begin{lem}[Modified coercivity for $\mathcal F$, second version]\label{Coer2}~

There exist $K,\nu_0>0$, independent of $K^*$ and $\ve$ such that for every $t\in [-T_\ve, T_\ve]$
$$
\mathcal F(t) \geq \nu_0 \|z(t)\|_{H^1(\R)}^2 -\abs{\int_\R Q_c(y) z}^2 - K (\ve e^{-\ga\ve |t|} + \ve^2)\|z(t)\|_{L^2(\R)}^2 - K  \|z(t)\|_{L^2(\R)}^3. 
$$ 
\end{lem}

\begin{proof}
We write
\bea
\mathcal F(t) & = & \frac 12 \int_\R (z_x^2 + c z^2 - m a_\ve \tilde u^{m-1} z^2) \label{Co1} \\
&  &  - \frac{1}{m+1}\int_\R a_\ve [(\tilde u+ z)^{m+1} - \tilde u^{m+1} - (m+1)\tilde u^{m}z - \frac 12 m(m+1) \tilde u^{m-1} z^2].\label{Co2}
\eea
In the case $m=2$ the term (\ref{Co2}) above is identically zero, and for $m=3,4$ we have $\abs{(\ref{Co2})} \leq K \|z(t)\|_{L^2(\R)}^3.$ 

On the other hand, the first term above looks as follows
$$
(\ref{Co1})   =  \frac 12 \int_\R (z_x^2 +c(\ve t) z^2 - m  Q_c^{m-1} z^2) - \ve \frac {m a'(\ve \rho)}{2 a (\ve \rho)} \int_\R   y Q_c^{m-1} z^2 + O(\ve^2 \|z(t)\|_{L^2(\R)}^2).
$$
It is clear that 
$$
\abs{\ve \frac {m a'(\ve \rho)}{2 a (\ve \rho)} \int_\R   y Q_c^{m-1} z^2}\leq K  \ve e^{-\ga\ve |t|} \|z(t)\|_{L^2(\R)}^2.
$$
Finally, from Lemma \ref{surL}, we have the existence of constants $K, \nu_0>0$ such that for all $t\in [-T_\ve, T^*]$
$$
\frac 12 \int_\R (z_x^2 +c(\ve t) z^2 - m  Q_c^{m-1} z^2) \geq \nu_0 \|z(t)\|_{H^1(\R)}^2 -K\Big|\int_\R Q_c z \Big|^2.
$$
\end{proof}
Now we use a coercivity argument to obtain independent estimates for $\mathcal F(T^*)$.  

\begin{lem}[Estimates on $\mathcal F(T^*)$]\label{Ka}~

The following properties hold for any $t\in [-T_\ve, T^*]$.
\begin{enumerate}
\item First time derivative. 
\bea
\mathcal F'(t) &  = &   -\int_\R z_t \big\{ z_{xx} -cz + a_\ve [ (\tilde u+z)^m -\tilde u^m ] \big\} + \frac 12 \ve c'(\ve t)\int_\R z^2 \nonumber \\
& &  -\int_\R a_\ve \tilde u_t[ (\tilde u+z)^m -\tilde u^m -m\tilde u^{m-1}z]. \label{Fp}
\eea
\item Integration in time. There exist constants $K,\ga>0$ such that
\bee\label{IntF}
\mathcal F(t) -\mathcal F(-T_\ve)&  \leq &  K(K^*)^4\ve^{4\kappa  -\frac 1{100}}  + K(K^*)^3\ve^{3\kappa  -\frac 1{100}} + KK^* \ve^{2\kappa }\nonumber  \\
 & & + K\int_{-T_\ve}^t  \ve e^{-\ve\ga |t|}\|z(t)\|_{H^1(\R)}^2 dt . 
\eee
\end{enumerate}
\end{lem}

\begin{proof}
First of all, (\ref{Fp}) is a simple computation.
Let us  consider (\ref{IntF}). Replacing (\ref{Eqz1}) in (\ref{Fp}) we get
\bea
\mathcal F'(t) & = & (c(\ve t)-\la) (1+\rho_1') \int_\R 
 a_\ve [ (\tilde u+z)^m -\tilde u^m ]  z_x \label{Fp1} \\
& & - \rho_1' \int_\R \tilde u_t \big\{ z_{xx} -cz + a_\ve [ (\tilde u+z)^m -\tilde u^m ] \big\} \label{Fp2} \\
& & + (1+\rho_1') \int_\R S[\tilde u] \big\{ z_{xx} -cz + a_\ve [ (\tilde u+z)^m -\tilde u^m ] \big\} \label{Fp3} \\
& & + \frac 12 \ve c'(\ve t)\int_\R z^2 -\int_\R a_\ve \tilde u_t [ (\tilde u+z)^m -\tilde u^m -m\tilde u^{m-1}z]. \label{Fp4}
\eea 
Now we consider separate cases. First let us suppose $m=2$. After some simplifications, we get
\bee
(\ref{Fp1})&  =&  (c-\la) (1+\rho_1')\int_\R a_\ve [2\tilde u z + z^2 ] z_x \\
& = & -(c-\la)(1+\rho_1')\int_\R [ a_\ve \tilde u_x z^2  + \ve a'(\ve x) \tilde u z^2 + \frac 13 \ve a'(\ve x) z^3 ] .
\eee
From this
$$
\abs{(\ref{Fp1}) + (c-\la)(1+\rho_1') \int_\R a_\ve \tilde u_x z^2} \leq K \ve e^{-\ga\ve |t|} \|z(t)\|_{L^2(\R)}^2 + K\ve \|z(t)\|^3_{H^1(\R)}. 
$$
On the other hand,
\bee
(\ref{Fp2}) & = & - \rho_1' \int_\R( \tilde u_t + (c-\la)\tilde u_x) \big\{ z_{xx} -cz + a_\ve [ 2\tilde u z + z^2 ] \big\}  +  (c-\la) \rho_1' \int_\R a_\ve   \tilde u_x  z^2 \\
& & \quad +  (c-\la) \rho_1' \int_\R z  [ \tilde u_{xx} -c\tilde u +  a_\ve\tilde u^2 ]_x -(c-\la)\rho_1' \ve \int_\R a'(\ve x) \tilde u^2 z  . 
\eee
In particular, using estimates (\ref{Est1}), (\ref{Est20}) and (\ref{Est2}) we obtain
$$
\abs{(\ref{Fp2}) - (c-\la) \rho_1' \int_\R a_\ve   \tilde u_x  z^2 } \leq K  \ve |\rho_1'|  e^{-\ga\ve|t|} \|z(t)\|_{H^1(\R)} 
$$
We also have
$$
(\ref{Fp3})  =   ( 1+\rho_1') \int_\R z[ S[\tilde u]_{xx} - cS[\tilde u] + 2 a_\ve \tilde uS[\tilde u] + a_\ve z S[\tilde u] ],
$$
thus using (\ref{rho1})
$$
\abs{(\ref{Fp3}) } \leq K \|z(t)\|_{L^2(\R)} \|S[\tilde u](t)\|_{H^2(\R)}.
$$
Finally,
\bee
(\ref{Fp4}) & = &  \frac 12 \ve c'(\ve t)\int_\R z^2 - \int_\R a_\ve (\tilde u_t + (c-\la)\tilde u_x) z^2 +  (c-\la)\int_\R a_\ve \tilde u_x z^2.
\eee
We get then from (\ref{Est2})
$$
\abs{(\ref{Fp4})-(c-\la)\int_\R a_\ve \tilde u_x z^2 } \leq 
K \ve e^{-\ga\ve |t|} \|z(t)\|^2_{L^2(\R)}.
$$
Collecting the above estimates and (\ref{rho1}), and after an integration, we finally get
$$
\abs{\mathcal F(t) -\mathcal F(-T_\ve)}  \leq  K(K^*)^3\ve^{3\kappa  -\frac 1{100}}  + KK^* \ve^{2\kappa }+ K \int_{-T_c}^t  \ve e^{-\ga\ve |s|} \|z(s)\|_{L^2(\R)}^2 ds.
$$ 

The cases $m = 3,4$ are similar, but more involved. From (\ref{Fp1})-(\ref{Fp4}), and after some integration by parts, the result is the following:
\bea
& & \mathcal F'(t)  =  (c-\la)(1+\rho_1') \times \nonumber \\
& & \qquad \Big[ \int_\R  a_\ve \big\{ (\tilde u+z)^m -\tilde u^m -m\tilde u^{m-1}z -\frac m2 (m-1)\tilde u^{m-2}z^2\big\} z_x \label{67} \\
& &  \qquad   -  \frac m2 \ve  \int_\R  a'(\ve x) \tilde u^{m-1} z^2 - \frac \ve 6 m(m-1) \int_\R  a'(\ve x)\tilde u^{m-2} z^3  \nonumber  \\
& & \qquad -  \frac m2  \int_\R  a_\ve (\tilde u^{m-1})_x z^2 - \frac m6 (m-1) \int_\R  a_\ve (\tilde u^{m-2})_x z^3   \Big]\label{68}\\
& & \qquad - \rho_1' \int_\R (\tilde u_t + (c-\la)\tilde u_x) \big\{ z_{xx} -cz + a_\ve [ (\tilde u+z)^m -\tilde u^m ] \big\} \nonumber \\
& & \qquad + (c -\la)\rho_1' \Big[  \int_\R z (\tilde u_{xx} -c\tilde u  +  a_\ve\tilde u^{m} )_x   - \ve \int_\R a'(\ve x)\tilde u^{m} z \Big]\nonumber  \\
& & \qquad + (c-\la)(1+ \rho_1') \int_\R  \tilde u_x  a_\ve \big\{ (\tilde u+z)^m -\tilde u^m -m\tilde u^{m-1} z  \nonumber \\
& & \qquad \qquad  - \frac m2 (m-1) \tilde u^{m-2} z^2  -\frac m6(m-1)(m-2) \tilde u^{m-3} z^3\big\}  \label{71} \\
& & \qquad + \frac m2  (c -\la)  \rho_1' \Big[ \int_\R  a_\ve (\tilde u^{m-1})_x z^2 + \frac 13  (m-1) \int_\R  a_\ve (\tilde u^{m-2})_x z^3\Big] \label{69} \\
& & \qquad+ (1+\rho_1') \int_\R z  \big\{ S[\tilde u]_{xx}  -cS[\tilde u]  + m a_\ve \tilde u^{m-1} S[\tilde u]\big\}\nonumber \\
& & \qquad + (1+\rho_1') \int_\R a_\ve \big\{ (\tilde u+z)^m -\tilde u^m -m\tilde u^{m-1}z \big\} S[\tilde u] \nonumber \\
& & \qquad+ \frac \ve2 c' \int_\R z^2 -\int_\R a_\ve (\tilde u_t + (c-\la)\tilde u_x)[ (\tilde u+z)^m -\tilde u^m -m\tilde u^{m-1}z ] \nonumber \\
& & \qquad+  \frac m2  (c-\la) \Big[ \int_\R a_\ve  (\tilde u^{m-1})_x z^2 + \frac 13 (m-1) \int_\R a_\ve (\tilde u^{m-2})_x z^3\Big]. \label{70}
\eea 


Note that (\ref{68}), (\ref{69}) and (\ref{70}) disappear. With (\ref{67}) and (\ref{71}), we need a little more care. Indeed, for $m=3$,
$$
\abs{ (\ref{67}) + (\ref{71})} = \abs{\frac 14 \ve (c-\la) (1+\rho_1') \int_\R  a'(\ve x) z^4} \leq \ve \|z(t)\|_{L^2(\R)}^4;
$$
In the case $m=4$,
\bee
 (\ref{67}) + (\ref{71}) & = &  (c-\la) (1+\rho_1') \int_\R  a_\ve [  z_x ( 4\tilde u z^3 + z^4)  + \tilde u_x z^4]  \\
 & =& -\ve (c-\la) (1+\rho_1') \int_\R a'(\ve x) (\tilde u z^4 + z^5).
\eee
Consequently we have
$$
\abs{ (\ref{67}) + (\ref{71}) }\leq K\ve e^{-\ga\ve |t|}\|z(t)\|_{L^2(\R)}^4 +  K\ve \|z(t)\|_{L^2(\R)}^5. 
$$
Finally, using (\ref{Est1}), (\ref{Est20}), (\ref{Est2}) we obtain
\bee
\mathcal F'(t) & \leq & K\ve \|z(t)\|_{H^1(\R)}^4 + K \ve e^{-\ga\ve |t|} \|z(t)\|_{L^2(\R)}^2 + K\ve \|z(t)\|_{H^1(\R)}^3 \\
& & + K|\rho_1'(t)|\ve e^{-\ga \ve|t|} \|z(t)\|_{H^1(\R)} + K\|S[\tilde u](t)\|_{H^2(\R)} \|z(t)\|_{L^2(\R)}.
\eee
Integrating and using (\ref{rho1}), we obtain
\bee
\mathcal F(t) -\mathcal F(-T_\ve) & \leq & K(K^*)^4 \ve^{4\kappa -\frac 1{100}} + K(K^*)^3 \ve^{3\kappa -\frac 1{100}}  + KK^*\ve^{2\kappa }\\
& & \qquad + K  \int_{-T_\ve}^t \ve e^{-\ga \ve|s|} \|z(s)\|_{H^1(\R)}^2 ds.
\eee
This finishes the proof.
\end{proof}

We are finally in position to show that $T^*<T_\ve$ leads to a contradiction. 

\subsection{End of proof of Proposition \ref{prop:I}} From Lemma \ref{DEFZ}, $\mathcal F(-T_\ve) \leq K c^{2\kappa },$ and from  Lemmas \ref{Coer2}, \ref{Qdir} and (\ref{IntF}) we get
\bee
\|z(t)\|_{L^2(\R)}^2 &  \leq & K\abs{\int_\R zQ_c(y) }^2 + K\ve^{2\kappa } +  K(K^*)^4\ve^{4\kappa  -\frac 1{100}} + K(K^*)^3\ve^{3\kappa  -\frac 1{100}}  \\
& & \qquad+ KK^* \ve^{2\kappa } + K \int_{-T_\ve}^t  \ve e^{-\ga\ve |t|} \|z(t)\|_{L^2(\R)}^2 dt\\
& & \leq   K\abs{\ve^\kappa  + K^*\ve^{\frac 12+\kappa } e^{-\ga\ve|t|} + (K^*)^2\ve^{2\kappa } + \|S[\tilde u]\|_{L^2(\R)}}^2 +  K\ve^{2\kappa }  \\
& & \qquad + K(K^*)^4\ve^{4\kappa  -\frac 1{100}} + K(K^*)^3\ve^{3\kappa  -\frac 1{100}}  + KK^* \ve^{2\kappa } \\
& & \qquad + K \int_{-T_\ve}^t  \ve e^{-\ga\ve |s|} \|z(s)\|_{L^2(\R)}^2 ds\\
& & \leq  K\ve^{2\kappa } + K(K^*)^4\ve^{4\kappa  -\frac 1{100}} +  K(K^*)^3\ve^{3\kappa  -\frac 1{100}}  + KK^* \ve^{2\kappa } \\
& & \qquad + K \int_{-T_\ve}^t  \ve e^{-\ga\ve |s|} \|z(s)\|_{L^2(\R)}^2 ds.
\eee

Using Gronwall's inequality (see e.g. \cite{Gr}) we conclude that for some large constant $K>0$, but independent of $K^*$ and $\ve$,
$$
\|z(t)\|_{H^1(\R)}^2 \leq  K\ve^{2\kappa } +  K(K^*)^4\ve^{4\kappa  -\frac 1{100}} + K(K^*)^3\ve^{3\kappa  -\frac 1{100}}  + KK^* \ve^{2\kappa }.
$$
From this estimate and taking $\ve$ small, and $K^*$ large enough, we obtain that for all $t\in [-T_\ve, T^*]$,
$$
\|z(t)\|_{H^1(\R)}^2 \leq  \frac 12 (K^*)^2 \ve^{2\kappa }.
$$
This estimate contradicts the definition of $T^*$, and concludes the proof of Proposition \ref{prop:I}.
\end{proof}

\subsection{Proof of Theorem \ref{T0}}

Now we prove the main result of this section, which describes the core of interaction soliton-potential. 

\begin{proof}[Proof of Theorem \ref{T0}]
Consider $u(t)$ the solution constructed in Theorem \ref{Tm1}. We first compare $u(t)$ with the approximate solution $\tilde u(t)$ constructed in Proposition \ref{CV} at time $t=-T_\ve$.

\medskip

\noindent
\emph{Behavior at $t= - T_\ve$}.  From (\ref{mTep}), Proposition \ref{atpmT} and more specifically  (\ref{mTe}) we have that
$$
\| u(-T_\ve) - \tilde u(-T_\ve) \|_{H^1(\R)} \leq K \ve^{10}.
$$

\noindent
\emph{Behavior at $t= T_\ve$}. Thanks to the above estimate and (\ref{SH2}) we can invoke Proposition \ref{prop:I} with $\kappa  := \frac 12$ to obtain the existence of $K_0,\ve_0>0$ such that for all $0<\ve<\ve_0$
$$
\| u( T_\ve +\rho_1(T_\ve) ) -\tilde u(T_\ve) \|_{H^1(\R)} \leq K_0\ve^{1/2},\quad |\rho_1(T_\ve)  | \leq K_0\ve^{-\frac 12-\frac 1{100}} \leq \frac{T_\ve}{100}.
$$
Therefore from (\ref{pTe}) and triangular inequality,
$$
\| u( T_\ve +\rho_1(T_\ve) ) - 2^{-1/(m-1)} Q_{c_\infty} (\cdot  -\rho(T_\ve)) \|_{H^1(\R)} \leq K_0\ve^{1/2}.
$$
(cf. also (\ref{defALPHA}).) Finally note that $(1-\la)T_\ve \leq \rho(T_\ve) \leq (2c_\infty(\la)-\la-1) T_\ve. $
This finishes the proof.
\end{proof}
Next step will be the study of long time properties, on the interval $[T_\ve, +\infty)$.

\bigskip

\section{Asymptotic for large times}\label{6}

\subsection{Statement of the results}

The purpose of this Section is to prove the asymptotic behavior of the solution $u(t)$ as described in Theorem \ref{LTB}. Recall the parameters $\la_0$ and $c_\infty(\la)$ from Theorems \ref{MT} and \ref{T0}.

\begin{thm}[Stability and Asymptotic stability in $H^1$]\label{Tp1}~

Suppose $m=2,4$ with $0< \la \leq \la_0$; or $m=3$ with $0\leq \la \leq \la_0$. Let $0<\beta< \frac 12(c_\infty(\la) -\la)$. There exists $\ve_0>0$ such that if $0<\ve <\ve_0$ the following hold. Suppose that for some time $t_1\geq \frac 12 T_\ve$ and $t_1 \leq X_0 \leq 2t_1$
\be\label{18}
\big\| u(t_1) - 2^{-1/(m-1)}Q_{c_\infty} (x - X_0) \big\|_{H^1(\R)} \leq  \ve^{1/2}.
\ee
where $u(t)$ is a $H^1$-solution of (\ref{aKdV}). 
Then $u(t)$ is defined for every $t\geq t_1$ and there exists $K, c^+>0$ and a $C^1$-function $ \rho_2(t)$ defined in $[t_1,+\infty)$ such that 
\begin{enumerate}
\item \emph{Stability}.
\be\label{S}
 \sup_{t\geq t_1} \big\| u(t) - 2^{-1/(m-1)} Q_{c_\infty} (\cdot - \rho_2(t)) \big\|_{H^1(\R)} \leq K \ve^{1/2},
\ee
where 
$$
|\rho_2(t_1)-X_0 | \leq  K\ve^{1/2},  \quad \hbox{ and for all } t\geq t_1,\quad  |\rho_2'(t)-c_\infty(\la)+\la | \leq K\ve^{1/2}.
$$
\item \emph{Asymptotic stability}. One has
\be\label{AS}
\lim_{t\to +\infty} \big\| u(t) - 2^{-1/(m-1)} Q_{c^+} (\cdot - \rho_2(t)) \big\|_{H^1(x>\beta t)} =0.
\ee
In addition,
\be\label{liminfinity}
\lim_{t\to +\infty} \rho_2'(t) = c^+ -\la, \qquad |c^+ - c_\infty | \leq K\ve^{1/2}. 
\ee
\end{enumerate}
\end{thm}

\begin{rem}
We do not know if stability results are valid in the cases $m=2,4$ and $\la=0$. In particular, note that the stability property as stated above is false if we have
$$
\limsup_{t\to +\infty} \|u(t)\|_{L^2(\R)} =+\infty.
$$
\end{rem}

\begin{rem}
Let us recall that for any $0<\la <\la_0$ the asymptotic stability property (\ref{AS}) holds for any $\beta > -\la$, provided $\ve_0$ small enough, however we will not pursue on this improvement.\footnote{In \cite{Mu3} we made use of this property.}
\end{rem}

 We shall split the proof in two different parts, according with the proof of stability (cf. (\ref{S})) and asymptotic stability (cf. (\ref{AS})). 

The proof of the stability result is standard and similar to Proposition \ref{prop:I}, see also \cite{Benj, MMT}. For this reason, our proof will be in some sense very sketchy. We invite to the reader to consult the references above mentioned for the original proof.  Concerning the asymptotic stability result, the proof will follow closely the papers \cite{MMnon, MMas2}.  

Let us recall that for large times ($t\geq T_\ve$) the soliton-like solution is expected to be far away from the region where $a_\ve$ varies. In particular, from (\ref{ahyp}), the stability and asymptotic stability properties will follow from the fact that in this region (\ref{ahyp}) behaves like the gKdV equation
$$
u_t + (u_{xx} -\la u+ 2u^m)_x =0, \quad \hbox{ in } \{t\geq T_\ve\} \times \R_x.
$$ 
Of course, this formal argument must be stated in a rigorous way.    

\subsection{Stability}
\begin{proof}[Proof of Theorem \ref{Tp1}, stability part]

Let us prove (\ref{S}). Let us assume that for some $K>0$ fixed,
\be\label{48bon}
\|u(t_1)- 2^{-1/(m-1)} Q_{c_\infty} (\cdot  - X_0)\|_{H^1(\R)}\leq K \ve^{1/2}.
\ee
From the local and global Cauchy theory exposed in Proposition \ref{Cauchy} and Theorems \ref{Tm1} and \ref{T0}, we know that  the solution $u$ is well defined for all $t\geq t_1$.

In order to simplify the calculations, note that from (\ref{simpli}) the function $v:=2^{1/(m-1)} u$ solves
$$
v_t + (v_{xx} -\la v + \frac{a_\ve}{2} v^m)_x=0 \quad \hbox{ on } \R_t\times \R_x,
$$
and (\ref{48bon}) now becomes
\be\label{48bon1}
\| v(t_1)-Q_{c_\infty} (\cdot  - X_0)\|_{H^1(\R)}\leq   \tilde K \ve^{1/2}.
\ee
With a slight abuse of notation we will rename $v:=u$ and $\tilde K:= K$, and we will assume the validity of (\ref{48bon1}) for $u$. The parameters $X_0$ and $c_\infty$ remains unchanged.

Let $D_0>2K$ be a large number to be chosen later,  and set 
\bea\label{Tprime}
  T^* & : = &  \sup\Big\{t\geq t_1 \ | \ \forall \ t'\in [t_1, t), \ \exists \ \tilde \rho_2(t')\in\R  \hbox{ smooth s.t. }  
| \tilde\rho_2'(t') - c_\infty +\la| \leq \frac 1{100}, \nonumber \\
& & \qquad \qquad   | \tilde\rho_2 (t_1) - X_0 |\leq \frac 1{100},   \hbox{ and } \; \| u(t')-Q_{c_\infty}(\cdot -\tilde\rho_2(t'))\|_{H^1(\R)}\le D_0 \ve^{1/2} \Big\}.
\eea
Observe that $T^*>t_1$ is well-defined since $D_0>2K$, (\ref{48bon}) and the continuity of $t\mapsto u(t)$ in ${H^1(\R)}$. The objective is to prove $T^*= +\infty$, and thus (\ref{S}). Therefore, for the sake of contradiction, in what follows {\bf we shall suppose} $T^* <+\infty$.

The first step to reach a contradiction is now to decompose the  solution on $[t_1,T^*]$ using modulation theory around the soliton. In particular, we will find a special $\rho_2(t)$ satisfying the hypotheses in (\ref{Tprime}) but with  
\be\label{onehalf}
\sup_{t\in [t_1, T^*]}\|u(t)- Q_{c_\infty}(\cdot -\rho_2(t))\|_{H^1(\R)}\le \frac 12 D_0 \ve^{1/2},
\ee
a contradiction with the definition of $T^*$.

\begin{lem}[Modulated decomposition]\label{3Dr}~

For $\ve>0$ small enough, independent of $T^*$, there exist  $C^1$ functions $\rho_2, c_2$, defined on $[t_1,T^*]$, with $c_2(t)>0$ and  such that the function $z(t)$ given by
\be\label{eta1a}
z(t,x):=u(t,x)-R (t,x),
\ee
where $R(t,x):= Q_{c_2(t) }(x-\rho_2(t))$, satisfies for all $t\in [t_1,T^*],$
\bea 
&&\int_\R  R(t,x) z(t,x)dx =  \int_\R (x-\rho_2(t))R (t,x)z(t,x)dx =0, \; (\hbox{Orthogonality} ), \label{10a}\\ 
&& \|z(t)\|_{H^1(\R)}+ |c_2(t) - c_\infty |  
  \leq  K D_0\ve^{1/2},\; \hbox{ andÊ}\label{11a}\\
&& \|z(t_1)\|_{H^1(\R)}+ |\rho_2(t_1)- X_0|+   |c_2(t_1) -c_\infty |   \leq  K \ve^{1/2} \label{12a},
\eea
where $K$  is not depending on $D_0$.  In addition, $z (t)$ now satisfies the following modified gKdV equation
\bea\label{13a}
& &   z_t + \big\{ z_{xx} -\la z+ \frac{a_\ve}2 [( R + z)^m -  R^m ] + (\frac{a_\ve(x)}{2}-1) Q_{c_2}^m \big\}_x\qquad\qquad\nonumber \\  
& & \qquad\qquad\qquad\qquad + \; c_2'(t)\Lambda Q_{c_2}  + (c_2-\la-\rho_2')(t)Q'_{c_2}   =0.
\eea
Furthermore, for some constant $\ga>0$ independent of $\ve$, we have the improved estimates:
\bea\label{rho2c2}
|\rho_2'(t)+\la -c_2(t)| & \leq & K(m-3)\Big[ \int_\R e^{-\ga |x-\rho_2(t)|}z^2(t,x) dx\Big]^{\frac 12}\nonumber \\
&& \qquad + K\int_\R e^{-\ga |x-\rho_2(t)|}z^2(t,x) dx + K e^{-\ga\ve t};
\eea
and
\be\label{c2rho2}
\frac{|c_2'(t)|}{c_2(t)} \leq K \int_\R e^{-\ga |x-\rho_2(t)|}z^2(t,x) dx  + K e^{-\ga\ve t}\|z(t)\|_{H^1(\R)} + K\ve e^{-\ve\ga t}.
\ee
\end{lem}
 
\begin{rem}
Note that from (\ref{11a}) and taking $\ve$ small enough we have an improved the bound on $\rho_2(t)$. Indeed, for all $t\in [t_1, T^*]$,
$$
|\rho_2' (t) - c_\infty +\la | +  |\rho_2 (t_1) - X_0 |  \leq 2D_0\ve^{1/2}.
$$
Thus, in order to reach a contradiction, we only need to show (\ref{onehalf}).
\end{rem}

\begin{proof}[Proof of Lemma \ref{3Dr}]
As in Lemma \ref{FM} and \ref{DEFZ}, the proof of (\ref{eta1a})-(\ref{12a}) are based in a Implicit Function Theorem application, and is very similar to the proof of Lemma A.1 in appendix A of \cite{MMas2}. 

On the other hand, equation (\ref{13a}) is a simple computation, completely similar to (\ref{EqZ0}) and (\ref{Eqz1}).

Now we claim that from the definition of $T^*$ we can obtain an extra estimate on the parameter $\rho_2(t)$. We claim that for any $t\geq t_1$, 
\be\label{boundapriori}
\rho_2(t) \geq \frac 1{10} (c_\infty(\la) -\la) t_1. 
\ee
Indeed, from (\ref{Tprime}) and after integration between $t_1$ and $t \in [t_1, T^*]$ we have the bound
$$
|\rho_2(t) -\rho_2(t_1) - (c_\infty-\la) (t-t_1)| \leq \frac 1{100}(t-t_1), \quad | \rho_2(t_1) - X_0 | \leq \frac 1{100}.
$$
Thus we have
$$
|\rho_2(t) - (c_\infty-\la) t| \leq \frac 1{100}(t-t_1+1) + | (c_\infty -\la)t_1 -X_0|.
$$
In particular, for any $t\in [t_1, T^*]$ (recall that $\rho_2(t_1) \sim X_0>0$)
$$
\rho_2(t)\geq (c_\infty-\la) t -\frac 1{100}(t-t_1+1) \geq \frac 1{10} c_\infty t.
$$
This inequality implies that the soliton position is far away from the potential interaction region. 

Now we prove the estimates in (\ref{rho2c2}) and (\ref{c2rho2}). For this, first denote $y:= x-\rho_2(t)$. Taking time derivative in the first orthogonality condition in (\ref{10a}) and using the equation (\ref{13a}) we obtain
\bee
0 & = &  - c_2'(t)  \int_\R \Lambda Q_{c_2} (Q_{c_2}- z) +  (c_2-\la-\rho_2')(t) \int_\R Q_{c_2}'   z   -   \frac 12 \int_\R  Q_{c_2}^{m} [( a_\ve - 2)  z]_x\\
& & - \frac{\ve}{2(m+1)} \int_\R a'(\ve x) Q_{c_2}^{m+1} (y)  + \frac{1}{2}\int_\R Q_{c_2}'   a_\ve [( R + z)^m - R^m -m R^{m-1}   z] . 
   .  
\eee
First of all, note that by scaling arguments
\be\label{scal}
\int_\R \Lambda Q_{c_2} Q_{c_2} = \theta c_2^{2\theta-1}(t) \int_\R Q^2.
\ee
Secondly, by redefining $\ga$ if necessary,
$$
\abs{\ve \int_\R a'(\ve x) Q_{c_2}^{m+1} (y)}   \leq  K \ve e^{-\ga\ve c_2(t)\rho_2(t)}\leq K \ve e^{-\ga\ve t}. 
$$
Similarly, from (\ref{boundapriori}) and following (\ref{dec1}) we have
$$
\abs{ \int_\R  Q_{c_2}^{m} [( a_\ve - 2)  z]_x} \leq K \|z(t)\|_{H^1(\R)} e^{-\ga\ve t}.
$$
Finally, note that for $\ga>0$ independent of $\ve$,
$$
\abs{\int_\R Q_{c_2}'   a_\ve [( R + z)^m - R^m -m R^{m-1}   z] } \leq K\int_\R e^{- \ga |y|}z^2.
$$
Collecting the above estimates, we have
\bea
\frac{|c_2'(t)|}{c_2(t)} & \leq & K\int_\R e^{-\ga|y|} z^2  + K |c_2(t) -\la-\rho_2'(t) | \Big[ \int_\R e^{-\ga|y|} z^2 \Big]^{\frac 12} \nonumber \\
& & + K e^{-\ga\ve t}\|z(t)\|_{H^1(\R)} + K \ve e^{-\ga\ve t}.\label{c2}
\eea

On the other hand, by using the second orthogonality condition in (\ref{10a}), we have
\bee
0 
& =&  (c_2-\la- \rho_2')(t) \int_\R   z (y R)_x   + c_2'(t) \int_\R y\Lambda Q_{c_2}   z +  \frac 1{2}(c_2-\la -\rho_2')(t)\int_\R  Q_{c_2}^2 \\
& & + \int_\R (y  R)_x  \Big\{ \frac 12a_\ve [( R + z)^m -  R^m -mR^{m-1}z] + (\frac{a_\ve(x)}{2}-1) Q_{c_2}^m \Big\}\\
& & + \int_\R (y  R)_x  (z_{xx} -c_2 z  + m R^{m-1}z  )  + \frac m2 \int_\R (yR)_x (a_\ve -2) R^{m-1}z.
\eee
Note that by integration by parts,
$$
\int_\R (y  R)_x  (z_{xx} -c_2 z  + m R^{m-1}z  ) =  \int_\R z (2R + (m-3) R^{m}) = (m-3)\int_\R z R^m.
$$
Using the same arguments as in the precedent computations, we have
\bee
|(c_2-\la-\rho_2')(t)| & \leq & K (m-3) (1+ \frac{|c_2'(t)|}{c_2(t)})\Big[ \int_\R z^2 e^{-\ga|y|}\Big]^{\frac 12} \\
& & \quad + K\int_\R z^2 e^{-\ga |y|}+ \abs{\int_\R Q_{c_2}^m(y) (a_\ve -2) }.
\eee
From (\ref{boundapriori}) and following (\ref{dec1}) we have
$$
\abs{\int_\R Q_{c_2}^m(y) (a_\ve -2) }\leq K e^{-\ga\ve t}. 
$$
Putting together (\ref{c2}) and the last estimates, we finally obtain the bounds in (\ref{11a}), and further we obtain (\ref{rho2c2}) and (\ref{c2rho2}), as desired.
\end{proof}

\subsubsection{Almost conserved quantities and monotonicity}

We continue with a complete analogous proof to Proposition \ref{Ue1} from Section \ref{3}. Recall from (\ref{tM}) the definition of the modified mass $\tilde M$. 

\begin{lem}[Almost conservation of modified mass and energy]\label{C2}~

Consider $\tilde M= \tilde M[R]$ and  $E_a=E_a[R ]$ the modified mass and energy of the soliton $R$ (cf. (\ref{eta1a})). Then for all $t\in [t_1, T^*]$ we have
\bea\label{dE0}
& &  \tilde M[R](t)  =  \frac{1}{2}c_2^{2\theta}(t) \int_\R Q^2 + O(e^{-\ve\ga t}); \\
& & E_a[R](t)  =  \frac{1}{2} c_2^{2\theta}(t)(\la- \la_0 c_2(t)) \int_\R Q^2 +O(e^{-\ve \ga t}). \label{dE01}
\eea
Furthermore, we have the bound
\bea\label{dE02}
& & \abs{E_a[R](t_1) -E_a[R](t)  +  (c_2(t_1)-\la) (\tilde M[R](t_1) - \tilde M[R](t)) }   \qquad  \nonumber \\
& & \qquad \qquad \qquad  \leq K \abs{ \Big[\frac{c_2(t)}{c_2(t_1)}\Big]^{2\theta}-1}^2 +K e^{-\ve\ga t_1}. 
\eea
\end{lem}

\begin{proof}
We start by showing the first identity, namely (\ref{dE0}). We consider the case $m=2,4$, the case $m=3$ being easier. First of all, note that from (\ref{tM}),
$$
\tilde M[R](t)  = \hat M[R](t) =   \frac 12 \int_\R \big(\frac{a_\ve}{2}\big)^{1/m}  R^2 = \frac 12c_2^{2\theta}(t) \int_\R Q^2 + \frac 12\int_\R \big[(\frac{a_\ve(x)}{2})^{1/m}-1\big] R^2.
$$
From (\ref{boundapriori})-(\ref{scal}) and following the calculations in (\ref{dec1}),
$$
\abs{ \int_\R (a_\ve^{1/m}(x) -2^{1/m})R^{2}} \leq Ke^{-\ga \ve t }, 
$$
for some constants $K,\ga>0$. 
Now we consider (\ref{dE01}). Here we have
\bee
E_a[R](t) & =&\frac 12\int_\R R_x^2  +\frac \la 2 \int_\R R^2 - \frac 1{2(m+1)} \int_\R a_\ve R^{m+1}  \\
& =&   c_2^{2\theta}(t)  \Big[ c_2(t) (\frac 12 \int_\R Q'^2 - \frac 1{m+1} \int_\R  Q^{m+1} )  + \frac \la 2 \int_\R Q^2 \Big]  +\frac 1{m+1} \int_\R (1-\frac {a_\ve}{2})R^{m+1}. 
\eee
Similarly to a recent computation, we have
$$
\abs{ \int_\R (2-a_\ve(x))R^{m+1}} \leq Ke^{-\ga \ve t }, 
$$
for some constants $K,\ga>0$. On the other hand, from Appendix \ref{AidQ} we have that $\frac 12 \int_\R Q'^2 - \frac 1{m+1} \int_\R  Q^{m+1} = - \frac {\la_0}2 \int_\R Q^2$, $\la_0= \frac{5-m}{m+3}$, and thus
$$
E_a[R](t) =   \frac 12c_2^{2\theta}(t)  ( \la -\la_0 c_2(t) ) \int_\R Q^2  + O( e^{-\ga\ve t}).
$$
Adding both identities we have
$$
E_a[R](t) + (c_2(t_1)-\la) \hat M[R](t) =   c_2^{2\theta}(t) ( c_2(t_1) -\la_0 c_2(t) ) M[Q]  + O(e^{-\ve \ga t}).
$$
In particular,
\bee
& & E_a[R](t_1) -E_a[R](t) + (c_2(t_1)-\la) ( \hat M[R](t_1) - \hat M[R](t) ) = \\
& & \; = \la_0 M[Q] \Big[  c_2^{2\theta +1}(t) -  c_2^{2\theta +1}(t_1)   - \frac{c_2(t_1)}{ \la_0} [ c_2^{2\theta}(t) -c_2^{2\theta}(t_1) ] \Big] 
+ O(e^{-\ve \ga t_1}).
\eee
To obtain the last estimate (\ref{dE02}) we perform a Taylor development up to the second order (around $y=y_0$) of the function $g(y):= y^{\frac {2\theta+1}{2\theta}}$; and where $y:= c_2^{2\theta}(t)$ and $y_0 := c_2^{2\theta}(t_1)$. Note that $\frac{2\theta+1}{2\theta} = \frac{1}{\la_0}$ and $y_0^{1/2\theta} = c_2(t_1)$. The conclusion follows at once.
\end{proof}




In order to establish some stability properties for the function $u(t)$ we recall the mass $\tilde M[u]$ introduced in (\ref{tM}). We have that for $m=3$ and $0\leq \la \leq \la_0$; and for $m=2,4$ and $0<\la\leq \la_0$,
\be\label{tm2}
\tilde M[u](t) - \tilde M[u](t_1) \leq 0.
\ee
for any $t\in [t_1, T^*]$. This result is a consequence of Remark \ref{MM}.

Now our objective is to estimate the quadratic term involved in (\ref{dE02}). Following \cite{MMT}, we should use a ``mass conservation'' identity. However, since the mass is not conserved, estimate (\ref{tm2}) is not enough to obtain a satisfactory estimate. In order to avoid this problem, we shall introduce a virial-type identity.

\subsubsection{Virial estimate}

First, we define some auxiliary functions. Let $\phi \in C(\R)$ be an \emph{even} function satisfying the following properties
\be\label{psip}
\begin{cases}
\phi' \leq 0 \; \hbox{ on } [0, +\infty); \quad  \phi (x) =1 \; \hbox{ on } [0,1], \\
\phi (x) = e^{-x}  \; \hbox{ on } [2, +\infty) \quad\hbox{and}\quad  e^{-x} \leq \phi (x) \leq 3e^{-x}  \; \hbox{ on } [0,+\infty).
\end{cases}
\ee
Now, set $\psi(x) := \int_0^x \phi $. It is clear that $\psi$ an odd function. Moreover, for $|x|\geq 2$,
\be\label{asy}
\psi(+\infty) -\psi (|x|) = e^{-|x|}.
\ee
Finally, for $A>0$, denote 
\be\label{psiA}
\psi_A(x) := A(\psi(+\infty) + \psi(\frac xA))>0; \quad e^{-|x|/A} \leq \psi_A'(x)   \leq 3e^{-|x|/A}. 
\ee
Note that $\lim_{x\to -\infty} \psi(x) =0$. We are now in condition of state the following

\begin{lem}[Virial-type estimate]\label{VL}~

There exist $K, A_0, \delta_0>0$  such that for all $t\in [t_1, T^*]$ and for some $\ga =\ga(c_\infty,A_0)>0$,
\bea\label{dereta}
& &  \partial_t \int_\R  z^2(t,x) \psi_{A_0}(x-\rho_2(t))  \leq \nonumber  \\
& & \qquad \leq   -\delta_0  \int_\R ( z_x^2 + z^2 )(t,x) e^{-\frac 1{A_0} |x-\rho_2(t)|} + K A_0  \| z (t)\|_{H^1(\R)} e^{-\ga \ve t }  .
\eea
\end{lem}

\begin{proof}
See Appendix \ref{SVL}.
\end{proof}

From Lemma \ref{VL} we can improve the estimate (\ref{dE02}) to obtain
\begin{cor}[Quadratic control on the variation of $c_2(t)$]\label{fin}
\bea\label{dE02b}
& & \abs{E_a[R](t_1) -E_a[R](t)  +  (c_2(t_1)-\la) (\tilde M[R](t_1) - \tilde M[R](t)) }   \qquad  \nonumber \\
& & \qquad \qquad \qquad  \leq K\|z(t)\|_{H^1(\R)}^4 + K\|z(t_1)\|_{H^1(\R)}^4 +K e^{-\ve\ga t_1}. 
\eea
\end{cor}
\begin{proof}
From (\ref{c2rho2}) and taking $A_0$ large enough (but fixed and independent of $\ve$) in Lemma \ref{VL}, we have after an integration of (\ref{dereta}) that
$$
\abs{c_2(t)-c_2(t_1)} \leq  K A_0 \|z(t)\|_{L^2(\R)}^2  +KA_0 \|z(t_1)\|_{L^2(\R)}^2  +  {KA_0D_0}{\ve^{-1/2}}e^{-\ga\ve t_1}.
$$
Plugin this estimate in (\ref{dE02}) and taking $\ga$ even smaller, we get the conclusion.
\end{proof}

\subsubsection{Energy estimates}

Let us now introduce the second order functional 
\bee
\mathcal F_2(t) & := & \frac 12\int_\R \Big\{ z_x^2 + [\la+ (c_2(t_1)-\la ) (\frac{a_{\ve}}{2})^{1/m}] z^2 \Big\} \\
& & \quad  -\frac 1{2(m+1)} \int_\R a_\ve [ (R+z)^{m+1}-R^{m+1}-(m+1)R^m z ].
\eee
This functional, related to the Weinstein functional, have the following properties.

\begin{lem}[Energy expansion]\label{EE3}~

Consider $E_a[u]$ and $\tilde M[u]$ the energy and mass defined in (\ref{Ea})-(\ref{tM}). Then we have for all $t\in [t_1,T^*]$,
\bee
E_a[u](t) +  (c_2(t_1) -\la)\tilde M[u](t) & = & E_a[R] + ( c_2(t_1)-\la) \tilde M[R]  + \mathcal F_2(t) \\
& & \quad  +\; O(e^{-\ga \ve  t}\|z(t)\|_{H^1(\R)}).
\eee
\end{lem}
\begin{proof}
Using the orthogonality condition (\ref{10a}), we have
\bee
E_a[u](t) &= & E_a[R] - \int_\R z 
  (a_\ve-2)  R^m + \frac 12 \int_\R z_x^2  + \frac \la 2 \int_\R z^2 \\
  & & \quad  -\frac 1{m+1} \int_\R a_\ve [ (R+z)^{m+1}-R^{m+1}-(m+1)R^m z ].
\eee
Moreover, following (\ref{dec1}), we easily get
$$
\abs{\int_\R z (a_\ve-2)  R^m } \leq K e^{-\ga \ve  t}\|z(t)\|_{H^1(\R)}.
$$
Similarly,
\bee
 \hat M[u](t)  =  \hat M[R] +  \hat M[z] + \int_\R ((\frac {a_\ve}{2})^{1/m} -1) R z = \hat M[R] +  \hat M[z] +O(e^{-\ve\ga t} \|z(t)\|_{H^1(\R)}).
\eee 
Collecting the above estimates, we have
\bee
& & E_a[u](t) +  (c_2(t_1) -\la)\tilde M[u](t) = \\
& & \;  E_a[R] + ( c_2(t_1)-\la) \tilde M[R]  + \frac 1{2}\int_\R \big\{ z_x^2 + [(c_2(t_1)-\la)(\frac{a_\ve}{2})^{1/m} + \la ]z^2\big\} \\
& & \; - \frac 1{2(m+1)} \int_\R a_\ve [ (R+z)^{m+1}-R^{m+1}-(m+1)R^m z ]+ O(e^{-\ga \ve  t}\|z(t)\|_{H^1(\R)}). 
\eee
This concludes the proof.
\end{proof}

\begin{lem}[Modified coercivity for $\mathcal F_2$]\label{Coer3}~

There exists $\ve_0>0$ such that for all $0<\ve<\ve_0$ the following hold. There exist $K,\tilde \la_0>0$, independent of $K^*$ such that for every $t\in [t_1, T^*]$
\be\label{Co3}
\mathcal F_2(t) \geq \tilde \la_0 \|z (t)\|_{H^1(\R)}^2 
  - K \ve e^{-\ga\ve t} \|z (t)\|_{L^2(\R)}^2 +O( \|z(t)\|_{L^2(\R)}^3). 
\ee
\end{lem}

\begin{proof}
First of all, note that 
\bee
\mathcal F_2(t) & = & \frac 1{2}\int_\R \big\{ z_x^2 + [(c_2(t_1)-\la)(\frac{a_\ve}{2})^{1/m} + \la ]z^2\big\} \\
& & \quad - \frac m{2} \int_\R  Q_{c_2}^{m-1}z^2 + O(\|z(t)\|_{H^1(\R)}^3) + O( e^{-\ga\ve t}\|z(t)\|_{H^1(\R)}^2 ).
\eee
Now take $R_0>0$ independent of $\ve$, to be fixed later. Consider the function $\phi_{R_0}(t,x):= \phi ( (x-\rho_2(t))/R_0)$, where $\phi$ is defined in (\ref{psip}). We split the analysis according to the decomposition $1=\phi_{R_0} + (1-\phi_{R_0})$.   Inside the region $\abs{x-\rho_2(t)} \leq R_0$, we have 
$$
2- a_\ve(x) \leq K e^{-\ga\ve |x|}\leq K e^{\ga\ve  R_0} e^{-\ga\ve \rho_2(t)}.
$$
This last estimate is a consequence of (\ref{ahyp}). Outside this region, we have $\phi_{R_0} \geq e^{-R_0}$. We have then 
$$
\int_\R \phi_{R_0}  [(c_2(t_1)-\la)(\frac{a_\ve}{2})^{1/m} + \la ]z^2 \geq [c_2(t_1) - Ke^{\ga\ve  R_0} e^{-\ga\ve \rho_2(t)} ] \int_\R \phi_{R_0} z^2;
$$
for some fixed constants $K, \ga>0$.

On the other hand, $|(1-\phi_{R_0}) Q_{c_2} |\leq K e^{-\ga  R_0}$, and thus
\bea
& & \int_\R (1-\phi_{R_0})   [(c_2(t_1)-\la)(\frac{a_\ve}{2})^{1/m} + \la ]z^2 - \frac m2 \int_\R (1-\phi_{R_0}) Q_{c_2}^{m-1}z^2 \qquad \nonumber \\
& & \qquad\qquad \geq  [(c_2(t_1)-\la)(\frac{1}{2})^{1/m} + \la -K e^{-\ga R_0} ] \int_\R (1-\phi_{R_0}) z^2,\label{R0}
\eea 
for some fixed $K,\ga>0$. Taking $R_0=R_0(m,\la)$ large enough, we have 
$$
(\ref{R0}) \geq \frac 1{2^{1/m}} c_2(t_1) \int_\R (1-\phi_{R_0}) z^2.
$$
Therefore,
\bee
\mathcal F_2(t) & \geq & \frac 1{2}\int_\R \phi_{R_0}\big\{ z_x^2 + c_2(t_1)z^2   - m Q_{c_2}^{m-1}z^2 \big\}  +\frac 1{2}\int_\R (1-\phi_{R_0}) \big\{ z_x^2 + \frac 1{2^{1/m}} c_2(t_1)z^2 \big\}  \\
& & \quad  - K e^{\ga\ve R_0} e^{-\ga\ve\rho_2(t)} \int_\R \phi_{R_0} z^2 + O(\|z(t)\|_{H^1(\R)}^3) + O( e^{-\ga\ve t}\|z(t)\|_{H^1(\R)}^2 ).
\eee
Taking $R_0$ even large if necessary (but independent of $\ve$), and using a localization argument as in \cite{MMannals}, we obtain that there exists $\tilde \la_0>0$ such that
\bee
\mathcal F_2(t) & \geq & \tilde \la_0 \int_{\R} (z_x^2 + z^2) - K e^{\ga\ve R_0} e^{-\ga\ve\rho_2(t)} \int_\R \phi_{R_0} z^2 + O(\|z(t)\|_{H^1(\R)}^3)  \\
& & \qquad    + O( e^{-\ga\ve t}\|z(t)\|_{H^1(\R)}^2 ).
\eee
Finally, taking $\ve_0$ smaller if necessary, we have
$$
\mathcal F_2(t)  \geq  \tilde \la_0 \int_{\R} (z_x^2 + z^2)  + O(\|z(t)\|_{H^1(\R)}^3) + O( e^{-\ga\ve t}\|z(t)\|_{H^1(\R)}^2 ),
$$
for a new constant $\tilde \la_0>0$.
\end{proof}

\subsubsection{Conclusion of the proof} Now we prove that our assumption $T^*<+\infty$ leads inevitably to a contradiction. Indeed, from Lemmas \ref{EE3} and \ref{Coer3}, we have for all $t\in [t_1, T^*]$ and for some constant $K>0,$
\bee
 \|z(t)\|_{H^1(\R)}^2 & \leq &  K \mathcal F_2(t_1) +  E_a[u](t)-E_a[u](t_1) +  (c_2(t_1) -\la)[ \tilde M[u](t)-\tilde M[u](t_1)] \\
& &  \qquad +  E_a[R](t_1) -E_a[R](t) + ( c_2(t_1)-\la) [ \tilde M[R](t_1)- \tilde M[R](t)]\\
& &  \qquad   +  K \ve \sup_{t\in [t_1, T^*]} e^{-\ga\ve t} \|z (t)\|_{L^2(\R)} + K\sup_{t\in [t_1, T^*]} \|z(t)\|_{L^2(\R)}^3. 
\eee
From Lemmas \ref{3Dr} and \ref{C2}, Corollary \ref{fin} and the energy conservation we have
\bee
\|z(t)\|_{H^1(\R)}^2 & \leq & K \ve  +  (c_2(t_1) -\la)[ \tilde M[u](t)- \tilde M[u](t_1)] \\
& &  \qquad +  K \sup_{t\in [t_1, T^*]} \|z(t)\|_{H^1(\R)}^4 + K e^{-\ve\ga t_1}(1+D_0\ve^{1/2})  + K D_0^3\ve^{3/2}. 
\eee
Finally, from (\ref{tm2}) we have $ \tilde M[u](t)- \tilde M[u](t_1)\leq 0$. Collecting the preceding estimates we have for $\ve>0$ small and $D_0=D_0(K)$ large enough
$$
\|z(t)\|_{H^1(\R)}^2 \leq \frac 14D_0^2 \ve,
$$
 which contradicts the definition of $T^*$. The conclusion is that 
 $$
 \sup_{t\geq t_1} \big\| u(t) -2^{-1/(m-1)} Q_{c_2(t)}(\cdot -\rho_2(t)) \big\|_{H^1(\R)} \leq K \ve^{1/2}.
 $$
 Using (\ref{11a}), we finally get (\ref{S}). This finishes the proof.
\end{proof}

\subsection{Asymptotic stability}

Now we prove (\ref{AS}) in Theorem \ref{Tp1}. 

\begin{proof}[Proof of Theorem \ref{Tp1}, Asymptotic stability part]

We continue with the notation introduced in the proof of the stability property (\ref{S}). We have to show the existence of $K,c^+>0$ such that
$$
\lim_{t\to +\infty}\|u(t) -Q_{c^+} (\cdot -\rho_2(t)) \|_{H^1(x>\frac 1{10}c_\infty t)} =0; \quad |c_\infty-c^+| \leq K\ve^{1/2}. 
$$
From the stability result above stated it is easy to check that the decomposition proved in Lemma \ref{3Dr} and all its conclusions  hold \textbf{for any time $t\geq t_1$}. 

%

%



\subsubsection{Monotonicity for mass and energy}

The next step in the proof is to prove some \emph{monotonicity formulae} for local mass and energy.

Let $K_0>0$ and 
\be\label{phiphi}
\phi (x) := \frac 2\pi \arctan ( e^{x/K_0 }).
\ee
It is clear that $\lim_{x\to +\infty} \phi(x) =1$ and $\lim_{x\to -\infty} \phi(x) =0$. In addition, $\phi (-x) =1-\phi(x)$, for all $x\in \R$, and
$$
0< \phi'(x) = \frac{2}{\pi K_0}\frac{e^{x/K}}{1+e^{2x/K_0}}; \qquad  \phi^{(3)}(x)\leq \frac 1{K_0^2} \phi'(x).
$$
Moreover, we have $1- \phi(x) \leq Ke^{-x/K_0}$ as $x\to +\infty$, and $\phi(x) \leq Ke^{x/K_0} $ as $x\to -\infty$.

Let $\sigma, x_0 >0$. We define, for $t,t_0\geq t_1$,  and $\tilde y(x_0):= x- (\rho_2(t_0) + \sigma (t-t_0) + x_0 )$, 
\be\label{I0}
I_{x_0,t_0}(t) := \int_\R u^2(t,x) \phi ( \tilde y(x_0)) dx, \qquad \tilde I_{x_0,t_0}(t) :=  \int_\R u^2(t,x) \phi (\tilde y(-x_0)) dx,
\ee
and
$$
J_{x_0,t_0} :=  \int_\R[ u_x^2  + u^2 -\frac {2a_\ve}{m+1} u^{m+1}](t,x) \phi (\tilde y(x_0)) dx.
$$

\begin{lem}[Monotonicity formulae]\label{Mon3}~

Suppose $0<\sigma < \frac12(c_\infty(\la) -\la) $ and $K_0> \sqrt{\frac {2}\sigma}$. There exists $K,\ve_0 >0$ small enough such that for all $0<\ve<\ve_0$ and for all $t,t_0\geq t_1$ with $t_0\geq t$ we have
\be\label{I}
I_{x_0,t_0}(t_0) -I_{x_0,t_0}(t) \leq K \big[  e^{-x_0/K_0} + \ve^{-1}e^{-\ga\ve T_\ve} e^{-\ve \ga x_0/K_0 }\big] .
\ee
On the other hand, if $t\geq t_0$  and $\rho_2(t_0)\geq  t_1 +x_0$,
\be\label{tI}
\tilde I_{x_0,t_0}(t) -\tilde I_{x_0,t_0}(t_0) \leq K \big[ e^{-x_0/K_0}  +  \ve^{-1} e^{-\ve \ga \rho_2(t_0)} e^{\ve\ga x_0/K_0}\big],
\ee
and finally if $t_0\geq t$,
\be\label{J}
J_{x_0,t_0}(t_0) -J_{x_0,t_0}(t) \leq K \big[  e^{-x_0/K_0} + \ve^{-1} e^{-\ga\ve T_\ve} e^{-\ve \ga x_0/K_0 }\big].
\ee
\end{lem}

\begin{proof}
For the sake of brevity we prove this Lemma in Appendix \ref{MonF}.
\end{proof}

\subsubsection{Conclusion of the proof}

Now we finally sketch the proof of the asymptotic stability theorem, namely (\ref{AS}). Consider $0<\ve<\ve_0$ and $u(t)$ satisfying (\ref{18}). From Lemma \ref{3Dr}, we can decompose $u(t)$ for all $t\geq t_1$ such that
$u(t,x) =2^{-1/(m-1)} Q_{c_2(t)} (x-\rho_2(t))  + z(t,x),$
where $ z$ satisfies (\ref{10a}), (\ref{11a}), (\ref{12a}), (\ref{rho2c2}) and (\ref{c2rho2}). We claim that there exists $K=K(D_0)>0$ such that
\be\label{3.2}
\int_{t_1}^{+\infty}\!\!\! \int_\R ( z_x^2 + z )(t,x) e^{-\frac 1{A_0} |x-\rho_2(t)|} \leq K(D_0) \ve.
\ee
This last estimate is a simple consequence of Lemma \ref{VL} and an integration in time.

Now we claim that 
\be\label{cp}
c^+ := \lim_{t\to +\infty} c_2(t) <+\infty, \quad \hbox{ and }\quad |c^+ - c_{\infty}| \leq K \ve^{1/2} .
\ee
In fact, note that from (\ref{3.2}) there exists a sequence $t_n\uparrow +\infty$, $t_n \in [n,n+1)$ such that
\be\label{tn0}
\lim_{n\to +\infty} \int_\R ( z_x^2 + z )(t_n,x) e^{-\frac 1{A_0} |x-\rho_2(t_n)|} =0.
\ee
From this and (\ref{rho2c2})-(\ref{c2rho2}), and taking $A_0>0$ large such that $\frac 1{A_0} <\ga$, we get
$$
\abs{c_2'(t)} \leq K\int_\R  z^2(t,x) e^{-\frac 1{A_0} |x-\rho_2(t)|} + K e^{-\ga\ve t}.
$$
This inequality combined with (\ref{3.2}) and (\ref{12a}) allow us to conclude (\ref{cp}). Note that this proves the first part of (\ref{liminfinity}).

The next step is to prove that 
$$
\limsup_{t\to +\infty} \int_\R ( z^2_x +  z^2) (t,x+\rho_2(t)) \phi(x-x_0) \leq K e^{-x_0/2K_0} + K\ve^{-1} e^{-\ve\ga T_\ve} e^{-\ve\ga x_0/K_0}.
$$
First of all, note that the above assertion follows directly from the decay properties of $R$ and the estimate  
\be\label{LSa}
\limsup_{t\to +\infty} \int_\R ( u^2_x + u^2) (t,x+\rho_2(t)) \phi(x-x_0) \leq K e^{-x_0/2K_0} + K\ve^{-1} e^{-\ve\ga T_\ve} e^{-\ve\ga x_0/K_0}.
\ee
So we are now reduced to prove this last estimate.
We start from (\ref{J}): we have for $t_0\geq t_1$,
$$
J_{x_0,t_0} (t_0) \leq J_{x_0, t_0} (t_1)  + K e^{-x_0/K_0} + K\ve^{-1} e^{-\ve\ga T_\ve} e^{-\ve\ga x_0/K_0}.
$$
From the equivalence between the energy and $H^1$-norm (we are in a subcritical case), we have
\bee
 \int_\R (u_x^2 + u^2) (t_0, x+ \rho_2(t_0))\phi ( x- x_0) &  \leq & K \int_\R (u_x^2 + u^2) (t_1, x+ \rho_2(t_1))\phi ( x- y_0) \\
 & & \qquad   + K e^{-x_0/2K} + K\ve^{-1} e^{-\ve\ga T_\ve } e^{-\ve\ga x_0/K_0} ,
\eee
where $y_0 := \rho_2(t_0) -\rho_2(t_1) +\sigma ( t_1- t_0) + x_0. $ Now we send $t_0\to +\infty$ noticing that $y_0 \to +\infty$. This gives (\ref{LSa}), as desired.

The next step in the proof is to prove that 
\be\label{tn}
\lim_{n\to +\infty} \int_\R ( z_x^2 + z^2)(t_n, x) \phi (x-\rho_2(t_n) + x_0)dx  =0.
\ee
where $(t_n)_{n\in \N}$ is the sequence from (\ref{tn}). Indeed, note that for any $x_1>0$,
\bee
 \int_\R ( z_x^2 + z^2)(t_n, x + \rho_2(t_n)) \phi (x + x_0) & \leq &  K (e^{\frac{x_0}{A_0}} + e^{\frac{x_1}{A_0}} ) \int_\R ( z_x^2 + z^2)(t_n, x + \rho_2(t_n)) e^{-\frac{|x|}{A_0}}  \\
& & \qquad + K\int_\R  ( z_x^2 + z^2)(t_n, x + \rho_2(t_n)) \phi (x - x_1).
\eee
Thus, using (\ref{tn}) we are able to take in the above inequality the limit $n\to +\infty$ with $x_0,x_1$ fixed. Next, we send $x_1\to +\infty$ to obtain the conclusion.

We finally prove that the above result holds for any sequence $t_n\to +\infty$. Let $\beta < c_\infty(\la) -\la$ to be fixed. We want to prove that for $\ve$ small enough,
$$
\lim_{t\to +\infty } \int_\R ( z_x^2 + z^2)(t, x) \phi (x- \beta t)dx  =0.
$$
First, we claim that for any $ t_2, t_3>t_1$ with $t_2<t_3$ and $\rho_2(t_2) > x_0 + t_1$, we have
\be\label{IDN}
\int_\R u^2(t_3,x)\phi (x- y_3)dx \leq \int_\R u^2(t_2,x)\phi (x- y_2)dx + K e^{-x_0/K_0} + K\ve^{-1} e^{-\ga\ve \rho_2(t_2)} e^{\ga\ve x_0/K_0},   
\ee
where $y_3:= \rho_2(t_2) + \frac 12\beta (t_3-t_2) -x_0$ and $y_2 := \rho_2(t_2) -x_0$. In fact, the left hand side of the above inequality corresponds to $\tilde I_{x_0, t_2}(t_3)$ and the right one is $\tilde I_{x_0, t_2}(t_2)$, with $\sigma:= \frac 12 \beta$ (cf. (\ref{I0}) for the definitions). Thus the above inequality is a consequence of  Lemma \ref{Mon3}, more specifically of (\ref{tI}).

Now the rest of the proof is similar to \cite{MMnon}. Since $\int_\R z(t,x+\rho_2(t)) R(x)=0$, we have
$$
 \Big| \int_\R z(t,x+\rho_2(t)) R(x)\phi(x+x_0) \Big| \leq 
K\ve^{1/2} e^{-x_0/2K_0}
$$
Second, we use the decomposition $u(t,x) =2^{-1/(m-1)} Q_{c_2(t)}(x- \rho_2(t)) + z(t,x)$ in (\ref{IDN}) to get
\bea\label{IDN2}
\int_\R z^2(t_3,x)\phi (x- y_3)dx & \leq & \int_\R z^2(t_2,x)\phi (x- y_2)dx + K e^{-x_0/2K_0}  \\
& & \; + K\ve^{-1} e^{-\ga\ve \rho_2(t_2)} e^{\ga\ve x_0/K_0} + K|c_2(t_2) - c_2(t_3)|.    \nonumber
\eea
Third, consider $t>t_1$ large, and define $t' \in (t_1,t)$ such that $\beta t := \rho_2(t') +\frac \beta 2 (t-t') -x_0. $ Note that $t'\to +\infty$ as $t\to +\infty$.  Since $t_n\in [n,n+1)$ there exists $n=n(t)$ such that $0<t-t_n\leq 2$, and then 
$$
\beta t := \rho_2(t_n) +\frac \beta 2 (t-t_n) -\tilde x_0, \qquad \hbox{ with } |\tilde x_0-x_0|\leq 10. 
$$
Now we apply (\ref{IDN2}) between $t_3=t$ and $t_2=t_n$. We get
\bee
\int_\R z^2(t,x)\phi (x- \beta t)dx & \leq & \int_\R z^2(t_n,x)\phi (x- \rho_2(t_n) + \tilde x_0)dx + K e^{-x_0/2K_0}\\
& & \quad + K\ve^{-1} e^{-\ga\ve \rho_2(t_n)} e^{\ga\ve x_0/K_0} + K|c_2(t) - c_2(t_n)|.   
\eee
Since $n(t)\to +\infty$ as $t\to +\infty$, by (\ref{tn}) and (\ref{cp}) we obtain  
$$
\limsup_{t\to +\infty} \int_\R z^2(t,x) \phi(x-\beta t) dx \leq Ke^{-x_0/2K_0},
$$
and since $x_0$ is arbitrary (because of $\lim_{t\to +\infty} \rho_2(t_n) =+\infty$), we get the desired result. 
The same result is still valid for $z_x$. We have
$$
\limsup_{t\to +\infty} \int_\R z_x^2(t,x) \phi(x-\beta t) dx \leq Ke^{-x_0/2K_0}.
$$
Finally, let $w^+(t,x) :=  u(t,x) - 2^{-1/(m-1)}  Q_{c^+}(x-\rho_2(t))  =  z(t,x) + 2^{-1/(m-1)} [ Q_{c_2(t)}(x-\rho_2(t)) - Q_{c^+}(x-\rho_2(t))].$ From (\ref{cp}) and the above result we finally obtain (\ref{AS}). 
\end{proof}

\bigskip

\section{Proof of the Main Theorems}\label{7}

In this section we prove the Main Theorems of this work, namely Theorems \ref{MT}, \ref{LTB} and \ref{MTcor}.

\medskip

The proof of Theorem \ref{MT} essentially combines Theorems \ref{Tm1} and \ref{T0} in order to obtain the global solution $u(t)$ with the required properties. This method had also been employed in \cite{MMcol1, MMMcol, Mu}. The proof of Theorem \ref{LTB} is a consequence of Theorem \ref{Tp1}. Finally, Theorem \ref{MTcor} will require several additional arguments, in particular the fundamental Lemma \ref{DE}.

\subsection{Proof of Theorem \ref{MT}}

From Theorem \ref{Tm1} there exists a solution $u$ of (\ref{aKdV}) satisfying $u \in C(\R, H^1(\R))$ and (\ref{lim0}).
This solution also satisfies, from (\ref{mTep}),
$$
\|u(-T_\ve) -Q(\cdot +(1-\la)T_\ve) \|_{H^1(\R)} \leq K\ve^{10},
$$
for $\ve$ small enough. In addition, $u$ is unique in the cases $\la>0$ and $m=2,4$; and $m=3$, $\la\geq 0$. This proves the part $(1)$ in Theorem \ref{MT}.

Next, we invoke Theorem \ref{T0} to obtain part $(2)$ in Theorem \ref{MT}. In particular, we have (\ref{INT41}) and (\ref{INT42}). We define $\tilde T_\ve := T_\ve + \rho_1(T_\ve)$, and $\rho_\ve := \rho(T_\ve)$. Then (\ref{PARA}) - (\ref{Para}) are straightforward.

\subsection{Proof of Theorem \ref{LTB}}

The proof of Theorem \ref{LTB} is a consequence of Theorem \ref{Tp1}. Indeed, suppose $m=2,3,4$ with $\la>0$ for $m=2,4$. Define $t_1 := T_\ve + \rho_1(T_\ve)$ and $X_0 := \rho(T_\ve)$. Then, from the above estimates and Theorem \ref{Tp1} we have \emph{stability} and \emph{asymptotic stability} at infinity. In other words, there exists a constant $c^+>0$ and a $C^1$ function $\rho_2(t)\in \R$ such that 
$$
w^+(t) := u(t) -2^{-1/(m-1)}Q_{c^+}(\cdot -\rho_2(t))
$$
satisfies (\ref{S}) and (\ref{AS}). This proves (\ref{MT2}) and (\ref{MT3}).

We finally prove (\ref{Pc}) and (\ref{Pc1}). From the energy conservation, we have for all $t\geq t_1$,
$$
E_a[u](-\infty) = E_a \big[2^{-1/(m-1)} Q_{c^+}(\cdot -\rho_2(t)) + w^+(t) \big]
$$
In particular, from (\ref{AS}) and Appendix \ref{IdQ} we have as $t\to +\infty$
\be\label{Eplus}
(\la-\la_0)M[Q]    =  \frac{(c^+)^{2\theta}}{2^{2/(m-1)}} (\la -\la_0 c^+ )M[Q] + E^+.
\ee
From this identity $E^+ := \lim_{t\to +\infty} E_a[w^+](t)$ is well defined. This proves (\ref{Pc}). To deal with (\ref{Pc1}), note that from the stability result (\ref{S}) and the Morrey embedding we have that for any $\la>0$
\bee
E[w^+](t) & = & \frac 12\int_\R (w^+_x)^2(t) +\frac \la 2 \int_\R (w^+)^2(t) -\frac 1{m+1} \int_\R a_\ve (w^+)^{m+1}(t)\\
& \geq &  \frac 12\int_\R (w^+_x)^2(t) +\frac \la 2 \int_\R (w^+)^2(t) -K\ve^{(m-1)/2} \int_\R a_\ve (w^+)^{2}(t)\\
& \geq & \mu \|w^+(t)\|_{H^1(\R)}^2
\eee
for some $\mu=\mu(\la)>0$. Passing to the limit we obtain (\ref{Pc1}).

Now we prove the bound (\ref{Pc2}). First, the treat the cubic case with $\la=0$. Here, from (\ref{Eplus}) we have
$$
E^+ = \la_0( \frac{(c^+)^{3/2}}{2^{2/(m-1)}} -1) M[Q].
$$
Since in this case we have $2^{2/(m-1)} = 2 = c_\infty ^{3/2}$, $M[Q]=2$ and $\la_0 = 1/3$, we obtain
$\frac 32E^+ =  \big(\frac{c^+}{c_\infty}\big)^{3/2} -1.$

Now we deal with the case $\la>0$. First of all, note that after an algebraic manipulation the equation for $c_\infty$ in (\ref{cinf}) can be written in the following form:
$$
  \frac{c_\infty^{2\theta}}{2^{2/(m-1)}} (\la_0 c_\infty -\la)M[Q] =(\la_0-\la)M[Q].
$$
On the other hand, note that from (\ref{Eplus}) and (\ref{Pc1}) we have
$$
\mu \limsup_{t\to +\infty }\|w^+(t)\|_{H^1(\R)}^2  \leq  \frac{(c^+)^{2\theta}}{2^{2/(m-1)}} (\la_0 c^+ -\la)M[Q] -  (\la_0-\la)M[Q].
$$ 
Putting together both estimates, we get
$$
 \tilde \mu \limsup_{t\to +\infty }\|w^+(t)\|_{H^1(\R)}^2  \leq   (c^+)^{2\theta+1}-c_\infty^{2\theta+1}   -\frac \la{\la_0} ( (c^+)^{2\theta} -c_\infty^{2\theta} ), 
$$
for some $\tilde \mu>0$. Using a similar argument as in Lemma \ref{C2} we have
$$
 \tilde \mu \limsup_{t\to +\infty }\|w^+(t)\|_{H^1(\R)}^2  \leq \frac 1{\la_0}(c_\infty - \la)((c^+)^{2\theta} -c_\infty^{2\theta} )   + O(\abs{(c^+)^{2\theta}-c_\infty^{2\theta} }^2).
$$
From this inequality and the bound $\abs{c^+-c_\infty}\leq K\ve $ we get
$$
 \big( \frac{c^+}{c_\infty} \big)^{2\theta} -1 \geq \tilde \mu \limsup_{t\to +\infty }\|w^+(t)\|_{H^1(\R)}^2,
$$
as desired.

\subsection{Proof of Theorem \ref{MTcor}}

In this section we prove that there is no pure soliton at infinity. To obtain this result we will use a contradiction argument, together with a monotonicity formula which provides polynomial decay of the solution and $L^1$-integrability, something contradictory with the change of scaling of the soliton.

\begin{proof}[Proof of Theorem \ref{MTcor}]
By contradiction, we suppose that (\ref{MTcor1}) is false. In particular, 
$$
\lim_{t\to +\infty}\|w^+(t)\|_{H^1(\R)} =0.
$$

First of all, note that from the above limit and the sub-criticality we have $E^+= 0$. 
Therefore, by using (\ref{Eplus}), and after some basic algebraic manipulations we have that $c^+$ must satisfy the following algebraic equation (compare with (\ref{cinf})):
$$
(c^+)^{\la_0 } (c^+   -\frac \la{\la_0 } )^{1-\la_0 } = 2^p (1 - \frac \la{\la_0 })^{1-\la_0 }.
$$
This relation and the uniqueness of $c_\infty$ gives
\be\label{cplus}
c^+ =
c_\infty(\la).
\ee
In other words, the soliton-solution is \emph{pure} (cf. Definition \ref{PSS}). 

Let us consider now the decomposition result for $u(t)$ from Lemma \ref{3Dr}.  We claim that $ z (t)$ also vanishes at infinity. Indeed, from Lemma \ref{3Dr}, the fact that for $t\geq t_1$
$$
u(t) = R(t) +  z (t) =   w^+(t) + 2^{-1/(m-1)} Q_{c_\infty} (\cdot -\rho_2(t)), 
$$
and estimates (\ref{11a})-(\ref{cp}), we have
\be\label{LIeta}
\lim_{t\to +\infty} \| z(t)\|_{H^1(\R)} =0,
\ee
and
$$ 
u(t, \cdot +\rho_2(t)) \to 2^{-1/(m-1)}  Q_{c_\infty} \; \hbox{ in }  H^1(\R)  \hbox{Ê as  } t\to +\infty,\qquad
\lim_{t\to +\infty} \rho_2'(t) -(c_\infty(\la)-\la) = 0.
$$

In order to prove the following results, we need a simple but important result.

\begin{lem}[Monotonicity of mass backwards in time]\label{MMr}~

Suppose $u(t)$ solution of (\ref{aKdV}) constructed in Theorem \ref{Tm1}, satisfying (\ref{S}) and (\ref{AS}). Define
\be\label{Mback}
\mathcal M[u](t) :=\int_\R \frac {u^2(t,x)}{a_\ve(x)} dx.
\ee
Then, under the additional hypothesis $\la>0$ for $m=2,3,4$, we have that for all $t, t' \geq t_1$, with $t'\geq t$,
\be\label{decr}
\mathcal M[u](t) -\mathcal M[u](t') \leq K e^{-\ve \ga t}.
\ee
\end{lem}
\begin{proof}
First of all, a simple computation tell us that the time derivative of $\mathcal M[u](t)$ is given by
$$
\partial_t \int_\R \frac{u^2}{ a_\ve } =  2\ve  \int_\R u_x^2  \frac{ a_\ve'}{a_\ve^2} +  \ve  \int_\R u^2 \big[  \la \frac{ a_\ve'}{a_\ve^2} -\ve^2 (\frac{a_\ve'}{a_\ve^2})'' \big] -  2\ve  \int_\R \frac{a_\ve' }{a_\ve} u^{m+1}.
$$
Replacing the decomposition $u=R + z$  given by Lemma \ref{3Dr}, assumption (\ref{3d1d}), and using similar estimates to (\ref{dec1}), and the smallness of $\|z(t)\|_{H^1(\R)}$, we get
$$
\partial_t \mathcal M[u](t) \geq  -K\ve e^{-\ve\ga t},
$$
for some $K,Ê\ga >0$. The final conclusion is direct after integration.
\end{proof}

\begin{rem}
Note that estimate (\ref{decr}) in Lemma \ref{Mback} is valid under the additional assumption $0<\la\leq \la_0$. This extra hypothesis unfortunately does not hold for the case $m=3$, $\la=0$.
\end{rem}

The last result allows to prove a new version of Theorem \ref{Tm1}, for positive times. 

\begin{prop}[Backward uniqueness]\label{Tmp1}~

Suppose $m=2,3,4$. Let $\beta\in \R$ and $0<\la\leq \la_0$. There exist constants $K,\ga,\ve_0>0$ and a unique solution $v= v_\beta \in C([\frac 12T_\ve, +\infty), H^1(\R)) $ of (\ref{aKdV}) such that 
\be\label{plus1}
\lim_{t\to +\infty} \big\|v(t) - 2^{-1/(m-1)} Q_{c_\infty} (\cdot -  (c_\infty(\la) -\la) t -\beta )\big\|_{H^1(\R)} =0.
\ee
Furthermore, for all $t\geq \frac 12 T_\ve$ and $s\geq 1$ the function $v(t)$ satisfies
\be\label{plus2}
\big\|v(t) - 2^{-1/(m-1)} Q_{c_\infty} (\cdot -  (c_\infty(\la) -\la) t -\beta) \big\|_{H^s(\R)} \leq K\ve^{-1} e^{-\ve\ga t}.
\ee
Finally, suppose that there exists $\tilde v(t) \in H^1(\R)$ solution of (\ref{aKdV}) such that 
\be\label{plus3}
\lim_{t\to +\infty} \| \tilde v(t) - 2^{-1/(m-1)} Q_{c_\infty} (\cdot -  \rho_2(t))\|_{H^1(\R)} =0.
\ee
Then $\tilde v \equiv v_\beta$ for some $\beta \in \R$.
\end{prop}

\begin{proof}
Given $\beta\in \R$, the proof of existence and uniqueness of the solution $v_\beta$ satisfying (\ref{plus1}) and (\ref{plus2}) is identical the proof of Theorem \ref{Tm1} in Section \ref{3} and Appendix \ref{Thm0}. 
Indeed, first we construct a sequence of functions $v_n$ as in (\ref{CPn}) for times $t\sim T_n$. Next, we prove a decomposition lemma as in Lemma \ref{FM}. This decomposition allows to prove a version of (\ref{decr}) for $\mathcal M[v_n](t)$. The main difference is given in estimates (\ref{Phi})-(\ref{Phi1}), where now we introduce the modified mass $\mathcal M[v_n](t)$ defined in (\ref{Mback}). The energy functional in \ref{Expans1} is now given by $E_a[v_n](t) + ( c_\infty(\la) -\la) \mathcal M[v_n](t)$. The rest of the proof, including the uniqueness, adapts \emph{mutatis mutandis}.

Now consider $\tilde v$ a solution of (\ref{aKdV}) satisfying (\ref{plus3}). Using monotonicity arguments, similar to the proof of Lemma \ref{Uni1}, we have the existence of $\beta\in \R$ such that 
$$
\big\| \tilde v(t) - 2^{-1/(m-1)} Q_{c_\infty} (\cdot - (c_\infty(\la) -\la) t -\beta )\big\|_{H^1(\R)} \leq K \ve^{-1} e^{-\ve\ga t},
$$
for some $K, \ga>0$. This implies that there exists $\beta\in \R$ such that $\tilde v$ satisfies (\ref{plus1}). The conclusion follows from the uniqueness of $v(t)$.
\end{proof}

As a consequence of this result together with (\ref{LIeta}), the solution $u(t)$ constructed in Theorem \ref{Tm1} satisfies the following exponential decay at infinity: there exist $K,\ga>0$ and $\beta\in \R$ such that, for all $t\geq t_1$, if $\tilde \rho_2(t) :=(c_\infty(\la) -\la)t + \beta,$ then
\be\label{decofin}
\tilde z(t) :=u(t) - 2^{-1/(m-1)} Q_{c_\infty}(\cdot - \tilde \rho_2(t)), \quad \hbox{ satisfies }\quad  \| \tilde z(t) \|_{H^2(\R)} \leq K \ve^{-1} e^{-\ve \ga t}.
\ee

Now we prove that this strong $H^1$-convergence gives rise to strange localization properties.

\begin{lem}[$L^2$-exponential decay on the left the soliton solution]\label{der}~

There exist $K,\tilde x_0>0$ large enough such that for all $t\geq T_0$ and for all $x_0\geq \tilde x_0$
\be\label{Maizq}
\|u(t, \cdot + \tilde \rho_2(t))\|_{L^2(x \leq -x_0)}^2 \leq Ke^{- x_0/K}.
\ee
\end{lem}

\begin{proof}
Suppose $x_0>0$, $t,t_0\geq t_1$ and $\sigma>0$ from (\ref{AS}). Consider the modified mass
$$
\tilde I_{t_0,x_0}(t) := \frac 12\int_\R \frac {u^2(t,x)}{a_\ve(x)}(1- \phi(y))dx,
$$
with $y:= x-( \tilde \rho_2(t_0) + \sigma (t-t_0)- x_0)$ and $\phi$ defined in (\ref{phiphi}). For this quantity we claim that for $x_0>\tilde x_0$ and for all $t\geq t_0$,
\be\label{Alpha}
\tilde I_{t_0,x_0}(t_0)-\tilde I_{t_0,x_0}(t) \leq K e^{-x_0/K}(1 + e^{-\frac 12 \sigma (t-t_0)/K}).
\ee
Let us assume this result for a moment. After sending $t\to +\infty$ and using (\ref{AS}), we have $\lim_{t\to +\infty}\tilde I_{t_0,x_0}(t) =0$ and thus
$$
\tilde I_{t_0,x_0}(t_0) \leq K e^{-x_0/K}.
$$
From this last estimate (\ref{Maizq}) is a direct consequence of the fact that $t_0\geq t_1$ is arbitrary.

Finally, let us prove (\ref{Alpha}). A direct calculation tell us that 
\bee
\frac 12\partial_t \int_\R \frac{(1-\phi(y))}{ a_\ve } u^2 & =&  \frac 32 \int_\R \frac{\phi' }{  a_\ve } u_x^2  + \frac 32 \ve \int_\R \frac{a_\ve' }{a_\ve^2} (1-\phi) u_x^2  - \frac{m}{m+1} \int_\R \phi' u^{m+1} \\
& & \quad  +\frac 12 \int_\R u^2 \big[ ( \sigma + \la)\frac{\phi'}{ a_\ve }   - \frac{\phi^{(3)}}{a_\ve} +3 \ve\phi'' \frac{a_\ve'}{a_\ve^2} + 3\ve^2 \phi' (\frac{a_\ve'}{a_\ve^2})'   \big] \\
& & \quad + \frac \ve2 \int_\R u^2 \big[ \la \frac{ a_\ve'}{a_\ve^2} -\ve^2 (\frac{a_\ve'}{a_\ve^2})'' \big](1-\phi) - \ve  \int_\R \frac{a_\ve' }{a_\ve} u^{m+1}(1-\phi).
\eee
Using the decomposition (\ref{decofin}), we have
$$
\abs{ \int_\R \phi' u^{m+1}} \leq K \ve^{(m-1)/2} \int_\R \phi' \tilde z^2 + K e^{-\frac 12\sigma (t-t_0)} e^{-x_0/K},
$$
and
$$
\abs{ \int_\R \frac{a_\ve' }{a_\ve} u^{m+1}(1-\phi)} \leq K e^{-\frac 1 2\sigma (t-t_0)} e^{-x_0/K} + K \ve^{(m-1)/2} \int_\R \frac{a_\ve' }{a_\ve} \tilde z^{2}(1-\phi).
$$
After these two estimates, it is easy to conclude that
$$
\frac 12\partial_t \int_\R \frac{(1-\phi(y))}{ a_\ve } u^2 \geq -K  e^{-\frac 12\sigma (t-t_0)} e^{-x_0/K}.
$$
The conclusion follows after integration in time.
\end{proof}

The proof of decay on the right hand side of the soliton requires more care, and is valid under the assumption $\limsup_{t\to +\infty} \|w^+(t)\|_{H^1(\R)} =0$ and $\la>0$. We do not expect to have exponential decay in a general situation, but for our purposes we only need a polynomial decay. The following result is due to Y. Martel.

\begin{lem}[$L^2$-polynomial decay on the right the soliton solution]\label{izq}~

There exist $K,\tilde x_0>0$ large enough but independent of $\ve$, such that for all $t\geq T_0$ and for all $x_0\geq \tilde x_0$
$$
\int_\R  (x- x_0)_+^{2} \tilde z^2(t,  x+ \tilde \rho_2(t)) dx \leq K,
$$
where $x_+ := \max \{x,0\}$.
\end{lem}

\begin{proof}
Take $x_0>0$, $t_0,t \geq t_1$ and define 
$$
\hat I_{t_0,x_0} (t) := \int_\R \tilde z^2(t,x) \phi( \tilde y) dx; \quad \tilde y :=x- (\tilde \rho_2(t_0) + \tilde \sigma(t-t_0) + x_0),
$$ 
and
$$
\hat J_{t_0,x_0} (t) := \int_\R \tilde z_x^2(t,x) \phi( \tilde y) dx.
$$
Here $\phi$ is the cut-off function defined in (\ref{phiphi}), and $\tilde \sigma$ is a fixed constant satisfying $\tilde \sigma  > 2(c_\infty(\la)-\la) $. 
First of all we claim that there exists $K>0$ such that (for simplicity we omit the dependence if no confusion is present)
\be\label{DT1}
| \partial_t \hat I_{t_0,x_0}(t)|  \leq  K\int_\R (\tilde z_x^2+\tilde z^2)[\phi' + \ve a'(\ve x)\phi ] dx + K \|\tilde z(t)\|_{H^1(\R)}e^{- \ve(t- t_0) /K}e^{- \ve x_0/K},
\ee
and
\be\label{DT2}
|\partial_t \hat J_{t_0,x_0}(t)|  \leq  K\int_\R (\tilde z_{xx}^2 + \tilde z_x^2+\tilde z^2)[\phi' + \ve a'(\ve x)\phi ] dx +K \| \tilde z(t)\|_{H^2(\R)}e^{- \ve(t- t_0) /K}e^{-\ve x_0/K}.
\ee
Indeed, these estimates are proved in the same way as in Lemma \ref{VL} and Appendix \ref{SVL}. For the sake of brevity we skip the details.

From Proposition \ref{Tmp1} and the exponential decay of $z$ we have that both right-hand sides in (\ref{DT1})-(\ref{DT2}) are integrable between $t_0$ and $+\infty$. We get
\be\label{A}
\hat I_{t_0,x_0}(t_0) \leq K\int_{t_0}^{+\infty}\!\!\! \int_\R (\tilde z_x^2+ \tilde z^2)[\phi' + \ve a'(\ve x)\phi ] dx dt + K\ve^{-1} \sup_{t\geq t_0}\| \tilde z(t)\|_{H^1(\R)}e^{- \ve x_0/K}.
\ee
In the same line, we have
\be\label{Bq}
\hat J_{t_0,x_0}(t_0) \leq K\int_{t_0}^{+\infty}\!\!\! \int_\R (\tilde z_{xx}^2 + \tilde z_x^2+\tilde z^2)[\phi' + \ve a'(\ve x)\phi] dx dt + K\ve^{-1} \sup_{t\geq t_0}\|\tilde z(t)\|_{H^2(\R)}e^{- \ve x_0/K}.
\ee
Note that both quantities above are \textbf{integrable} with respect to $x_0$. 

Let us denote $\xi_0(\tilde y) := \phi (\tilde y)$, and $\xi_j(\tilde y) := \int_{-\infty}^{\tilde y} \xi_{j-1} (s) ds,$ for $j=1,2$. Recall that $\xi_j$ are positive and increasing functions on $\R$, with $\xi_j(\tilde y)\to 0$ as $\tilde y \to -\infty$, and $\xi_j(\tilde y)- \tilde y^j \to 0$ as $\tilde y \to + \infty$. 
Integrating between $x_0$ and $+\infty$ in (\ref{A}), and using Fubini's theorem we obtain
\be\label{linear}
\int_\R \xi_1(\tilde y(t_0)) \tilde z^2(t_0)  \leq  K\int_{t_0}^{+\infty}\!\!\! \int_\R (\tilde z_x^2+\tilde z^2)[\xi_0  + \ve a'(\ve x)\xi_1]  +  K \ve^{-2} \sup_{t\geq t_0}\| \tilde z(t)\|_{H^1(\R)}e^{- \ve x_0/K},
\ee
and similarly, from (\ref{Bq})
\be\label{linearen}
\int_\R \xi_1(\tilde y(t_0)) \tilde z_x^2(t_0,x) dx  \leq  K \ve^{-3} e^{-2\ve\ga t_0} + K\ve^{-2} \sup_{t\geq t_0}\| \tilde z(t)\|_{H^2(\R)}e^{- \ve x_0/K}. 
\ee
In conclusion, thanks to the exponential decay of $\tilde z$ and  (\ref{linear})-(\ref{linearen}), we have
$$
\int_{t_0}^{+\infty}\!\!\! \int_\R \xi_1(x- \tilde\rho_2(t) -x_0) (\tilde z_x^2 + \tilde z^2)(t,x) dx dt <+\infty.
$$
Furthermore, note that for all $t\geq t_0$ one has $\tilde \rho_2(t) \leq \tilde \rho_2(t_0) + \sigma (t-  t_0) $. Thus we have
\be\label{quad0}
\int_{t_0}^{+\infty}\!\!\! \int_\R \xi_1(\tilde y(t)) (\tilde z_x^2 + \tilde z^2) (t,x) dx dt <+\infty.
\ee
In addition, an easier calculation gives us 
\be\label{quad1}
\int_{t_0}^{+\infty}\!\!\! \int_\R a'(\ve x)\xi_2(\tilde y(t)) (\tilde z_x^2+\tilde z^2)(t,x) dx dt <+\infty.
\ee
From (\ref{quad0}) and (\ref{quad1}), we can perform a second integration with respect to $x_0$ in (\ref{linear}) to obtain 
$$
\int_\R \xi_2(\tilde y(t_0)) \tilde z^2(t_0,x) dx \leq K(\ve),
$$
uniformly for $x_0$ large. Since $t_0$ is arbitrary, this last estimate gives the conclusion.
\end{proof}

\begin{lem}[$L^1$-integrability and smallness]\label{DE}~

Under the assumption (\ref{LIeta}) the following holds. There exists $K, T_0>0$ large enough such that  for all  $t\geq T_0$ one has $u(t, \cdot +\tilde \rho_2(t)) \in L^1(\R)$. Moreover,  
\be\label{small}
 \abs{\int_\R  z (t)}\leq \frac 1{100}.
\ee
Finally, from the $L^1$ conservation law (\ref{L1}), we have $u(t)\in L^1(\R)$ for all $t\in\R$ and
\be\label{L1eq}
\int_\R u(t) =\int_\R Q.
\ee
\end{lem}
\begin{proof}
Let $x_0 \geq \tilde x_0$ to be fixed below. First of all, note that if $|x| \geq x_0$ we have $2^{-1/(m-1)} Q_{c_\infty} (x)  \leq K e^{-\sqrt{c_\infty} |x|}$. In particular, since $\tilde z (t, x+\tilde \rho_2(t)) =u(t, x+\tilde \rho_2(t)) - 2^{-1/(m-1)} Q_{c_\infty} (x)$, by using Lemma \ref{der} and the stability bound (\ref{S}), in addition to a Galiardo-Nirenberg type inequality, we get 
\bee
| \tilde z (t, x +\tilde \rho_2(t)) | & \leq & K\| \tilde z (t, \cdot +\tilde \rho_2(t))\|_{L^2(y\geq x)}^{\frac 12}\| \tilde z_y (t, \cdot + \tilde \rho_2(t))\|_{L^2(\R)}^{\frac 12} \\
& \leq & K \ve^{1/4}  e^{x /K}, 
\eee
for all $x \leq - x_0$. 

On the other hand, inside the interval $[ - x_0, x_0]$ one has 
$$
\int_{[- x_0, x_0]} \tilde z(t, x+\tilde \rho_2(t)) \leq K x_0^{1/2} \|\tilde z(t, x+\tilde \rho_2(t))\|_{L^2(\R)}^{1/2} \leq K x_0^{1/2}\ve^{1/4}.
$$
 The case $x \geq x_0$ requires more care. From Lemma \ref{izq} and the Cauchy-Schwarz inequality, we have (for clarity we drop the dependence on $x+\tilde \rho_2(t)$)
$$
\abs{\int_{x\geq x_0} z(t) } \leq \frac K{(x_0-\tilde x_0)^{1/2-}} \big[\int_{x\geq  x_0} (1 + (x- \tilde x_0)^2) z^2(t) \big]^{1/2} \leq \frac K {x_0^{1/2-}},
$$
for $x_0$ large enough, independent of $\ve$. From the above estimates we finally obtain the smallness condition (\ref{small}). 

The final assertion, namely $u(t)\in L^1(\R)$ for all $t\in \R$, is a consequence of Proposition \ref{Cauchy}. It is clear that from this last fact (\ref{L1}) remains constant for all time and (\ref{L1eq}) holds.   
The proof is now complete. 
\end{proof}

\subsubsection{Conclusion of the proof} From the above lemma we can use (\ref{L1eq}) to get the desired contradiction. Indeed, from (\ref{cplus}) and Appendix \ref{IdQ} we have
$$
\lim_{t\to +\infty} \int_\R  z (t) =  [ 1 - \frac {(c^+)^{\theta -\frac 14}}{2^{1/(m-1)}} ] \int_\R Q = (1- \kappa_m) \int_\R Q \neq 0; \quad \kappa_m:= \frac{c_\infty^{\frac{3-m}{2(m-1)} }}{2^{1/(m-1)}},
$$
a contradiction with (\ref{small}). Indeed, for the cases $m=3, 4$ we easily have $1-\kappa_m >\frac 1{10}$. In the case $m=2$ we have $\kappa_2 = \frac 12 c_\infty^{1/2};$ but from (\ref{boundCp}) we know that $c_\infty \leq 2^{\frac 43}$. Thus we have $1-\kappa_m >\frac 1{10}$ for every $m$. In concluding,
$$
\abs{\lim_{t\to +\infty} \int_\R z(t)} \geq \frac 1{10}\int_\R Q,
$$
a contradiction with (\ref{small}). This finishes the proof of (\ref{MTcor1}).
\end{proof}

\bigskip

\appendix

\section{Proof of Theorem \ref{Tm1}}\label{Thm0}

In this section we sketch the proof of Theorem \ref{Tm1}, for the complete proof, see \cite{Martel}.

Let $(T_n)_{n\in \N}\subseteq \R$ an increasing sequence with $T_n\geq \frac 1{2}T_\ve$ for all $n$ and $\lim_{n\to +\infty} T_n =+\infty$. For notational simplicity we denote by $\tilde T_n$ the sequence $(1-\la)T_n$. Consider $u_n(t)$ the solution of the following Cauchy problem  
\be\label{CPn}
\begin{cases}
(u_n)_t + ((u_n)_{xx}-\la u_n +a_\ve u_n^m)_x=0, \quad \hbox{ in }\; \R_t\times \R_x, \\
u_n(-T_n)=Q(\cdot -\tilde T_n).
\end{cases}
\ee
In other words, $u_n$ is a solution of (aKdV) that at time $t=-T_n$ corresponds to the soliton $Q(\cdot -\tilde T_n)$. It is clear that $Q(\cdot -\tilde T_n)\in H^s(\R)$ for every $s\geq 0$; moreover, there exists a uniform constant $C=C(s)>0$ such that
$$
\|Q(\cdot -\tilde T_n)\|_{H^s(\R)}\leq C.
$$

\noindent
According to Proposition \ref{Cauchy} and Proposition \ref{GWP0}, we have that $u_n$ is locally well-defined in time, and global for positive times in $H^1(\R)$. Let $I_n$ be its maximal interval of existence.

Following \cite{Martel}, the next step is to establish uniform estimates starting from a fixed time $t =-\frac 1{2} T_\ve <0$ large enough such that the soliton is sufficiently away from the region where the influence of the potential $a_\ve$ is present. This is the purpose of the following

\begin{prop}[Uniform estimates in $H^s$ for large times, see also \cite{Martel}]\label{Ue1}~Ê

There exist constants  $K,\ga>0$ and $\ve_0>0$ small enough such that for all $0<\ve <\ve_0$ and for all $n\in \N$ we have 
$$
 [-T_n, -\frac 1{2}T_\ve] \subseteq I_n,  \quad (\hbox{namely } u_n\in C([-T_n, -\frac 1{2}T_\ve],H^s(\R)) ),
$$
and for all $t\in [-T_n, -\frac 1{2}T_\ve]$,
\be\label{Ue2}
\|u_n(t) -Q(\cdot -(1-\la)t) \|_{H^s(\R)} \leq K\ve^{-1} e^{\ga \ve t}.
\ee
In particular, there exists a constant $C_s>0$ such that for all $t\in [-T_n, -\frac 1{2}T_\ve]$
\be\label{Ue3}
\|u_n(t) \|_{H^s(\R)} \leq C_s.
\ee
\end{prop}
Using Proposition \ref{Ue1} we will obtain the existence of a \emph{critical element} $u_{0,*}\in H^s(\R)$, with several good compact properties, non dispersive and uniformly close to the desired soliton. 

Indeed, consider the sequence $(u_n(-\frac 1{2}T_\ve))_{n\in \N}\subseteq H^s(\R)$. We claim the following result.

\begin{lem}[Compactness property]\label{CompProp}~

Given any number $ \delta >0$, there exist $\ve_0>0$ and a constant $K_0>0$ large enough such that for all $0<\ve<\ve_0$ and for all $n\in \N$, 
\be\label{CP1}
\int_{|x| >K_0} u_n^2(-\frac 1{2}T_\ve) <\delta.
\ee
\end{lem}

\begin{proof}
The proof is by now a standard result. See \cite{Martel} for the details.
\end{proof}

Let us come back to the proof of Theorem \ref{Tm1}. From (\ref{Ue3}) we have that 
$$
\|u_n(-T_\ve/2)\|_{H^1(\R)}\leq  C_0,
$$
independent of $n$. Thus, up to a subsequence we may suppose $u_n(-\frac 1{2} T_\ve ) \rightharpoonup u_{*,0} $ in the $H^1(\R)$ weak sense, and $u_n(-\frac 1{2} T_\ve ) \to u_{*,0}$ in $L^2_{loc}(\R)$, as $n\to +\infty$. In addition, from (\ref{CP1}) we have the strong convergence in $L^2(\R)$. Moreover, from interpolation and the bound (\ref{Ue3}) we have the strong convergence in $H^s(\R)$ for any $s\geq 1$.

Let $u_*=u_*(t)$ be the solution of (\ref{gKdV0}) with initial data $u_*(-\frac 1{2} T_\ve ) = u_{*,0}$. From Proposition \ref{Cauchy}  we have $u_*\in C(I,H^s(\R))$, where $-\frac 1{2} T_\ve \in I$, the corresponding maximal interval of existence. Thus, using the continuous dependence of $u_n$ and $u_*$, we obtain $u_n(t) \to u_*(t)$ in $H^s(\R)$ for every $t \leq -\frac 1{2}T_\ve  \subseteq I$.  Passing to the limit in (\ref{Ue2}) we obtain for all $t\leq -\frac 1{2}T_\ve$,
$$
\|u_*(t) -Q(\cdot - (1-\la)t) \|_{H^s(\R)} \leq K\ve^{-1} e^{\ve \ga t},
$$
as desired. This finish the proof of the existence part of Theorem \ref{Tm1}. 
\subsection{Uniform $H^1$ estimates. Proof of Proposition \ref{Ue1}}
In this paragraph we explain the main steps of the proof of Proposition \ref{Ue1} in the $H^1$ case; for the general case the reader may consult \cite{Martel}. 

The first step in the proof is the following bootstrap property:

\begin{prop}[Uniform estimates with and without decay assumption]\label{Uealpha}~

Let $m=2,3$ or $4$, and $0\leq \la \leq \la_0<1$. There exist constants $K,\ga, \ve_0>0$ such that for all $0<\ve<\ve_0$ the following is true. 

\begin{enumerate}
\item Suppose $m=3$ or $m=2,4$ with $\la>0$. Then there exists $\al_0>0$ such that for all $0<\al<\al_0$, if for some $-T_{n,*} \in [-T_n, -\frac 1{2}T_\ve]$ and for all $t\in [-T_n, -T_{n,*}]$ we have
\be\label{Ue1alpha}
\|u_n(t) -Q(\cdot -(1-\la) t ) \|_{H^1(\R)}\leq 2\al,
\ee
then, for all $t\in [-T_n, -T_{n,*}]$
\be\label{Ue2alpha}
\|u_n(t) -Q(\cdot -(1-\la) t ) \|_{H^1(\R)}\leq K \ve^{-1} e^{\ve \ga t}.
\ee
\item Suppose now $m=2,4$ and $\la=0$. Then the same conclusion (\ref{Ue2alpha}) holds if for some $-T_{n,*} \in [-T_n, -\frac 1{2}T_\ve]$ and for all $t\in [-T_n, -T_{n,*}]$ one has 
\be\label{Ue1alpha24}
\|u_n(t) -Q(\cdot -(1-\la) t ) \|_{H^1(\R)}\leq 2K \ve^{-1}e^{\ve \ga t},
\ee
\end{enumerate}
\end{prop}

\begin{proof}[Proof of Proposition \ref{Ue1}, assuming the validity of Proposition \ref{Uealpha}]
We prove the first case, the second one being similar. Firstly note that from (\ref{CPn}) we have 
$$
\|u_n(-T_n) -Q(-(1-\la)T_n)\|_{H^1(\R)} =0,
$$
so there exists $t_0=t_0(n,\alpha)>0$ such that (\ref{Ue1alpha}) holds true for all $t\in [-T_n, -T_n +t_0]$. 
Now let us consider (we adopt the convention $T_{*,n}>0$)
$$
-\tilde T_{*,n}  :=   \sup \{ t\in [-T_n, -\frac 12T_\ve] \ | \   \hbox{ for all $t' \in [-T_n, t ] $, }\;  \| u_n(t') -Q( \cdot - (1-\la)t' ) \|_{H^1(\R)} \leq  2\al \}.
$$
Assume, by contradiction, that $-\tilde T_{*,n} < -\frac 12 T_\ve$. From Proposition \ref{Uealpha}, we have
$$
\| u_n(t') -Q( \cdot - (1-\la)t' ) \|_{H^1(\R)} \leq K\ve^{-1} e^{\ga\ve t} \leq \alpha ,
$$
for $\ve$ small enough (recall that $t\leq -\frac 12T_\ve =-\frac 1{2(1-\la)} \ve^{-1-\frac 1{100}}$), a contradiction with the definition of $\tilde T_{*,n}$.
\end{proof}
Now we are reduced to prove Proposition \ref{Uealpha}. 

\begin{proof}[Proof of Proposition \ref{Uealpha}]
The first step in the proof is to decompose the solution preserving a standard orthogonality condition. To obtain this fact, and without loss of generality, by taking $T_{n,*}$ even large we may suppose that for all $t\in [-T_n, -T_{n,*}]$ 
\be\label{Uer}
\|u_n(t) -Q(\cdot -(1-\la) t -r_n(t) ) \|_{H^1(\R)}\leq 2\al,
\ee
for all smooth $r_n=r_n(t)$ satisfying $r_n(-T_n)=0$ and $\abs{r_n'(t)}\leq \frac 1{t^2}$. A posteriori we will prove that this condition can be improved and extended to any time $t\in [-T_n, -\frac 12T_\ve ]$.

\medskip

For notational simplicity, in what follows we will drop the index $n$ on $-T_{*,n}$ and $u_n$, if no confusion is present.

\begin{lem}[Modulation]\label{FM}~

There exist $K,\ga, \ve_0>0$ and a unique $C^1$ function $\rho_0:[ -T_n, -T_*]\to \R $ such that for all $0<\ve<\ve_0$ the function $z$ defined by
\be\label{Ortho0}
z(t,x) := u(t,x) - R(t,x) ; \quad R(t,x) := Q(x-(1-\la)t -\rho_0(t))
\ee
satisfies for all $t\in [-T_n, -T_*]$,
\be\label{Ortho01}
\int_\R z(t,x) R_x(t,x) dx=0, \quad \|z(t)\|_{H^1(\R)} \leq K \alpha, 
 \quad \rho_0(-T_n) =0.
\ee
Moreover, $z$ satisfies the following modified gKdV equation,
\be\label{EqZ0}
z_t + \big\{ z_{xx} -\la z + a_\ve[(R+z)^m -R^m]  +  (1 - a_\ve)R^m \big\}_x - \rho_0' (t)R_x=0,
\ee
and
\be\label{rho0}
|\rho_0'(t)|\leq K\big[ e^{\ve \ga t} + \|z(t)\|_{H^1(\R)} + \|z(t)\|_{L^2(\R)}^2\big].
\ee
\end{lem}



\begin{proof}[Proof of Lemma \ref{FM}]

The proof of (\ref{Ortho01}) is a standard consequence of the Implicit Function Theorem, the definition of $T_*$  $(=T_{*,n})$, and the definition of $u_n(-T_n)$ given in (\ref{CPn}), see for example \cite{Martel} for a detailed proof. 
Similarly, the proof of (\ref{EqZ0}) follows after a simple computation.

Now we deal with (\ref{rho0}). Taking time derivative in  (\ref{Ortho0}) and using (\ref{EqZ0}), we get 
\bee
0 & = & \int_\R z_t R_x - (1-\la+\rho_0') \int_\R z  R_{xx}  \\
& =& \int_\R \big\{ z_{xx} -z + a_\ve [(R+z)^m -R^m]  +  (1 - a_\ve ) R^m \big\} R_{xx}  + \rho_0' \int_\R R_x(R_x +z_x).
\eee
First of all, note that
$$
\int_\R R_x(R_x +z_x) = \int_\R Q'^2 + O(\|z(t)\|_{L^2(\R)}).
$$
On the other hand, from (\ref{ahyp}), (\ref{Ortho01}), the uniform bound on $\rho_0'(t)$ in the definition of $T_*$  and the exponential decay of $R$, we have
\be\label{dec1}
\abs{\int_\R (1 - a_\ve )R^m R_{xx} } \leq  K e^{\ve \ga t}.
\ee
Indeed, first note that from (\ref{Uer}), by integrating between $-T_n $ and $t$ and using (\ref{Ortho01}) we get 
$$
\rho_0(t) \leq - \frac 1{T_n} -\frac 1t \leq \frac 2{T_\ve}\leq K\ve^{1+\frac 1{100}} .
$$
Thus $t +\rho_0(t) \leq t  +K \ve^{1+\frac 1{100}} \leq \frac 9{10}t .$ Therefore, by possibly redefining $\ga$, we have from (\ref{ahyp}),
\bee
\abs{\int_\R (1 - a_\ve )R^m R_{xx} }& \leq & K\int_{-\infty}^0e^{\ga \ve x}e^{-(m+1)|x-(t+\rho_0(t) )|} dx \\
& & \qquad +  Ke^{(m+1)(t+\rho_0(t))} \int_0^\infty e^{-(m+1)x} dx  \\
&  \leq & K \exp \big[  \ga\ve (t+\rho_0(t) )\big] + K  \exp \big[  \ga(m+1) (t+\rho_0(t) ) \big] \leq  Ke^{\ga\ve t}.
\eee
Finally,
$$
  \int_\R R_{xx} \big\{ z_{xx} -z + a_\ve [(R+z)^m -R^m]  \big\}   
  = O(\|z(t)\|_{L^2(\R)} + \|z(t)\|_{L^2(\R)}^2).
$$
Collecting the above estimates we obtain (\ref{rho0}). 
\end{proof}

\subsubsection{Almost conservation of mass and energy} Now let us recall that from remark \ref{MM} the modified mass defined in (\ref{tM}) satisfies 
\be\label{Phi}
\tilde M[u](t)\leq \tilde M[u](-T_n).
\ee
for all $-T_n\leq t\leq -\frac 12T_\ve$. Moreover, in the case $m=2,4$ and $\la =0$, since (\ref{dMa}) and (\ref{Ue1alpha24}) hold, there exist $K,\ga>0$ such that 
\be\label{Phi1}
M[u](t) \leq M[u](-T_n) + K \ve e^{\ga\ve t},
\ee
for $\ve$ small enough. By extending the definition of $\tilde M[u]$ to the latter case, we have almost conservation of mass, with exponential loss for all cases. 

Similarly, note that in the region considered the soliton $R(t)$ is an almost solution of (\ref{aKdV}), in particular it must conserve mass $\tilde M$ (\ref{tM}) and the energy $E_a$ (\ref{Ea}), at least for large negative time. Indeed, arguing as in Lemma \ref{C2} (but with easier proof), one has 
\be\label{dE01a}
E_a[R](-T_n) - E_a[R](t) + (1-\la)\big[ \tilde M[R](-T_n)  - \tilde M[R](t)\big]  \leq K e^{\ga\ve t }. 
\ee
for some constant $K>0$ and all time $t\in [-T_n, T_*]$

The next step is the use the energy conservation law to provide a control of the $R(t)$ direction (note that $R(t)$ is a essential direction to control in order to obtain some coercivity properties, see Lemma \ref{surL}). Following e.g. Lemma \ref{Qdir}, one has  
\be\label{CRD}
\abs{\int_\R R z(t)} \leq \frac K{1-\la} \Big[ e^{\ga \ve t } +  \|z(t)\|_{L^2(\R)}^2 +  e^{\ga\ve t}\|z(t)\|_{L^2(\R)}\Big].
\ee
for some constants $K,\ga>0$, independent of $\ve$.

Now, consider $E_a[u]$ and $\tilde M[u]$ the energy and mass defined in (\ref{Ea})-(\ref{tM}). Then one has
\be\label{Expans1}
E_a[u](t) + (1-\la) \tilde M[u](t) = E_a[R](t) + (1-\la) \tilde M[R](t) - \int_\R z(a_\ve -1)R^m  + \mathcal F_0(t),
\ee
where $\mathcal F_0$ is the quadratic functional
$$
 \mathcal F_0(t)  :=  \frac 12\int_\R (z_x^2 + \la z^2)  + (1-\la) \tilde M[z]- \frac 1{m+1} \int_\R  a_\ve[ (R+z)^{m+1} -R^{m+1} -(m+1)R^mz].
$$
In addition, for any $t\in [-T_n, -T_*]$,
\be\label{dec3}
\abs{ \int_\R z(a_\ve -1)R^m }\leq K e^{\ga \ve  t}  \|z(t)\|_{L^2(\R)}. 
\ee
The proof of this identity is essentially an expansion of the energy-mass functional using the relation $u(t) =R(t) + z(t)$. 
The proof of (\ref{dec3}) is similar to (\ref{dec1}).

On the other hand, the functional $\mathcal F_0(t)$ above mentioned enjoys the following coercivity property: there exist $K,\la_0>0$ independent of $\ve$ such that for every $t\in [-T_n, -T_*]$
\be\label{F00}
\mathcal F_0(t) \geq \la_0 \|z(t)\|_{H^1(\R)}^2 -\abs{\int_\R R(t)  z(t)}^2 - K e^{\ga\ve t} \|z(t)\|_{L^2(\R)}^2 - K \|z(t)\|_{L^2(\R)}^3. 
\ee
This bound is simply a consequence of the inequality $\la + (1-\la) a_\ve^{1/m}(x) \geq 1$, (\ref{Ortho01}) and Lemma \ref{surL}.

\subsubsection{End of proof of Proposition \ref{Uealpha}}
Now by using (\ref{Expans1}), (\ref{F00}), and the estimates (\ref{Phi})-(\ref{Phi1}) and (\ref{CRD}) we finally get (\ref{Ue2alpha}). Indeed, note that 
$$
 E_a[u](t) -E_a[u](-T_n) + (1-\la)[\tilde M[u](t)- \tilde M[u](-T_n)]  \leq K e^{\ve\ga t}.
$$
On the other hand, from (\ref{Expans1}) and (\ref{Ortho01}),
\bee
& & E_a[u](t) -E_a[u](-T_n)  + (1-\la) [\tilde M[u](t) - \tilde M[u](-T_n)] \qquad  \\
& & \qquad \qquad\qquad \geq  \mathcal  F_0(t) -   Ke^{\ga \ve t} -Ke^{\ga \ve t}\|z(t)\|_{L^2(\R)}, 
\eee
since  $z(-T_n) =0$ and $\mathcal F_0(-T_n) =0$. Finally, from (\ref{F00}) and \ref{CRD} we get
$$
\|z(t)\|_{H^1(\R)} \leq K e^{\ga\ve t}.
$$
Plugging this estimate in (\ref{rho0}), we obtain that $\abs{\rho_0'(t)} \leq K e^{\ga\ve t},$
and thus after integration we get the final uniform estimate (\ref{Ue2alpha}) for the $H^1$-case. Note that we have also improved the estimate on $\rho_0'(t)$ assumed in (\ref{Uer}). This finishes the proof.
\end{proof}

\subsection{Proof of Uniqueness}\label{Uni}

First of all let us recall that the solution $u$ above constructed is in $C(\R, H^s(\R))$ for any $s\geq 1$, and satisfies the exponential decay (\ref{minusTe}). Moreover, every solution converging to a soliton satisfies this property.

\begin{prop}[Exponential decay, see also \cite{Martel}]\label{Uni1}~

Let $m=3$, or $m=2,4$ with $0<\la\leq\la_0$. Let $v=v(t)$ a $C(\R,H^1(\R))$ solution of (\ref{gKdV0}) satisfying
$$
\lim_{t\to -\infty} \|v(t) -Q(\cdot -(1-\la)t)\|_{H^1(\R)} =0.
$$
Then there exist $K,\ga, \ve_0>0$ such that for every $t\leq -T_\ve$ we have
$$
\|v(t) -Q(\cdot -(1-\la) t) \|_{H^1(\R)}\leq K\ve^{-1} e^{\ga\ve t}.
$$
\end{prop}

\begin{proof}
Fix $\alpha>0$ small. Let $\ve_0=\ve_0(\al)>0$ small enough such that for all $\ve\leq \ve_0$ and $t\leq -T_\ve$
$$
\|v(t) -Q(\cdot -(1-\la)t)\|_{H^1(\R)} \leq \alpha.
$$
Possibly choosing $\ve_0$ even smaller, we can apply the arguments of Proposition \ref{Uealpha} to the function $v(t)$ on the interval $(-\infty, -\frac 1{2}T_\ve]$ to obtain the desired result. Indeed, we follow Proposition \ref{Uealpha}, part (1). Lemma \ref{FM} holds for $z(t) := v(t)-Q(\cdot -(1-\la)t -\rho_0(t))$ and $t\leq -\frac 12 T_\ve$, but now we have, by hypothesis,
$$
\lim_{t\to -\infty} |\rho_0(t)| + \|z(t)\|_{H^1(\R)} =0;
$$
and therefore $\lim_{t\to -\infty} \mathcal F_0 (t) =0$ (this can be done in a rigorous form by taking a sequence $t_n \to -\infty$ large enough and such that $\|v(t_n) -Q(\cdot -(1-\la)t_n)\|_{H^1(\R)} \leq \frac 1n.$ With this choice one has $|\rho_{0,n}(t_n)| + \|z_n(t_n)\|_{H^1(\R)} \to 0$, independent of $\ve$. Rerunning as usual the proof in the interval $[t_n, t]$ and finally taking the limit $n\to +\infty$, we obtain the conclusion.) The rest of the proof is direct.
\end{proof}

\begin{rem}
For the proof of the above result a key ingredient is the monotony of mass from remark \ref{MM}; this property apparently does not hold in the cases $\la=0$, $m=2,4$.
\end{rem}

Now we are ready to prove the uniqueness part.

\begin{proof}[Sketch of proof of uniqueness]
Let $w(t) := v(t) -u(t)$. Then $w(t) \in H^1(\R)$ and satisfies the equation
\be\label{W}
\begin{cases}
w_t + (w_{xx} -\la w+ a_\ve [ (u+ w)^m -u^m ])_x =0, \quad \hbox{ in } \; \R_t \times \R_x,\\
\|w(t)\|_{H^1(\R)}\leq K\ve^{-1} e^{\ga \ve t }\quad \hbox{ for all  } t\leq -\frac 1{2}T_\ve.
\end{cases}
\ee
The idea is to prove that $w(t)\equiv 0$ for all $t\in \R$. For this purpose, one defines the second order functional
$$
\mathcal F_0 (t) := \frac 12\int_\R w_x^2 + \frac 12 \int_\R w^2 -\frac 1{m+1}\int_\R a_\ve(x)[ (u+w)^{m+1} -u^{m+1} -(m+1) u^m w].
$$
It is easy to verify that 
\begin{enumerate}
\item Lower bound. There exists $K>0$ such that for all $t\leq -\frac 1{2}T_\ve$,
$$
\mathcal F_0(t) \geq \frac 12\int_\R (w_x^2 + w^2 -mQ^{m-1}w^2)(t) - K\ve^{-1} e^{\ga\ve t}\sup_{t'\leq t} \|w(t')\|_{H^1(\R)}^2.
$$
\item Upper bound. There exists $K,\ga>0$ such that 
$$
\mathcal F_0(t) \leq K\ve^{-2} e^{\ga\ve t} \sup_{t'\leq t}\|w(t')\|_{H^1(\R)}^2.
$$
\end{enumerate} 
These estimates are proved similarly to the proof of Lemma \ref{Ka}. However, this functional is not coercive; so in order to obtain a satisfactory lower bound, one has to modify the function $w$ in $(-\infty, -\frac 1{2}T_\ve]$ as follows. Let 
$$
\tilde w(t) := w(t) +b(t) Q'(\cdot -t), \quad b(t):= \frac{\int_\R w(t)Q'(\cdot -t)}{\int_\R Q'^2},
$$
This modified function satisfies
\begin{enumerate}
\item Orthogonality to the $Q'$ direction:
$ 
\displaystyle{\int_\R \tilde w(t) Q'(\cdot -t) =0.}
$
\item Equivalence. There exists $C_1,C_2>0$ independent of $\ve$ such that
$$
C_1 \|w(t)\|_{H^1(\R)} \leq \|\tilde w(t)\|_{H^1(\R)} + \abs{b(t)} \leq C_2 \|w(t)\|_{H^1(\R)}.
$$
Moreover,
$$
\frac 12\int_\R (w_x^2 + w^2 -mQ^{m-1}w^2)(t)  = \frac 12\int_\R (\tilde w_x^2 + \tilde w^2 -mQ^{m-1}\tilde w^2)(t) + O(e^{-\ve\ga |t|}). 
$$
\item Control on the $Q$ direction:
$$
\abs{\int_\R \tilde w(t) Q(\cdot -t)} \leq K\ve^{-1} e^{\ve \ga t} \sup_{t'\leq t} \|w(t')\|_{H^1(\R)}. 
$$
This property is proved similarly to the proof of (\ref{c2rho2}): We use the fact that variation in time of the above quantity is of quadratic order on $\tilde w$.
\item Coercivity. There exists $\la>0$ independent of $t$ such that
$$
\frac 12\int_\R (\tilde w_x^2 + \tilde w^2 -mQ^{m-1}\tilde w^2)(t) \geq \la \| \tilde w(t)\|_{H^1(\R)}^2 -K\abs{\int_\R \tilde w(t) Q(\cdot -t)}^2.
$$
\item Sharp control. From the equivalence $w$-$\tilde w$ and the coercivity property we obtain
\be\label{Sc}
\|\tilde w(t)\|_{H^1(\R)} +\ve |b(t)| \leq  K \ve^{-2} e^{\ve \ga t/2} \sup_{t'\leq t} \|w(t')\|_{H^1(\R)}. 
\ee
Note that the bound on $b(t)$ is proved similarly to (\ref{rho2c2}).
\end{enumerate}

The proof of these affirmations follows closely the argument of Proposition 6 in \cite{Martel}, with easier proofs. 
Finally, from (\ref{Sc}) we have for $\ve $ small enough and $t\leq -\frac 12T_\ve$,
$$
\|w(t)\|_{H^1(\R)} \leq  K\ve^{-2} e^{\ve\ga t} \sup_{t'\leq t} \|w(t')\|_{H^1(\R)} <\frac 12 \sup_{t'\leq t} \|w(t')\|_{H^1(\R)} .
$$
This inequality implies $w\equiv 0$, and in conclusion the uniqueness. 
\end{proof}

\bigskip

\section{Proof of Proposition \ref{prop:decomp}}\label{AppA}

The proof is similar to Proposition 2.2 in \cite{MMcol2} and Appendix in \cite{MMcol1}. 
\begin{proof}
First of all, we recall the error term $S[\tilde u]$ introduced in (\ref{2.2bis}), subsection \ref{sec:2-1}.

We easily verify that 
\begin{equation}\label{eq:sion}
S[\tilde u]= \bf I +II +III,
\end{equation}
where (we omit the dependence on $t,x$) 
\be\label{eq:sion1}
{\bf I} := S[R], \quad {\bf II} = {\bf II}(w) := w_t  +  (w_{xx} -\la w + m\ a_\ve R^{m-1} w)_x , 
\ee
and
\be\label{eq:sion2}
{\bf III} : = \left\{ a_\ve [(R + w)^m - R^m - mR^{m-1}w ]\right\}_x.
\ee

In the next lemmas, we expand the terms in \eqref{eq:sion}.
\begin{lem}\label{lem:SQ}~
Suppose $m=2,3$ or $4$. We have
\be\label{eq:SQ}
{\bf I}= \ve F_1(\ve t; y) + \frac{\ve^2 a'' }{2\tilde a^m}(y^2Q_c^m)_y + \ve^3 f_I(\ve t)F^I_c(y),
\ee
where 
$$
F_1(\ve t; y) :=  \frac{ c'}{\tilde a}\Lambda Q_c-  \frac {\tilde a'}{\tilde a^2} (c-\la) Q_c +  \frac{ a' }{\tilde a^m}(yQ_c^m)_y  \in \mathcal Y,
$$
and $\abs{f_I(\ve t)}\leq K$, $F^I_c\in \mathcal Y$. Finally, for every $t\in [-T_\ve, T_\ve]$
$$
\|\ve^3 f_I(\ve t) F^I_c(y) \|_{H^2(\R)} \leq K\ve^3.
$$
\end{lem}

\begin{proof}[Proof of Lemma \ref{lem:SQ}.]
Recall that $\tilde a := a^{\frac 1{m-1}}$ and
$$
R(t,x) = \frac{Q_{c(\ve t)} (y)}{\tilde a(\ve \rho(t))}, \quad y= x-\rho(t), \quad \partial_t \rho(t) = c(\ve t)-\la. 
$$
Thus we  have 
\bee
{\bf I} & = & R_t + (R_{xx}  -\la R+ a_\ve R^m )_x\\
& =&  \frac{\ve c'}{\tilde a}\Lambda Q_c - \frac{(c-\la)}{\tilde a}  Q_c' - \ve \frac {\tilde a' (c-\la)}{\tilde a^2} Q_c + \frac 1{\tilde a} Q_c^{(3)} -\frac{\la}{\tilde a}  Q_c' +  \frac{1}{\tilde a^m}( a(\ve x)Q_c^m)_x.
\eee
Note that via a Taylor expansion,
$$
( a(\ve x)Q_c^m)_x = a(\ve \rho) (Q_c^m)_x + \ve a'(\ve \rho) (yQ_c^m)_x + \frac 12 \ve^2 a''(\ve \rho) (y^2 Q_c^m)_x + O_{H^2(\R)}(\ve^3).
$$
Therefore,
\bee
{\bf I} &= & \frac{\ve c'}{\tilde a}\Lambda Q_c - \frac{(c-\la)}{\tilde a}  Q_c' - \frac{\ve}{m-1} \frac {a' (c-\la)}{\tilde a^m} Q_c + \frac 1{\tilde a} Q_c^{(3)} -\frac{\la}{\tilde a}  Q_c' + \frac 1{\tilde a} (Q_c^m)' + \frac{\ve a' }{\tilde a^m}(yQ_c^m)_x  \\
& & + \frac{\ve^2 a'' }{2\tilde a^m}(y^2Q_c^m)_x  + \ve^3 f_I(\ve t)F^I_c(y) \\
& = & \frac 1{\tilde a} (Q_c''- cQ_c + Q_c^m)'  +  \frac{\ve c'}{\tilde a}\Lambda Q_c 
- \ve \frac {\tilde a'}{\tilde a^2} (c-\la) Q_c+  \frac{\ve a' }{\tilde a^m}(yQ_c^m)_y \\
& & +  \frac{\ve^2 a'' }{2\tilde a^m}(y^2Q_c^m)_y + \ve^3 f_I(\ve t)F^I_c(y) \\
&  = & \ve \big[ \frac{c'}{\tilde a}\Lambda Q_c 
-  \frac {\tilde a'}{\tilde a^2} (c-\la) Q_c +  \frac{ a' }{\tilde a^m}(yQ_c^m)_y  \big] +  \frac{\ve^2 a'' }{2\tilde a^m}(y^2Q_c^m)_y + \ve^3 f_{I}(\ve t)F^{I}_c(y).
\eee
Moreover $\abs{f_I(\ve t)}\leq K$, $F^I_c(y)\in \mathcal Y$ and
$$
\|\ve^3 f_I(\ve t) F^I_c(y) \|_{H^2(\R)} \leq K\ve^3.
$$
This finishes the proof.
\end{proof}

\begin{lem}[Decomposition of {\bf II}]\label{lem:dSKdVw}~
We have
\bee
{\bf II} & = & -\ve (\mathcal{L} A_c)_y(\ve t; y) + \ve^2 [(A_c)_t + c'(\ve t)\Lambda A_c](\ve t; y) \\
& & +m \ve^2  \frac{a'(\ve \rho)}{a(\ve \rho)} ( yQ_c^{m-1}(y)A_c(\ve t; y) )_y + \ve^3  F_c^{\bf II}(\ve t; y).
\eee
with $F_c^{\bf II}(\ve t; \cdot) \in \mathcal Y$, uniformly in time. In addition, suppose {\bf (IP)} holds for $A_c$.  Then
$$
\|\ve^3 F_c^{\bf II}(\ve t; y)\|_{H^2(\R)}\leq K \ve^3 e^{-\ga\ve|t|}.
$$
\end{lem}

\begin{proof}
We compute:
\bee
{\bf II}& =& \ve (A_c(\ve t ; y))_t + \ve \big[ (A_c)_{yy}(\ve t; y) - \la A_c (\ve t; y)+  \frac{a_\ve}{a(\ve \rho)} mQ_c^{m-1}(y)A_c(\ve t; y) \big]_x \\
& = &  - \ve (\mathcal L A_c)_y(\ve t; y)  +  \ve^2 (A_c)_t(\ve t; y) + \ve^2 c'(\ve t) \Lambda A_c(\ve t, y)    \\
& & + m \ve^2  \frac{a'(\ve \rho)}{a(\ve \rho)} ( yQ_c^{m-1}(y)A_c(\ve t; y) )_y +  \ve^3 F_c^{\bf II}(\ve t; y), 
\eee
where  $F_c^{\bf II}(\ve t; y) =O(y^2Q_c^{m-1}(y)A_c(\ve t; y) )_y) \in \mathcal Y$ and thus, thanks to the {\bf (IP)} property,
$$
\|\ve^3 F_c^{\bf II}(\ve t; y)\|_{H^2(\R)}\leq K \ve^3 e^{-\ga\ve|t|}.
$$
This concludes the proof.
\end{proof}

\begin{lem}[Decomposition of {\bf III}]\label{lem:SintIII}~
  Suppose {\bf (IP)} holds for $A_c$. Then we have
$$
{\bf III} = \ve^3  a'(\ve x)[ \ve^{m-2}A_c^m(\ve t; y) + \tilde F^{\bf III}_c(\ve t; y) ] + \ve^2 a_\ve G^{\bf III}_c(\ve t; y),
$$
with $\tilde F^{\bf III}_c(\ve t; \cdot), G^{\bf III}_c(\ve t; \cdot )\in \mathcal Y$, uniformly for every $t\in [-T_\ve, T_\ve]$. Moreover, we have the estimate
\be\label{IIIH2}
\|{\bf III}\|_{H^2(\R)} \leq K \ve^2e^{-\ga\ve|t|},
\ee
for every $t\in [-T_\ve, T_\ve]$.
\end{lem}

\begin{proof}
Define ${\bf \tilde{III} } := a_\ve[ (R+ w)^m - R^m - m R^{m-1} w]$. We consider separate cases. First, note that for $m=2$, ${\bf \tilde{III} }  =  a_\ve w^2 =  \ve^2  a_\ve A_c^2;$ thus taking derivative
$$
{\bf {III} }  =  \ve^3 a'(\ve x)  A_c^2 + \ve^2  a_\ve (A_c^2)'.
$$
Note that $ (A_c^2)'\in \mathcal Y$ because {\bf (IP)} property holds for $A_c$.

Suppose now $m=3$. We have ${\bf \tilde{III} }  =   \ve^2 a_\ve [ 3 Q_cA_c^2  + \ve A_c^3 ].$
From this we get
$$
{\bf III }  = \ve^3 a'(\ve x) [ 3  Q_cA_c^2  + \ve  A_c^3 ] +  \ve^2 a_\ve [ 3  (Q_cA_c^2)'  + \ve (A_c^3)' ].
$$
Finally, for the case $m=4$
\bee
{\bf III } &  = & \big\{ a_\ve  \ve^2 [ 6 Q_c^2A_c^2  + 4\ve Q_cA_c^3 + \ve^2  A_c^4 ] \big\}_x \\
&  = &  \ve^3   a'(\ve x)[ 6 Q_c^2A_c^2  + 4\ve^2  Q_cA_c^3 + \ve^2 A_c^4 ]  +  \ve^2  a_\ve [ 6  (Q_c^2A_c^2)'  + 4\ve  (Q_cA_c^3)' + \ve^2 (A_c^4)' ]. 
\eee

Under the {\bf (IP)} property, for each $m=2,3$ and $4$, we can estimate ${\bf III}$ as follows
$$
\|{\bf III }\|_{H^2(\R)} \leq K \ve^2 e^{-\ga\ve|t|}. 
$$ 
\end{proof}

Now we collect the estimates from Lemmas \ref{lem:SQ}, \ref{lem:dSKdVw} and \ref{lem:SintIII}. We finally get
\bee
S[\tilde u] & = & {\bf I + II +III}\\
& = &  \ve [ F_1-(\mathcal L A_c)_y ](\ve t; y) + \ve^2 [(A_c)_t + c'(\ve t)\Lambda A_c](\ve t; y)  + O(\ve^2e^{-\ga\ve |t|}), 
\eee
provided {\bf (IP)} holds for $A_c$. 
\end{proof}

\bigskip

\section{End of Proof of Lemma \ref{CV}}\label{CV1}
In this section we will show that for all $t\in [-T_\ve, T_\ve]$ (cf. (\ref{SH2}))
\be\label{SH2b}
\|S[\tilde u](t)\|_{H^2(\R)} \leq K \ve^{\frac 32} e^{-\ga\ve|t|},
\ee
where $\tilde u $ is the modified approximate solution defined in (\ref{hatu}).

\begin{proof}[Proof of (\ref{SH2b})]
Similarly to the proof of Proposition \ref{prop:decomp} in Appendix \ref{AppA}, we claim that we can decompose
$$
S[\tilde u ] = {\bf I + \tilde{II} + \tilde{III}},
$$
(cf. the definitions in (\ref{eq:sion})-(\ref{eq:sion2})).

First of all, note that  the conclusions of Lemma \ref{lem:SQ}  in Appendix \ref{AppA} remains unchanged. In particular, (\ref{eq:SQ}) holds without any variation.

Concerning the term ${\bf \tilde{III} }$, we have the following

\begin{Cl}[Decomposition of ${\bf \tilde{III}}$ revisited]\label{8}
We have
$$
 {\bf \tilde{III} } = \ve^3  a'(\ve x)[ \ve^{m-2}\eta_c^m A_c^m (\ve t; y) + \tilde F^{\bf III}_c(\ve t; y) ] + \ve^2 a_\ve[ G^{\bf III}_c(\ve t; y) + \ve^{m-1}(\eta_c^m)' A_c^m],
$$
with $\tilde F^{\bf III}_c(\ve t; \cdot), G^{\bf III}_c(\ve t; \cdot )\in \mathcal Y$, uniformly for every $t\in [-T_\ve, T_\ve]$. Moreover, we have the estimate
\be\label{IIIH2a}
\|{\bf \tilde{III}}\|_{H^2(\R)} \leq K \ve^2e^{-\ga\ve|t|},
\ee
for every $t\in [-T_\ve, T_\ve]$.
\end{Cl}
\begin{proof}
The proof is identical to Lemma \ref{lem:SintIII}, being the unique new element in the proof the emergency of the term
$$
\ve^{m+1} a_\ve(\eta_c^m)' A_c^m,\quad \hbox{ with }\quad  \|\ve^{m+1} a_\ve(\eta_c^m)' A_c^m\|_{H^2(\R)} \leq K \ve^{m+\frac 12} e^{-\ga\ve|t|}.
$$
The other terms and their respective estimates remain unchanged. This finishes the proof.  
\end{proof}

Finally we consider the term ${\bf \tilde{II}}$. 

\begin{Cl}[Decomposition of ${\bf \tilde{II}}$ revisited]\label{9}
We have
\bee
{\bf \tilde{II}}   & = & -\ve\eta_c(y)(\mathcal L A_c)_y(\ve t; y) + O_{H^2(\R)}(\ve^{\frac 32} e^{-\ga\ve|t|}).
\eee
\end{Cl}

\begin{proof}
We follow on the lines of the proof of Lemma \ref{lem:dSKdVw}:  First we have
\bee
(\ve A_\#(\ve t; y))_t & = & -(c-\la) \ve^2 \eta_\ve' A_c(\ve t; y)-  (c-\la) \ve  \eta_\ve (A_c)_y(\ve t; y) \\
&   & + \ve^2 \eta_\ve (A_c)_t(\ve t; y) + \ve^2 c'(\ve t) \eta_\ve \Lambda A_c(\ve t; y).
\eee
We use now Lemma \ref{lem:omega} and (\ref{H1}) to estimate this last term. We get
\be\label{Gato1}
(\ve A_\#(\ve t; y))_t  = -  (c-\la) \ve  \eta_\ve(y) (A_c)_y(\ve t; y)  + O_{H^2(\R)}(\ve^{\frac 32}e^{-\ga\ve |t|} ). 
\ee
On the other hand,
\bee
& & \ve( (A_\#)_{xx} -\la A_\# + \frac{a_\ve}{a(\ve \rho)} m Q_c^{m-1}(y) A_\# )_x \\
&&  \qquad =  \ve\big\{  \eta_\ve [ (A_c)_{yy} -\la A_c +  \frac{a_\ve}{a(\ve \rho)} m Q_c^{m-1}(y) A_c]  + 2 \ve \eta_\ve' (A_c)_y +\ve^{2}\eta_\ve '' A_c \big\}_x \\
& & \qquad = \ve\eta_\ve \big[ (A_c)_{yy} -\la A_c +  \frac{a_\ve}{a(\ve \rho)} m Q_c^{m-1}(y) A_c\big]_x  \\
& & \qquad \quad +  \ve^2 \big[ 3\eta_\ve' (A_c)_{yy} -\la \eta_\ve' A_c + a_\ve m \eta_\ve' Q_c^{m-1}A_c + 3\ve\eta_\ve'' (A_c)_y + \ve^{2}\eta_\ve^{(3)} A_c \big]\\
& & \qquad  = \ve\eta_\ve \big[ (A_c)_{yy} -\la A_c +  m Q_c^{m-1}(y) A_c\big]_y +\ve^2 \eta_\ve m \frac{a'(\ve \rho)}{a(\ve \rho)} (yQ_c^{m-1}A_c)_y \\
& & \qquad \quad +  \ve^2 \big[ 3\eta_\ve' (A_c)_{yy}-\la\eta_\ve' A_c + a_\ve m \eta_\ve' Q_c^{m-1}A_c + 3\ve\eta_\ve'' (A_c)_y + \ve^{2}\eta_\ve^{(3)} A_c \big] \\
& & \qquad\qquad + O(\ve^3  \eta_\ve (y^2Q_c^{m-1}A_c)_y ).
\eee
We use now Lemma \ref{lem:omega} and the {\bf (IP)} property to estimate as follows
$$
m\ve^2\abs{\frac{a'(\ve \rho)}{a(\ve \rho)}} \norm{ \eta_\ve(yQ_c^{m-1}A_c)_y}_{H^2(\R)}\leq K\ve^2e^{-\ga\ve|t|},
$$
$$
\|O(\ve^3  \eta_\ve (y^2Q_c^{m-1}A_c)_y )\|_{H^2(\R)}\leq K\ve^3,\quad \ve^4\| \eta_c^{(3)}A_c\|_{H^2(\R)}\leq \ve^{\frac 72}e^{-\ga\ve|t|},
$$
$$
\|\ve^2  \la\eta_\ve' A_c \|_{H^2(\R)} \leq K \la \ve^{\frac 32} e^{-\ga\ve|t|},
$$
and
$$
\ve^2 \|3\eta_\ve' (A_c)_{yy} + a_\ve m \eta_\ve' Q_c^{m-1}A_c + 3\ve\eta_\ve'' (A_c)_y\|_{H^2(\R)}\leq K\ve^2 e^{-\ga\ve|t|}.
$$
Therefore
\bea\label{Gato2}
& & \ve[ (A_\#)_{xx} -\la A_\# + \frac{a_\ve}{a(\ve \rho)} m Q_c^{m-1}(y) A_\# ]_x = \nonumber \\
& & \qquad \qquad \ve\eta_\ve \big[ (A_c)_{yy} -\la A_c+  m Q_c^{m-1}(y) A_c\big]_y + O_{H^2(\R)}(\ve^2 e^{-\ga\ve|t|} +\ve^3).
\eea
The conclusion follows from (\ref{Gato1}) and (\ref{Gato2}).
\end{proof}

We return to the global estimate on $S[\tilde u]$. From (\ref{eq:SQ}), Claims \ref{8} and \ref{9} and Lemma \ref{lem:omega} we get
\bee
S[\tilde u] &= &\ve [F_1(\ve t, y) -\eta_c(y)(\mathcal LA_c)_y)(\ve t, y) ] + O_{H^2(\R)}(\ve^{\frac 32}e^{-\ga\ve|t|})\\
& = & \ve (1-\eta_c(y)) F_1(\ve t; y) + O_{H^2(\R)}(\ve^{\frac 32}e^{-\ga\ve|t|}).
\eee
The final conclusion of this appendix is a straightforward consequence of  the following fact: For every $t\in [-T_\ve, T_\ve]$
$$
\|\ve (1-\eta_c(y)) F_1(\ve t; y)  \|_{H^2(\R)}\leq K\ve e^{- \frac 1\ve  -\ga\ve |t|}\ll K\ve^{10}.
$$
for $\ve$ small enough. Indeed, note that $\supp (1-\eta_c(\cdot )) \subseteq (-\infty, -\frac 1\ve]$. From (\ref{eq:SQ}),
$$
\abs{F_1(\ve t; y)}\leq K e^{-\ga|y| -\ga \ve |t|}.
$$
From this estimate the desired estimate follows directly.
\end{proof}

\bigskip

\section{Proof of Lemma \ref{VL}}\label{SVL}

\bigskip

\subsection{Proof of Lemma \ref{VL}}
Our proof of the Virial inequality (\ref{dereta}) follows closely to the proof of Lemma 2 in \cite{MMnon}.
\begin{proof} Take $t\in [t_1, T^*]$, and denote $y:=x-\rho_2(t)$. Replacing the value of $ z_t$ given by (\ref{13a}),  we have
\bea
 \partial_t \int_\R  z^2 \psi_{A_0}(y) &  =&  2\int_\R  z  z_t  \psi_{A_0}(y) -\rho_2'(t) \int_\R  z^2 \psi_{A_0}'(y) \nonumber \\
&  = &  2\int_\R ( z   \psi_{A_0}(y) )_x (  z_{xx} -\la z+  mQ_{c_2}^{m-1}(y)  z ) \label{e0} \\
& & \quad  -(c_2(t)-\la)\int_\R  z^2 \psi_{A_0}'(y) - 2(c_2(t)-\la-\rho_2')(t)\int_\R  z Q'_{c_2} \psi_{A_0}(y)\label{e1} \\
& & \quad  + 2\int_\R ( z   \psi_{A_0}(y) )_x [(R +  z)^m -R^m -mR^{m-1} z ] \label{e2}\\
& & \quad - 2c_2'(t)\int_\R  z \Lambda Q_{c_2} \psi_{A_0}(y)    + (c_2-\la-\rho_2')(t) \int_\R  z^2 \psi_{A_0}'(y) \label{e3}\\
& & \quad  + \int_\R ( z   \psi_{A_0}(y) )_x (a_\ve-2) (R +  z)^m \label{e4}.
\eea
Now, following \cite{MMnon} and by using (\ref{rho2c2}) and (\ref{c2rho2})  it is easy to check that for $A_0$ large enough, and some constants $\delta_0, \ve_0$ small
$$
\abs{(\ref{e2}) + (\ref{e3})}  \leq \frac{\delta_0}{100}\int_\R ( z_x^2 +  z^2)(t)e^{-\frac 1{A_0}|y|}.
$$
On the other hand, the terms (\ref{e0}) and (\ref{e1}) goes similarly to the terms $B_1$ and $B_2$ in Appendix B of \cite{MMnon}. We get
$$
(\ref{e0}) + (\ref{e1})\leq  -\frac{\delta_0}{10}\int_\R ( z_x^2 +  z^2)(t)e^{-\frac 1{A_0}|y|}.
$$
Finally, the term (\ref{e4}) can be estimated as follows. First, from (\ref{11a}) and (\ref{12a}) we have for $t\geq t_1$
$$
c_2(t) = c_\infty + O(\ve^{1/2}), \quad  \rho_2(t) =  (c_\infty -\la) t + O(\ve^{1/2} (t-t_1)),
$$
and then
\be\label{param1}
 \frac {9}{10} c_\infty \leq c_2(t) \leq  \frac {11}{10}c_\infty; \quad \rho_2(t) \geq  \frac 9{10}(c_\infty-\la) t. 
\ee
On the other hand, we can write (\ref{e4}) in the following way
\bee
(\ref{e4}) &  = & \int_\R ( z\psi_{A_0})_x (a_\ve -2) [(R+ z)^m - z^m] + \int_\R ( z \psi_{A_0} )_x (a_\ve -2)  z^m\\
 & =& \int_\R (\psi_{A_0})_x (a_\ve -2) [(R+ z)^m - z^m]z + \int_\R \psi_{A_0}(a_\ve -2) [(R+ z)^m - z^m] z_x \\
 & & + \frac {m}{m+1}\int_\R (\psi_{A_0})_x (a_\ve -2) z^{m+1} -  \frac {\ve}{m+1}\int_\R \psi_{A_0} a'(\ve x) z^{m+1}.
\eee
Then, from (\ref{ahyp}), (\ref{psiA}) and by using that $t\geq t_1\geq \frac 12 T_\ve$, we get for some constant $\ga=\ga(A_0,c_\infty,\la)>0$ independent  of $\ve$ and $D_0$, (cf. (\ref{dec1}) for a similar computation)
\bee
\abs{\int_\R (\psi_{A_0})_x (a_\ve -2) [(R+ z)^m - z^m]z }&  \leq & K A_0 e^{-\ve\rho_2(t)/A_0}\|z(t)\|_{H^1(\R)} \\
& \leq & KA_0e^{-\ga\ve t} \|z(t)\|_{H^1(\R)}.
\eee
Similarly
$$
\abs{\int_\R \psi_{A_0}(a_\ve -2) [(R+ z)^m - z^m] z_x} \leq K A_0 e^{-\ga\ve t} \|z(t)\|_{H^1(\R)};
$$
and
$$
\abs{\int_\R (\psi_{A_0})_x (a_\ve -2) z^{m+1}} \leq K A_0 e^{-\ga\ve t} \|z(t)\|_{H^1(\R)}^{m+1}.
$$
Finally, from (\ref{asy}) and (\ref{param1}),
$$
\abs{\ve\int_\R \psi_{A_0} (y) a'(\ve x) z^{m+1}}\leq 
 K A_0 e^{- \ga \ve t}  \| z(t)\|_{H^1(\R)}^{m+1}.
$$
In conclusion, $(\ref{e4})  = O( A_0e^{- \ga\ve t}  \| z(t)\|_{H^1(\R)} ),$ 
for $\ve$ small enough.  

From (\ref{param1}) we obtain the second term in (\ref{dereta}). Collecting the above estimates we conclude the proof.
\end{proof}

\bigskip

\section{Proof of Lemma \ref{Mon3}}\label{MonF}

The proof of this result is very similar to Lemma 3 in \cite{MMnon}. 

\begin{proof}
First of all, recall that $ \phi = \phi (\tilde y(x_0)),$ with $ \tilde y(x_0)= x-( \rho_2(t_0) + \sigma (t-t_0) + x_0).$
Then we have
\bee
\partial_t \int_\R u^2 \phi  =  -\int_\R \big( 3u_x^2 +(\sigma +\la)u^2 -\frac {2m a_\ve }{m+1} u^{m+1} \big) \phi' + \int_\R u^2 \phi^{(3)}  -\frac {2\ve}{m+1}\int_\R a'(\ve x) u^{m+1}\phi,
\eee
and
\bea
\partial_t \int_\R \big( u_x^2 -\frac {2 a_\ve (x)}{m+1} u^{m+1}\big)\phi& = & \int_\R \big( -(u_{xx} + a_\ve u^m)^2-2u_{xx}^2 +2 m a_\ve  u_x^2 u^{m-1} \big) \phi' \nonumber \\
& & + \int_\R u_x^2 \phi^{(3)} - \sigma\int_\R \big( u_x^2 -\frac {2a_\ve}{m+1} u^{m+1}\big)\phi' \nonumber \\
& & -\frac \ve{m+1} \int_\R a_\ve' u^{m+1}\phi'' -\frac{\ve^2}{m+1}\int_\R a_\ve'' u^{m+1}\phi'. \label{dEnergy}
\eea
(see for example Appendix C in \cite{MMnon}). The conclusion follows from the arguments in \cite{MMnon}, after we estimate the unique new different term. In particular, we have
\be\label{dMas}
 -\int_\R \big( 3u_x^2 +(\sigma +\la)u^2 -\frac {2m a_\ve (x)}{m+1} u^{m+1} \big) \phi' + \int_\R u^2 \phi^{(3)}  \leq K e^{-(t_0-t) /2K_0} e^{-x_0/K_0}.
\ee
Indeed, using that $1/K_0^2 \leq \sigma/2$, we have (we discard the term with $\la$)
$$
-\int_\R \big( 3u_x^2 +\sigma u^2 -\frac {2m a_\ve (x)}{m+1} u^{m+1} \big) \phi' + \int_\R u^2 \psi^{(3)} \leq -\int_\R \big( 3u_x^2 +\frac \sigma 2u^2 -\frac {2m a_\ve (x)}{m+1} u^{m+1} \big) \phi'.
$$
Now we estimate the nonlinear term. Let $R_0>0$ to be chosen later. Consider the region $t\geq t_1$,  $|x-\rho_2(t)| \geq R_0$. In this region we have from the stability and the Morrey's embbeding
$$
|u(t,x)| \leq \|u(t) -R(t)\|_{L^\infty(\R)} + R(t,x) \leq K\ve^{1/2} + K e^{-\gamma R_0},
$$
with $\ga>0$ a fixed constant. Taking $0<\ve\leq\ve_0$ sufficiently small and $R_0$ large enough, we have
$\abs{ m a_\ve (x) u^{m-1}} \leq \sigma/4,$ in the considered region. Now we deal with the complementary region, $|x-\rho_2(t)| \leq R_0$. From (\ref{11a}) and the hypothesis $\sigma <\frac 12 (1-\la_0)$ we have
\be\label{imp1}
\abs{\tilde y(x_0)}  \geq |\rho_2(t_0)-\rho_2(t) - \sigma (t_0-t) + x_0| -|x-\rho_2(t)| \geq \frac 12\sigma (t_0-t) +x_0 -R_0.
\ee
Thus we have $\abs{\phi'(\tilde y)} \leq Ke^{- \ga(t_0-t)/K_0} e^{-x_0/K_0}. $
Collecting the above estimates we obtain (\ref{dMas}).
Now we claim that 
\be\label{EI2}
\abs{\frac{2\ve}{m+1}\int_\R a'(\ve x) u^{m+1}\phi } \leq K e^{-\ve \ga T_\ve} e^{- \ve\ga (t_0-t)/K_0 }e^{ - \ga\ve x_0/K_0}.
\ee
Indeed, denote $\tilde x(t) := \rho_2(t_0) + \sigma (t-t_0) +x_0$. Then from $\sigma<\frac 12(1-\la_0)$ and (\ref{11a}) we have
\bee
\tilde x(t) & = & \rho_2(t_0) -\rho_2(t) -\sigma (t_0-t) +  (x_0 + \rho_2(t)) \\
&   \geq & \frac 12\sigma (t_0 -t) + \rho_2(t_0) + x_0  \geq  \frac 12 \sigma(t_0 -t)  +  \frac 12 T_\ve +x_0, 
\eee
and thus for $\ve$ small,
\bee
\abs{\frac{2\ve}{m+1}\int_\R a'(\ve x) u^{m+1}\phi } & \leq &  K \ve \int_{-\infty}^{\tilde x} e^{-\ve \ga |x|} e^{(x-\tilde x)/K_0 } dx + K\ve \int_{\tilde x}^\infty e^{-\ve \ga x} \nonumber  \\
& \leq &  K \ve e^{-\tilde x/K_0}  +   K e^{-\ve \ga \tilde x} \nonumber \\
& \leq & K e^{-\ga \ve T_\ve} e^{ - \ga\ve (t_0 -t)/K_0}e^{ - \ga\ve x_0/K_0}. 
\eee
This last estimate proves (\ref{EI2}). Integrating between $t$ and $t_0$ we get (\ref{I}).

On the other hand, by following the same kind of calculations (see \cite{MMnon}), we have 
\bee
\partial_t \int_\R \big( u_x^2 + u^2 -\frac {2 a_\ve (x)}{m+1} u^{m+1}\big)\phi &  \leq & K e^{-\ga(t_0-t) /K_0} e^{-x_0/K_0} \\
& & \qquad + K e^{-\ga \ve T_\ve} e^{ - \ga\ve (t_0 -t)/K_0}e^{ - \ga\ve x_0/K_0}.
\eee
In consequence, after integration we get (\ref{J}).

Now we prove (\ref{tI}). The procedure is analogous to (\ref{I}); the main differences are in (\ref{imp1}) and (\ref{EI2}).
For the first case we have that $\tilde y(-x_0) = x-( \rho_2(t_0) +\sigma(t-t_0) -x_0)$ satisfies
$$
\abs{\tilde y}\geq |\rho_2(t) -\rho_2(t_0) -\sigma(t-t_0)+x_0 |-|x-\rho_2(t)| \geq \frac 12\sigma(t-t_0) +x_0 -R.
$$
From the hypothesis we have that $\hat x(t) := \rho_2(t_0) +\sigma(t-t_0) -x_0 > t_1\geq \frac 12T_\ve$. Therefore (\ref{EI2}) can be bounded as follows
\bee
\abs{\frac{2\ve}{m+1}\int_\R a'(\ve x) u^{m+1}\phi } & \leq &  K \ve \int_{-\infty}^{\hat x} e^{-\ve \ga |x|} e^{(x-\hat x)/K_0 } dx + K\ve \int_{\hat x}^\infty e^{-\ve \ga x} \nonumber  \\
& \leq &  K \ve e^{-\hat x/K_0}  +   K e^{-\ve \ga \hat x} \nonumber \\
& \leq & K e^{-\ga \ve \rho_2(t_0)} e^{ - \ga\ve (t-t_0)/K_0}e^{\ga\ve x_0/K_0}.
\eee
Collecting the above estimates and integrating between $t_0$ and $t$, we obtain the conclusion.
\end{proof}

\bigskip

\section{Some identities related to the soliton $Q$}\label{AidQ}

This section has been taken from Appendix C in \cite{MMcol1}.

\begin{lem}[Identities for the soliton $Q$]\label{IdQ}~
Suppose $m>1$ and denote by $Q_c := c^{\frac 1{m-1}} Q(\sqrt{c} x)$ the scaled soliton. Then
\begin{enumerate}
\item \emph{Energy}.
$$
E_1[Q]= \frac 12 (\la - \la_0)\int_\R Q^2 =(\la-\la_0)M[Q],  \qquad \hbox{with } \la_0 = \frac{5-m}{m+3}.
$$
\item \emph{Integrals}. Recall $\theta = \frac 1{m-1} -\frac 14$. Then
$$
\int_\R Q_c = c^{\theta-\frac 14} \int_\R Q, \quad \int_\R Q_c^{2} = c^{2\theta} \int_\R Q^2, \quad E_1[Q_c] = c^{2\theta +1}E_1[Q].
$$
and finally
$$
\int_\R Q_c^{m+1} = \frac{2(m+1)c^{2\theta +1}}{m+3} \int_\R Q^2, \qquad \int_\R \Lambda Q_c Q_c =\theta c^{2\theta -1} \int_\R Q^2.
$$
\end{enumerate}
\end{lem}

\bigskip

{\bf Acknowledgments}. The author wishes to thank Y. Martel and F. Merle for presenting him this problem and for  their continuous encouragement during the elaboration of this work. The author is also grateful of Gustavo Ponce for some useful discussions.

\end{document}